\documentclass{article}

\usepackage{textcomp}

\usepackage[utf8]{inputenc} 
\usepackage[T1]{fontenc}    
\usepackage{hyperref}       
\usepackage{url}            
\usepackage{booktabs}       
\usepackage{amsfonts}       
\usepackage{nicefrac}       
\usepackage{lipsum}

\usepackage{mathtools,tensor}

\usepackage{latexsym,amssymb,amsthm,amsmath,enumerate,epsfig,setspace,bbding,color,graphicx,caption}
\usepackage{stmaryrd}
\usepackage[all]{xy}
\usepackage{hyperref}
\usepackage{subfiles}
\usepackage[margin=10pt,font=small,labelfont=bf]{caption}
\usepackage{tikz-cd}
\usepackage{tkz-graph}
\usetikzlibrary{arrows}
\usepackage{verbatim}
\usepackage{amsmath}
\usepackage{amssymb}
\usepackage{amsfonts}
\usepackage{todonotes}
\usepackage{mathtools}
\usepackage{tikz}
\usepackage{cancel}
\usepackage{array}

\usepackage{enumitem}
\setlist[itemize]{itemsep=0cm}
\setlist[enumerate]{itemsep=0cm}

\setenumerate[0]{label=(\roman*)}

\usepackage[backend=bibtex, style=numeric-comp,maxbibnames=4]{biblatex} 
\usepackage{textcomp}

\newcommand{\G}{\Sigma}
\newcommand{\CC}{\mathcal{C}}
\newcommand{\F}{\mathcal{F}}
\newcommand{\Gaug}{\mathrm{Gauge}}
\newcommand{\g}{\mathfrak{g}}
\newcommand{\h}{\mathfrak{h}}

\newcommand{\m}{\mathfrak{m}}
\newcommand{\z}{\mathfrak{z}}
\newcommand{\RR}{\mathbb{R}}
\newcommand{\tto}{\rightrightarrows}
\newcommand{\lto}{\leftarrow}
\newcommand{\BB}{\mathbf{X}}

\newtheorem{definition}{Definition}[section]
\newtheorem{theorem}[definition]{Theorem}
\newtheorem{proposition}[definition]{Proposition}
\newtheorem{lemma}[definition]{Lemma}
\newtheorem{corollary}[definition]{Corollary}

\newtheorem*{introtheorem}{Theorem}

\theoremstyle{definition}

\newtheorem{remarkx}[definition]{Remark}            
\newenvironment{remark}
  {\pushQED{\qed}\remarkx}
  {\popQED\endremarkx}
  
\newtheorem{examplex}[definition]{Example}          
\newenvironment{example}
  {\pushQED{\qed}\examplex}
  {\popQED\endexamplex}

\newtheorem{manualtheoreminner}{Axiom}      
\newenvironment{axiom}[1]{%
  \renewcommand\themanualtheoreminner{#1}%
  \manualtheoreminner
}{\endmanualtheoreminner}

\newcommand{\dom}{\text{Dom}}

\DeclareMathOperator{\Bis}{Bis}
\DeclareMathOperator{\germ}{germ}

\DeclareMathOperator{\Germ}{Germ}

\DeclareMathOperator{\pr}{pr}
\DeclareMathOperator{\id}{id}
\DeclareMathOperator{\Fr}{Fr}

\DeclareMathOperator{\GL}{GL}

\begin{document}

\title{Pseudogroups of symmetries and Morita equivalences}

\author{Luca Accornero
\footnote{Departement Wiskunde, KU Leuven, Belgium, \url{luca.accornero@kuleuven.be} }
\and
Francesco Cattafi \footnote{Institut f\"ur Mathematik, Julius-Maximilians-Universit\"at W\"urzburg, Germany, \url{francesco.cattafi@mathematik.uni-wuerzburg.de} }}

\maketitle

\begin{abstract}

This work is a spin-off of an on-going programme \cite{INPROGRESS} which aims at revisiting the original studies of Lie and Cartan on pseudogroups and geometric structures from a modern perspective.

Within the framework of Lie groupoids equipped with a special multiplicative form -- called Pfaffian groupoids -- we focus on principal bibundles and Morita equivalences. In particular, we discuss in details the notion of Pfaffian Morita equivalence, its relation to the gauge construction in the Pfaffian setting, and its interactions with principal actions. 

We briefly present some examples and applications to transitive pseudogroups of symmetries, which we explored in great detail in \cite{AccorneroCattafi2}.
\end{abstract}

\begin{center}
\end{center}




\begin{center}
\end{center}

\tableofcontents

\section*{Introduction}

\addcontentsline{toc}{section}{Introduction}

The main subject of this paper is an enhancement of the classical notion of Morita equivalence between Lie groupoids. To be precise, we consider groupoids equipped with a vector-valued 1-form, called a {\it Pfaffian form}, which is compatible with the multiplication and originates from the theory of PDEs. Applications are discussed in \cite{AccorneroCattafi2}.

Both the point of view and the tools are part of an on-going programme conceived by Marius Crainic, which aims at revisiting the original studies of Lie \cite{LieEngel88TrGroupAll} and Cartan \cite{CARTANINFINITEGROUPS} on pseudogroups and geometric structures from a modern perspective. The first output of such long-term programme was the PhD thesis of Mar\'ia Amelia Salazar in 2013 \cite{MARIA}, followed by the PhD theses of Ori Yudilevich \cite{ORI} and of the two authors \cites{FRANCESCO, LUCA}. All the main outcomes will be collected in the forthcoming monograph \cite{INPROGRESS}, where also this paper will be embedded.

\subsection*{Motivation and background}
Our driving motivation is that groupoids equipped with a Pfaffian form -- called {\it Pfaffian groupoids} -- underlie the study of symmetries of geometric structures on manifolds. To be more precise, this paper fits the on-going programme of understanding and generalising Sophus Lie's~\cite{LieEngel88TrGroupAll} and \'Elie Cartan's~\cite{CARTANINFINITEGROUPS} seminal work on symmetries in differential geometries by means of Lie groupoid theory. Lie's and Cartan's work is about what nowadays are called {\it pseudogroups}. Loosely speaking, a pseudogroup on a manifold $\BB$ is a set of transformations 
\[
\Gamma \subset {\rm Diff}_{\rm loc}(\BB) :=\{\varphi \in C^\infty(U,V):\ U,V\subset \BB \text{ open subsets of } \BB,\ \varphi \text{ is a diffeomorphism}\}
\]
with the minimal properties expected from a set of symmetries, namely:
\begin{itemize}
\item it contains the identity;
\item it is closed under group-like operations (composition and inversion);
\item it is closed under sheaf-like operations (restriction and continuation/gluing).
\end{itemize}
Lie and Cartan focused on the case when the pseudogroup is {\it transitive} -- that is, the geometric structure is ``homogeneous'': any two points are connected by a symmetry of the structure. Apart for unraveling the structure behind their work, we are also interested in removing this ``homogeneity'' assumption. Our approach, which explicitly builds on previous work~\cite{MARIA,ORI,FRANCESCO,LUCA}, is designed to deal precisely with this task. Let us describe it briefly.

There is a deep relation between Lie's and Cartan's work on symmetries and {\it groupoids}. This fact was first made explicit by Charles Ehresmann~\cite{Ehresmann53}; later fundamental insights came from André Haefliger's work on foliations~\cite{HAEFLIGERGAMMASTRUCTURES}. The key observation is that any pseudogroup $\Gamma$ comes with a canonical tower of objects associated to it: the tower of jets of elements in $\Gamma$
\[
\dots \to J^{k+1}\Gamma \to J^k\Gamma \to \dots \to \BB\times \BB\to \BB,
\]
where
\[
J^k\Gamma := \{j^k_x\varphi \in J^k(\BB,\BB):\ \varphi\in \Gamma,\ x\in \dom(\varphi)\}.
\]
Since $\Gamma$ is a pseudogroup, all the $J^k\Gamma$'s are canonically groupoids. Intuitively, $J^k\Gamma$ represents the ``order-$k$ constraints'' that a local transformation of $\BB$ needs to meet in order to be an element of $\Gamma$. It is then natural to focus on the case where thess constraints are represented by regular PDEs. Geometrically, this corresponds to require that $J^k\Gamma$ cuts out an embedded submanifold in the jet manifold $J^k(\BB,\BB)$. In such a situation, the jet groupouids are {\it Lie} groupoids and $\Gamma$ is called a {\it Lie pseudogroup}.

However, the groupoid structure on $J^k\Gamma$ does not represent the full picture. In fact, as a submanifold of the jet manifold $J^k(\BB,\BB)$, $J^k\Gamma$ also comes equipped with the restriction of the canonical {\it Cartan form} on $J^k(\BB,\BB)$. Such form is a very well known object~\cite{SAUNDERS, Vinogradovetal99}, and underlies the geometric approach to the study of PDEs. What was realised only quite recently, see~\cite{MARIA}, is that the groupoid structure and the ``PDE'' structure, represented by the Cartan form, interact and are compatible with each other. This key fact led to the introduction of Pfaffian groupoids, as an abstract object whose main incarnation are jet groupoids of pseudogroup.

So far, the usage of Pfaffian groupoids as an abstract tool to understand symmetries has proved to be advantageous in many ways. Pfaffian groupoids provide a clean abstract framework where one can rigorously investigate the structure underlying symmetries, both from the local and the infinitesimal point of view~\cite{CrainicSalazarStruchiner15, INPROGRESS, MARIA}. Such framework has, in turn, been used for a variety of tasks:
\begin{itemize}
 \item producing a version of Cartan's structure equations that hold in great generality -- including for non-transitive pseudogroups of symmetries~\cite{ORI, CRAINICYUDILEVICH};
 \item provide new insights on the classical geometric approach to PDEs~\cite{CATTMARIUSMARIA};
 \item investigate the integrability problem for geometric structures in differential geometry~\cite{FRANCESCO, INPROGRESS};
 \item discuss cohomological invariants associated to certain classes of geometric structures~\cite{LUCA, AccCr2020}.
\end{itemize}

The focus of this paper is on the interaction between the theory of Pfaffian groupoids and the classical notion of Morita equivalence of Lie groupoids. In loose terms, Morita equivalence is known to capture the notion of ``transverse geometry'' of a Lie groupoid~\cite{MATIAS, MOERDIJK}, in the sense that Morita equivalent groupoids exhibit isomorphic behaviours transversely to their orbits. A bit more conceptually, Morita equivalent groupoids are different presentations of the same {\it differentiable stack}~\cite{STACKS}. In this paper we limit ourselves to the ``hands-on'' approach to Morita equivalence, by means of {\it principal bibundles}~\cite{MOERDIJK}.

The fact that Morita equivalence plays a role when studying symmetries of geometric structures is already made clear in the PhD thesis of the second author~\cite{FRANCESCO}. The focus is on geometric structures admitting local models; that is, geometric structures modelled on $\mathbb{R}^n$ equipped with some additional geometry. From~\cite{FRANCESCO} it follows that the pseudogroup of symmetries of the ``model geometry'' on $\mathbb{R}^n$ and the pseudogroup of symmetries of the structure on the manifold are Morita equivalent; to be more precise, the groupoids of {\it germs} of their elements are Morita equivalent. This induces Morita equivalences of the corresponding jet groupoids; and such Morita equivalences are indeed compatible with the Pfaffian forms. The Morita equivalences of jet groupoids were used as the inspiration for the notion of {\it Pfaffian Morita equivalence}~\cite[Definition 5.4.1]{FRANCESCO}. Such notion matches the one presented in this paper -- Definition~\ref{def:Pfaffain-Morita-equivalence} -- but was not thoroughly investigated. Three key questions, that this work answers, are the following:
\begin{itemize}
\item how much of the geometry of a Pfaffian groupoid is preserved under Pfaffian Morita equivalence?
\item can principal {\it Pfaffian} actions be ``transported'' across Pfaffian Morita equivalences?
\item what is the correct notion of {\it Pfaffian isotropy} -- which should provide a group-like object describing transitive Pfaffian groupoids, hence the transitive pseudogroups of Lie and Cartan?
\end{itemize}
The answer to these questions is not immediately clear even for jet groupoids of pseudogroups.

\subsection*{Structure of the paper and main contributions}

The paper assumes familiarity with the basics of Lie groupoid theory, as well as jet bundles and their Cartan distribution(s). For convenience of the readers not experienced with such topics, we collected the material that we need in an appendix. We feel safe to say that the ``initiated'' reader can skip most of it; we mention here some remarks especially relevant for the study of pseudogroups, concerning étale groupoids, Hausdorff issues, and transitivity:~\ref{rmk:transitivity_and_second_countability},~\ref{rmk:transitivity_of_pseudo}
and~\ref{rmk:transitivity_of_germ_groupoid}.

We start by providing a rather extensive introductory section, namely Section~\ref{sec:intro_section}, concerning (Lie) pseudogroups and their associated groupoids. In particular, we discuss the proof of Theorem~\ref{thm:Haef_corresp}, usually attributed to Haefliger, which is at the core of the groupoid approach to pseudogroup.

In Section~\ref{sec:Pfaffian-language} we recall and describe Pfaffian groupoids and principal Pfaffian bundles. Most of the theory discussed appears in the PhD theses~\cite{MARIA,FRANCESCO}, but has not yet appeared elsewhere in the literature; a thorough monograph on Pfaffian groupoids and their applications to geometric structures is currently in preparation~\cite{INPROGRESS}. In Section~\ref{sec:Pfaffian-language} we pay special attention to the case of {\it non-full Pfaffian groups}, which are relevant when describing isotropy groups of Pfaffian groupoids; most of this material first appeared in the PhD thesis of the first author~\cite[Chapter 1]{LUCA}. In particular, here we study principal Pfaffian bundles for an action of a Pfaffian group, and we prove that they are precisely the {\it Cartan bundles} introduced in~\cite{FRANCESCOPAPER}.

Section~\ref{sec:PME} contains most of the novel material. First, we introduce the notions of principal Pfaffian bibundle and Pfaffian Morita equivalence, Definitions~\ref{def:princ-Pfaff-bibundle} and~\ref{def:Pfaffain-Morita-equivalence}. Even though \ref{def:Pfaffain-Morita-equivalence} is clearly equivalent to~\cite[Definition 5.4.1]{FRANCESCO}, we provide a different characterisation in terms of the {\it gauge construction}. It is well known that a Lie groupoid $\Sigma\tto \BB$ is Morita equivalent to a groupoid $\Sigma'\tto \BB'$ if and only if $\Sigma'\tto \BB'$ is isomorphic to the gauge groupoid of $\Sigma\tto \BB$ with respect to some principal $\Sigma$-bundle. A Pfaffian version of the gauge construction is introduced in~\cite[Proposition 5.4.3]{FRANCESCO}. Here we use it to prove Theorem~\ref{prp:PME_is_gauge_construction}:
\begin{introtheorem}
Two Pfaffian groupoids $(\Sigma,\omega,E)$ and $(\Sigma',\omega',E')$ are Pfaffian Morita equivalent if and only if $(\Sigma',\omega',E')$ is Pfaffian isomorphic to the Pfaffian gauge of $(\Sigma,\omega,E)$ with respect to some Pfaffian principal $(\Sigma,\omega,E)$-bundle.
\end{introtheorem}
In Proposition~\ref{prp:Pfaffian-transverse-geometry} we discuss which geometric properties of a Pfaffian groupoid are preserved under Pfaffian Morita equivalence. We derive as a consequence that two Pfaffian groups are Pfaffian Morita equivalent if and only if they are Pfaffian isomorphic. We proceed to state and prove Theorem~\ref{prop:Pfaffian_principal_category}, one of the core theorems of the paper, especially for applications to jet groupoids and geometric structures:
\begin{introtheorem}
Let $(\Sigma,\omega,E)$ and $(\Sigma',\omega',E')$ be Pfaffian groupoids. The categories of principal Pfaffian $(\Sigma,\omega,E)$-bundles and of principal Pfaffian $(\Sigma',\omega',E')$-bundles are equivalent.
\end{introtheorem}
Then, we prove Theorem~\ref{prop_transitive_Pfaffian_Morita_equivalence}, which can be regarded as a characterisation of Pfaffian isotropy:
\begin{introtheorem}
Let $(\Sigma,\omega,E)$ be a (full) transitive Pfaffian groupoid. Its isotropy group possesses a canonical structure of (not necessarily full) Pfaffian group that is Pfaffian Morita equivalent to $(\Sigma,\omega,E)$.
\end{introtheorem}
In the last subsection, we discuss applications of Theorem~\ref{prop:Pfaffian_principal_category} and Theorem~\ref{prop_transitive_Pfaffian_Morita_equivalence} to transitive pseudogroups. This connects our approach to existing ones that focus on the transitive case, e.g.~\cite{GUILLEMINSTERNBERG, SingerSternberg65}, but also~\cite{SHARPE}. In fact, it shows that the Cartan bundles from Definition \ref{def:Cartan-bundle} can be used as a model for transitive differential geometry, as they encompass all the geometries arising from transitive Lie pseudogroups. For instance, we recover $G$-structures \cite{STERNBERG} and Cartan geometries \cite{SHARPE} as principal Pfaffian bundles. We explored these insights in great detail in~\cite{AccorneroCattafi2}.

\subsection*{Notations and conventions}
Throughout the paper, we use the notation $\BB$ for the spaces of objects of (Lie) groupoids and for the spaces/manifolds over which pseudogroups are defined. Base spaces of principal bundles are denoted by $M$, and principal bundles themselves are denoted by $P$.

We use the notation $\Sigma\tto \BB$ for arbitrary Lie groupoids, and $\mathcal{G}\tto \BB$ for {\it étale} groupoids (this distinction is relevant only in the first section). The Lie algebroid of a Lie groupoid is denoted by ${\rm Lie}(\Sigma)\to \BB$ or, if there is no risk of confusion, simply by $A\to \BB$. The unit of a (Lie) group $G$ is denoted by $e$, and its Maurer-Cartan form by $\omega_{\rm MC} \in \Omega^1 (G,\g)$.

If $V$ is a vector space, we write $ \Omega^k(\BB, V)$ for the space $\Omega^k(\BB, \BB\times V)$ of differential forms on $\BB$ valued in the {\it trivial} vector bundle $\BB\times V \to \BB$.

The canonical Cartan form on the jet prolongation $J^kY$ of a submersion $Y\to X$ is denoted by $\omega^k$; the same notation is used for the restriction of $\omega^k$ to a submanifold of $J^kY$. Similarly, when there is no danger of confusion, we use $\pr$ to denote the canonical projection from $J^kY$ to $J^hY$, whenever $h<k$. 

All group(oid) actions are considered from the left and all manifolds and maps are smooth, unless explicitly stated otherwise.

\subsection*{Acknowledgements}
Both the authors would like to thank Marius Crainic for starting the long-term project of which this paper is part, for several insightful discussions, and for the support during their PhD years.

The first author was partially supported by the NWO through the Utrecht Geometry Center graduate programme (The Netherlands), and by the FWO-FNRS under EOS project G0I2222N (Belgium).
 
The second author was partially supported by the FWO-FNRS under EOS project G0H4518N (Belgium), by the ESI under the Junior Research Fellowship ``Cartan geometries via Pfaffian groupoids'' (Austria), by the FWF under Mozart Grant I 5015-N (Austria) and by the DFG under Walter Benjamin project 460397678 (Germany), and is a member of the GNSAGA - INdAM (Italy).

\section{From Lie pseudogroups to Pfaffian groupoids}\label{sec:intro_section}

\subsection{(Pseudo)groups of symmetries}\label{sec:pseudogroups}

The definition below is the modern formulation of Lie's definition of "transformation group" of symmetries~\cite{LieEngel88TrGroupAll}.
\begin{definition}\label{def:pseudogroups}
A {\bf pseudogroup} $\Gamma$ on a manifold $\BB$ is a subset of the set of smooth embeddings of open sets of $\BB$ into $\BB$ which is 
\begin{itemize}
\item closed under the group-like operations, i.e.\ composition (when defined) and inversion, and containing the identity $\id_\BB:\BB\to \BB$;
\item local, i.e.\ such that restrictions of elements in $\Gamma$ belong to $\Gamma$;
\item closed under gluings, i.e.\ if $\{V_k\}_{k\in K}$ covers an open $U$ and $\varphi_k\in \Gamma$ are the restrictions of an embedding $\varphi:U\hookrightarrow \BB$ to the opens $V_k$'s, then $\varphi\in \Gamma$.
\end{itemize}
\end{definition}



The three defining properties of a pseudogroup listed above are the minimal ones to expect from sets of symmetries of geometric structures on manifolds -- Lie's idea of ``transformation group''.

\begin{example}
Let $\BB$ be a manifold. The set ${\rm Diff}_{\rm loc}(\BB)$ of diffeomorphisms defined on open subsets of $\BB$ is a pseudogroup. We will sometimes denote such a pseudogroup by $\Gamma^\BB$.
\end{example}

\begin{example}[Symmetries of geometric structures]\label{ex:pseudo_as_sym}
Let $M$ be a manifold equipped with some ``smooth geometric structure'' $\mathcal{S}$. The set
\[
\Gamma^M_\mathcal{S} = \{\varphi\in \Gamma^M:\ \varphi\ \text{is a ``local symmetry'' of } \mathcal{S}\}
\]
is a pseudogroup. To be more precise on the notion of symmetry, we provide a list of concrete examples below.

\begin{center}
\begin{tabular}{ |p{6.5cm}|p{7.5cm}| } 
 \hline
 \textbf{Geometric structure} & \textbf{Pseudogroup of symmetries} \\  [1.5ex]
 \hline\hline
 $(M,\omega)$ symplectic manifold & $\Gamma^M_{\omega}=\{\varphi\in \Gamma^M:\ \varphi^*\omega = \omega\}$ \\ [1.5ex]
 \hline
 $(M,J)$ complex manifold & $\Gamma^M_{J}=\{\varphi\in \Gamma^M:\ d\varphi\circ J = J\circ d\varphi\}$ \\ [1.5ex]
 \hline
 $(M,\mathcal{F})$ foliated manifold & $\Gamma^M_{\mathcal{F}}=\{\varphi\in \Gamma^M:\ \varphi \text{ sends leaves to leaves}\}$ \\ [1.5ex]
 \hline
 $(M,\alpha)$ co-orientable contact manifold & $\Gamma^M_{\alpha}=\{\varphi\in \Gamma^M:\ \varphi^*\alpha = \alpha\}$ \\ [1.5ex]
 \hline
 $(M,g)$ Riemannian manifold & $\Gamma^M_{ g}=\{\varphi\in \Gamma^M:\ \varphi_*g = g\}$ \\ [1.5ex]
 \hline
\end{tabular}
\end{center}
\end{example}

All the examples above except the last one share the following property: given two points in the manifold, there is an element of the pseudogroup sending one to the other. This is a simple consequence of the existence of local models for the geometric structures involved (for instance, the Darboux model for symplectic and contact structures). The last example is different: a general Riemannian metric does not have local models, and given two points of $M$ there might be no isometry sending one to the other. However, if the metric is flat, then this is surely true, again due to the fact that in such a situation the structure has a local  model: any point has a neighbourhood isometric to the Euclidean space.

The way to formalise this observation is through Definition~\ref{def:transitive-orbits} and Definition~\ref{def:Gamma-atlas} below.

\begin{definition}\label{def:transitive-orbits}
Let $\Gamma$ be a pseudogroup on a manifold $\BB$. The orbit $\mathcal{O}_x$ of $x\in \BB$ is the set of all points $y\in \BB$ for which there exists some $\varphi\in \Gamma$ such that $\varphi(x)=y$.

A pseudogroup on $\BB$ is called {\bf transitive} if, for all $x,y\in \BB$, there is some $\varphi\in \Gamma$ such that $\varphi(x)=y$, i.e.\ the only orbit of $\Gamma$ is $\BB$ itself. 
\end{definition}

All the pseudogroups of symmetries in Example \ref{ex:pseudo_as_sym} are transitive, with the exception of the pseudogroup of isometries of an arbitrary Riemannian metric.

\begin{definition}\label{def:Gamma-atlas}
Let $\Gamma$ be pseudogroup on $\BB$ and $M$ be a smooth manifold. A {\bf $\Gamma$-atlas} on $M$ is an atlas\footnote{An atlas on a manifold $M$ valued in another manifold $\BB$ is defined analogously to an ordinary atlas -- the only difference being that charts are $\BB$-valued.} whose transition functions are elements of $\Gamma$. A {\bf $\Gamma$-structure} $\mathcal{S}_\Gamma$ on $M$ is an equivalence class of $\Gamma$-atlases/a maximal $\Gamma$-atlas. 
\end{definition}

With the exception of Riemannian metrics,\footnote{Notice that flat Riemannian metrics are $\Gamma$-structures as well: they can be represented by atlases whose transition functions belong to the pseudogroup of isometries of the Euclidean metric.
} all the geometric structures considered in Example \ref{ex:pseudo_as_sym}
are $\Gamma$-structures, where $\Gamma$ is a suitable pseudogroup on $\mathbb{R}^n$: the pseudogroup of symplectomorphisms of the standard symplectic form, the pseudogroup of bi-holomorphisms of the standard complex structure, etc. Notice that the structure $\mathcal{S}_\Gamma$ on $M$ is completely determined by the $\Gamma$-atlas: the charts can be used to locally ``pull-back'' the structure from the Euclidean space to $M$, and the transition functions belonging to $\Gamma$ ensure that these local data glue together to define a structure on $M$.

\begin{remark}\label{rmk:pseudo_of_sym} 
The set of local symmetries of a $\Gamma$-structure $\mathcal{S}_\Gamma$ on $M$ is a pseudogroup, which we denote by $\Gamma^M_{\mathcal{S}_\Gamma}$.

However, let us stress that $\Gamma$-structures do not exhaust the class of geometric structures on manifolds whose local symmetries are pseudogroups. 
For example, local isometries of a Riemannian manifolds form a pseudogroup. The same holds e.g.\ for symmetries of almost symplectic structures, almost complex structures and distributions (i.e.\ versions of the structures from Example~\ref{ex:pseudo_as_sym} that do not possess a local model).
\end{remark}

\subsection{Haefliger's correspondence}\label{sec:Haef_corr}
In this section we make extensive use of Lie groupoid theory. The definitions and properties we need, which are rather well known and estabilished, are briefly discussed in Appendix~\ref{app:Lie_gpds}. In particular, note that we impose Axiom \ref{axiom_moment_map} on principal groupoid bundles.

In the previous subsection we observed that the transitivity property of a pseudogroup $\Gamma$ on a manifold $M$ can be derived from the existence of local models. This is a manifestation of something deeper that lies at the core of our approach to pseudogroups. We describe it below.

\begin{theorem}\label{thm:Haef_corresp}
There is a one to one correspondence between pseudogroups and effective étale groupoids (see definitions~\ref{def:etale_gpd} and~\ref{def:effective_gpd}). Under such correspondence, if $\mathcal{G}_\Gamma$ denotes the groupoid associated to a pseudogroup $\Gamma$, the following facts are true.
\begin{enumerate}
\item $\Gamma$-structures $\mathcal{S}_\Gamma$ on a manifold $M$ correspond to principal $\mathcal{G}_\Gamma$-bundles over $M$ whose moment map is étale (see Definition~\ref{def:princ_gpd_bundle}).
\item Let $\Gamma^M_{\mathcal{S}_\Gamma}$ (cf.\  Remark~\ref{rmk:pseudo_of_sym}) be the pseudogroup of symmetries of a $\Gamma$-structure on $M$. Then $\mathcal{G}_\Gamma$ and $\mathcal{G}_{\Gamma^M_{\mathcal{S}_\Gamma}}$ are Morita equivalent (see Definition~\ref{def:classical-ME}).
\end{enumerate}
\end{theorem}

The core ideas behind Theorem~\ref{thm:Haef_corresp} and its proof -- in particular, the correspondence between pseudogroups and groupoids and the fact that $\Gamma$-structures induce groupoid valued cocycles -- can be traced back to André Haefliger~\cite[section I.6]{HAEFLIGERGAMMASTRUCTURES}. Later authors -- in particular,  Moerdijk~\cite{IEKEHAEFLIGER} and Mr\v{c}un~\cite{JANEZ} -- significantly elaborated on such core ideas. The result is the theorem in the form seen above. See also~\cite{FRANCESCO,LUCA} for expositions with aims similar to those of this paper.

For ours and the reader's convenience, we recall the main steps of the proof.

\begin{proof}
The claimed correspondence is based on the germ construction: if $\Gamma$ is a pseudogroup on $\BB$, one sets
\[
\mathcal{G}_\Gamma := {\rm Germ}(\Gamma) = \{{\rm germ}_x(\varphi):\ \varphi\in \Gamma,\ x\in {\rm dom}(\varphi)\}.
\]
The set $\mathcal{G}_\Gamma$ is a groupoid over $\BB$ due to the fact that $\Gamma$ is a pseudogroup over $\BB$ (in particular, it is closed under group-like operations): the source and target maps are defined by 
\[
s: {\rm germ}_x(\varphi)\mapsto  x,\quad t: {\rm germ}_x(\varphi)\mapsto \varphi(x),
\]
the unit and inversion maps are defined by 
\[
u: x\to\mapsto  {\rm germ}_x({\rm id}_\BB),\quad i: {\rm germ}_x(\varphi)\mapsto {\rm germ}_{\varphi(x)}(\varphi^{-1}),
\]
and the multiplication is defined by
\[
{\rm germ}_{\varphi(x)}(\varphi_2)\cdot {\rm germ}_x(\varphi_1) = {\rm germ}_x(\varphi_2\circ \varphi_1),
\]
for all $\varphi$, $\varphi_1$ and $\varphi_2$ in $\Gamma$. Moreover, $\mathcal{G}_\Gamma$ can be equipped with the étale topology, which makes source and target into local homeomorphisms. 
Notice that local bisections of $\mathcal{G}_\Gamma$ correspond precisely to elements of $\Gamma$. Consequently, $\mathcal{G}_\Gamma$ is {\bf effective}, i.e.\ for $\sigma_1, \sigma_2 \in \Bis_{\rm loc}(\mathcal{G}_\Gamma)$,
\[
t\circ \sigma_1 = t\circ \sigma_2 \iff \sigma_1 = \sigma_2.
\]
The resulting map $\Gamma\to \mathcal{G}_\Gamma$ from pseudogroups to effective étale groupoids is a bijection: its inverse is the map assigning to an effective étale groupoid $\mathcal{G}$ the set
\[
\Gamma_\mathcal{G} := \{t \circ \sigma:\ \sigma\in {\rm Bis}_{\rm loc}(\Gamma)\} \subset {\rm Diff}_{\rm loc}(\BB),
\]
which can be easily checked to be a pseudogroup (see also Example \ref{effective_etale_pseudogroups}).

The rest of the proof is mostly based on work in~\cite{JANEZ}.

\begin{enumerate}
\item The key observation here is the following. Let $\{U_i\}_{i\in I}$ be an open cover of $M$ supporting a $\Gamma$-atlas. If $f_i:U_i\subset M\to \BB$ and $f_j:U_j\subset M\to \BB$ are charts of the $\Gamma$-atlas, the transition function
\[
 \varphi_{ij}\in \Gamma, \quad \varphi_{ij}: f_i(U_i\cap U_j)\to f_j(U_i\cap U_j),
\]
is equivalently encoded into the ${\rm Germ}(\Gamma)$-valued map
\[
\sigma_{ij}: x\mapsto {\rm germ}_{f_i(x)}({\varphi_{ij}}), \quad x\in U_i\cap U_j.
\]
In fact, charts themselves can be seen as ${\rm Germ}(\Gamma)$-valued maps whose image lies in the unit manifold; if $f_i:U_i\to \BB$ is a chart of a $\Gamma$-atlas then
\[
\sigma_{ii}: x\mapsto {\rm germ}_{f_i(x)}({{\rm id}_\BB}), \quad x\in U_i
\]
completely determines $f_i$. 

The properties defining a $\Gamma$-atlas imply (and are in fact, equivalent to) the following: the datum of the cover $\{U_i\}_{i\in I}$ with the maps $\{\sigma_{ij}\}_{i,j\in I}$ is a $\Germ(\Gamma)$-valued cocycle, i.e.:
\begin{itemize}
\item $\sigma_{ii}(x)$ is a unit for all $x\in U_i$;
\item $\sigma_{ij}(x)$ is an arrow from $\sigma_{ii}(x)$ to $\sigma_{jj}(x)$ for all $x\in U_i\cap U_j$;
\item the cocycle condition
\[
\sigma_{ik}(x) = \sigma_{jk}(x)\cdot \sigma_{ij}(x),\quad x\in U_i\cap U_j\cap U_k
\]
holds.
\end{itemize}
From this one sees that $\Gamma$-structures are in one to one correspondence with equivalence classes of ${\rm Germ}(\Gamma)$-cocycles such that the maps $\{\sigma_{ii}\}_{i\in I}$ are diffeomorphisms onto their image; the equivalence relation is the one identifying two cocycles when there is a larger cocycles containing both of them. Details are provided in~\cite{JANEZ} (see in particular Proposition {\rm I}.3.1), where the correspondence between such cocycles and principal bundles with étale moment map is also discussed. In the end, if $\mathcal{A}=(\{U_i\}_{i\in I},\{f_i\}_{i\in I})$ is a maximal $\Gamma$-atlas, the principal bundle associated to the $\Gamma$-structure is
\[
\xymatrix{
\Germ(\Gamma) \ar@<0.25pc>[dr] \ar@<-0.25pc>[dr]  & \ar@(dl, ul) &  \Germ(\mathcal{A}) \ar[dl]^{\mu}\ar[dr]^{\pi} &    \\
&\BB  & &  M,}
\]
where 
\[
\Germ(\mathcal{A}) := \{{\rm germ}_x(f_i):\ x\in U_i\subset M\},
\]
\[
\pi({\rm germ}_x(f_i)) := x, \quad \quad \mu({\rm germ}_x(f_i)) := f_i(x),
\]
and the action is induced by left composition.
\item This point follows quickly using the previous one. If $M$ carries a $\Gamma$-structure, we have a maximal $\Gamma$-atlas $\mathcal{A}$ and the principal bundle
\[
\xymatrix{
\Germ(\Gamma) \ar@<0.25pc>[dr] \ar@<-0.25pc>[dr]  & \ar@(dl, ul) &  \Germ(\mathcal{A}) \ar[dl]^{\mu}\ar[dr]^{\pi} &    \\
&\BB  & &  M.}
\]
By applying the gauge construction (cf.\ Theorem~\ref{thm:gauge_eq_morita} and the discussion above it), we get a groupoid $\Gaug(\Germ(\mathcal{A}))$ over $M$ and a principal bibundle
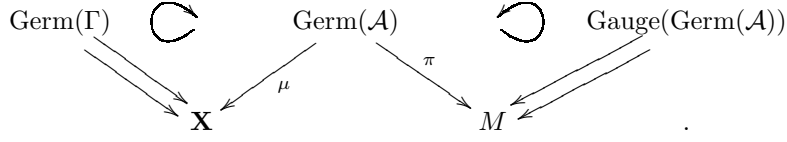
\[
\xymatrix{
\Germ(\Gamma) \ar@<0.25pc>[dr] \ar@<-0.25pc>[dr]  & \ar@(dl, ul) &  \Germ(\mathcal{A}) \ar[dl]^{\mu}\ar[dr]^{\pi} &\ar@(dr, ur) & \Gaug(\Germ(\mathcal{A}))\ar@<0.25pc>[dl] \ar@<-0.25pc>[dl]\\
&\BB  & &  M & .}
\]
It can be checked directly that the pseudogroup associated to $\Gaug(\Germ(\mathcal{A}))$ is precisely the pseudogroup of symmetries of the $\Gamma$-structure on $M$ defined by $\mathcal{A}$. \qedhere
\end{enumerate}
\end{proof}

\begin{remark}
A complete proof of point $(i)$ in Theorem~\ref{thm:Haef_corresp} can be read in~\cite[Proposition ${\rm I}.3.1$]{JANEZ}. However, we need to stress here that~\cite[Proposition ${\rm I}.3.1$]{JANEZ} is a more general statement: it deals with the case when the moment map is not necessarily étale, $\dim(M)$ is not necessarily equal to $\dim(\BB)$, and $\Germ(\Gamma)$ is replaced with a -- not necessarily effective -- étale groupoid $\mathcal{G}\tto \BB$. At such level of generality, the author proves that there is a bijective correspondence between isomorphism classes of principal $\mathcal{G}$-bundles over $M$ and {\it Haefliger $\mathcal{G}$-structures} on $M$. Defining Haefliger $\mathcal{G}$-structures goes beyond the scope of our discussion; $\Gamma$-structures are a particular case arising when $\mathcal{G}=\Germ(\Gamma)$ and $\dim(M)=\dim(\BB)$.
\end{remark}

\subsection{The ``continuity'' condition}
Observe that all the pseudogroups from Example~\ref{ex:pseudo_as_sym} share a common feature: their elements arise as solutions of a certain system of PDEs. Furthermore, such PDEs are rather well behaved, since they arise from well-behaved geometric objects on manifolds (symplectic forms, complex structures, contact forms, Riemannian metrics). In general, pseudogroups can exhibit a wilder (and richer!) behaviour. 
\begin{example}\label{ex:gamma_coord}
On $\mathbb{R}^n$ with coordinates $(x^1,\dots,x^n)$ define
\[
\Gamma_{x^n} := \{\varphi\in {\rm Diff}_{\rm loc}(\mathbb{R}^n):\ x^n\circ \varphi = x^n|_{{\rm dom}(\varphi)}\}.
\]
The set $\Gamma_{x^n}$ is a pseudogroup, geometrically characterised as the set of transformations of $\mathbb{R}^n$ preserving the hyperplanes cut by $x^n = c$, $c\in \mathbb{R}$. While $\Gamma_{x^n}$ is not transitive, its orbits define a regular foliation in $\mathbb{R}^n$ (cf.\ Definition~\ref{def:transitive-orbits}). Notice that similar examples can be constructed on any manifold equipped with a regular foliation. Notice also that the elements $\varphi$ of $\Gamma_{x^n}$ can be equivalently characterised by saying that their $n$-th component $\varphi_n$ is constant; that is, they are solutions of the system of PDEs
\[
\frac{\partial \varphi_n}{\partial x^i} = 0,\quad i = 1,\dots, n. \qedhere
\]
\end{example}

\begin{example}\label{ex:gamma_conf}
On $\mathbb{R}^4$, with coordinates $(x,y,z,t)$ consider the 1-form $\Omega= dz-tydx$ and define 
\[
\Gamma_{\Omega} := \{\varphi\in {\rm Diff}_{\rm loc}(\mathbb{R}^4):\ \varphi^*\Omega = \Omega\}.
\]
The set $\Gamma_{\Omega}$ is a pseudogroup, geometrically characterised as the set of transformations of $\mathbb{R}^4$ preserving the form $\Omega$. Again, $\Gamma_\Omega$ is not transitive, because $\Omega$ is closed along the submanifold cut by the equation $t=0$, and not closed elsewhere. There are three orbits of $\Gamma_\Omega$ (cf.\ Definition~\ref{def:transitive-orbits}): the submanifold cut by $t=0$ and the two connected components of its complement. Hence, the orbits of $\Gamma_\Omega$ do not define a regular foliation on $\mathbb{R}^4$. 
The elements of $\Gamma_\Omega$ are solutions of the system of PDEs
\[
\begin{cases}
\dfrac{\partial\varphi_3}{\partial x} - 	\varphi_4\varphi_2\dfrac{\partial \varphi_1}{\partial x} = -ty\\
\dfrac{\partial\varphi_3}{\partial y} - \varphi_4\varphi_2\dfrac{\partial \varphi_1}{\partial y} = 0\\
\dfrac{\partial\varphi_3}{\partial z} - \varphi_4\varphi_2\dfrac{\partial \varphi_1}{\partial z} = 1
\\
\dfrac{\partial\varphi_3}{\partial t} - \varphi_4\varphi_2\dfrac{\partial \varphi_1}{\partial t} = 0
    \end{cases} \qedhere
\]
\end{example}

\begin{example}\label{ex:gamma_hyp}
On $\mathbb{R}^2$, consider the function $f(x,y) = x^2-y^2$ and define 
\[
\Gamma_f := \{\varphi\in {\rm Diff}_{\rm loc}(\mathbb{R}^2):\ f\circ \varphi = f|_{{\rm dom}(\varphi)}\}.
\]
The set $\Gamma_f$ is a pseudogroup, geometrically characterised as the set of transformations of $\mathbb{R}^2$ preserving the level sets of $f$. Not only $\Gamma_f$ is not transitive, but its orbits are given by the level sets of $f$: they do not define a regular foliation on $\mathbb{R}^2$ 
(cf.\ Definition~\ref{def:transitive-orbits}). The elements of $\Gamma_f$ are also characterised by the fact that $\varphi_1^2-\varphi_2^2$ is constant, i.e.\ they are solutions of the system of PDEs
\[
\begin{cases}
2\varphi_1\dfrac{\partial\varphi_1}{\partial x}-2\varphi_2\dfrac{\partial\varphi_2}{\partial x} = 0\\
2\varphi_1\dfrac{\partial\varphi_1}{\partial y}-2\varphi_2\dfrac{\partial\varphi_2}{\partial y} = 0
\end{cases} \qedhere
\]
\end{example}

In his analysis of ``transformation groups'' Lie restricted to a class of well behaved ones, which he called ``continuous transformation groups''. There are essentially three conditions characterising ``continuous transformation groups'' among all ``transformation groups'':
\begin{itemize}
\item the elements of the ``transformation group'' arise as solutions of a system of PDEs;
\item the system of PDEs is sufficently well behaved;
\item the ``transformation group'' is transitive, cf.\ Definition~\ref{def:transitive-orbits}.
\end{itemize}
The last assumption seems, and indeed is, of a different nature than the first two. One of the main merit of our approach to pseudogroups compared to older ones -- e.g.~\cite{SingerSternberg65} -- is that we are able to get rid of the transitivity assumptions. We still need (a version of) the first two assumptions. Below, we provide details. 

A starting point is the following observation:

\begin{lemma}
There is a one to one correspondence between pseudogroups $\Gamma$ on $\BB$ and sets $\hat{\Gamma}$ of local bisections of the pair groupoid $\BB\times \BB$ satisfying the following properties:
\begin{itemize}
\item $\hat{\Gamma}$ is closed under group-like operations, i.e.\ multiplication of bisections (when defined) and inversion, and the unit bisection is an element of $\hat{\Gamma}$;
\item $\hat{\Gamma}$ is closed under restriction of bisections;
\item  $\hat{\Gamma}$ is closed under gluing of bisections, i.e.\ if $\{V_k\}_{k\in K}$ covers an open subset $U\subset \BB$, $\sigma:U\hookrightarrow \BB\times \BB$ is a bisection over $U$ and $\sigma|_{V_k}\in \Gamma$ for all $k\in K$, then $\sigma\in \Gamma$.
\end{itemize}
\end{lemma}
The proof is a rather simple check, based on the correspondence between diffeomorphisms of $\BB$ and bisections of $\BB\times \BB \tto \BB$.

\begin{remark}
The properties satisfied by $\hat{\Gamma}$ in the proposition above are completely analogous to the properties defining a pseudogroup (Definition~\ref{def:pseudogroups}). Indeed, one can define a {\it generalised pseudogroup} as a subset $\Gamma$ of the set of local bisections of some groupoid $\G\tto \BB$ -- called the {\it support} of $\Gamma$ -- satisfying the ``pseudogroup properties'' listed above. This point of view allows to treat as ``continuous transformation groups'' geometric examples whose associated {\it classical} pseudogroup is not ``continuous'' in the sense of Lie. See~\cite{ORI,FRANCESCO,LUCA} for details. 

We stress that the techniques of this paper, developed in Sections~\ref{sec:Pfaffian-language} and~\ref{sec:PME}, apply verbatim to generalised pseudogroups -- another advantage of our approach. 
\end{remark}

Under the identification of $\Gamma$ with $\hat{\Gamma}$, any pseudogroup can be realised as bisections of both $\Germ(\Gamma) \tto \BB$ and $\BB \times \BB \tto \BB$. What makes the pair groupoid $\BB\times \BB\tto \BB$ interesting for our purposes is that, differently from $\Germ(\Gamma)$, $\BB\times\BB$ carries a non-discrete topology ``along the fibres''. This topology can be used to impose conditions on $\Gamma$.


From now on, together with Lie groupoids, we will make extensive use of jet prolongations and their canonical Cartan distributions, which provide a well estabilished framework for the geometric study of systems of PDEs. The definitions and facts that we use are recalled in Appendix~\ref{app:jets}.

We observe now that, since
\[
s:\BB\times \BB \to \BB
\] 
is a surjective submersion, we have the jet prolongations (cf.\ Definition~\ref{def:jet_prol})
\[
\dots \to J^{k+1}(\BB\times \BB)\to J^k(\BB\times \BB) \to \dots\to \BB\times \BB \to \BB,
\]
where $k\in \mathbb{N}$. For each $k$, $J^k(\BB\times \BB)\to \BB$ is a Lie groupoid; the multiplication map is the one sending composable pairs of jets to the jet of their composition. We also have an analogous tower 
\[
\dots \to J^{k+1}\Gamma\to J^k\Gamma \to \dots\to J^0\Gamma \to \BB,
\]
where, for each $k$, the set $J^k\Gamma$ is defined as
\[
J^k\Gamma := \{j^k_x\varphi\in J^k(\BB\times \BB):\ \varphi\in \Gamma\}.
\]
Notice that $J^k\Gamma$ is a subgroupoid of $J^k(\BB\times \BB)$. In particular, it can be viewed as the order $k$ differential constraints satisfied by elements of $\Gamma$. Notice that $J^k\Gamma$ is not necessarily a smooth submanifold, hence a system of PDEs according to Definition~\ref{def:diff_eq}; this will be required in Definition \ref{def:Lie_pseudo}.

\begin{definition}\label{def:order}
Let $\Gamma$ be a pseudogroup on $\BB$. We say that $\Gamma$ is {\bf of order $k$} if $\varphi \in \Gamma$ if and only if $j^{k}\varphi\in J^{k}\Gamma$.
\end{definition} 
Elements of $\Gamma$ are local solutions of the system of PDEs $J^k\Gamma\subset J^k(\BB\times \BB)$, for all $k$'s. In more geometric terms:
\begin{lemma}\label{lemma_holonomic_bisection_jet_pseudogroup}
Let $\varphi\in \Gamma$ be an element of $\Gamma$. Its jet prolongation (cf.\ Definition~\ref{def:hol_sec})
\[
j^k\varphi:{\rm dom}(\varphi)\to J^k\Gamma,\quad x\mapsto j^k_x\varphi
\]
is a local holonomic bisection of $J^k\Gamma\subset J^k(\BB\times \BB)$.

If $\Gamma$ has order $k$ -- see Definition~\ref{def:order} -- then the map
\[
\Gamma\to {\rm Bis}_{\rm loc}(J^k(\BB\times \BB))
\]
is bijective onto holonomic bisections with image in $J^k\Gamma$.
\end{lemma} 
Using the construction of the germ groupoid from Theorem~\ref{thm:Haef_corresp}, there is a one to one correspondence between elements of a pseudogroup $\Gamma$ over $\BB$ and bisections (with no additional property) of its germ groupoid $\Germ(\Gamma)\tto \BB$.

Lie's ``continuity'' condition for pseudogroups is expressed in terms of the jet groupoids $J^k\Gamma\tto \BB$.

\begin{definition}\label{def:Lie_pseudo}
A pseudogroup $\Gamma$ on a manifold $\BB$ is called {\bf Lie pseudogroup} when there exists $k_\Gamma\in \mathbb{N}$ such that: 
\begin{itemize}
\item $J^k\Gamma$ is a Lie subgroupoid of $J^k(\BB, \BB)$, for all $k\geq k_\Gamma$;
\item all the groupoid maps $J^{k+1}\Gamma\to J^{k}\Gamma$ are surjective submersions for $k\geq k_\Gamma$.
\end{itemize}
\end{definition}

Lie pseudogroups are the modern version of Lie's ``continuous transformation groups''. As promised, transitivity is not assumed in our definition.
\begin{remark}
Notice that a subgroupoid of the pair groupoid $\BB\times \BB$ is a Lie subgroupoid if and only if its orbits define a regular foliation on $\BB$; hence, if $\Gamma$ is a Lie pseudogroup on $\BB$, to have $k_\Gamma=0$ in Definition~\ref{def:Lie_pseudo} the orbits of $\Gamma$ need to define a regular foliation on $\BB$.
\end{remark}

\begin{example}
All the examples from Example~\ref{ex:pseudo_as_sym} are Lie pseudogroups of order $1$. Notice that the integer $k_\Gamma$ from Definition~\ref{def:Lie_pseudo} is equal to $0$ in all examples except the pseudogroup $\Gamma^M_{g}$ of local isometries of a Riemannian metric $(M,g)$. In fact, the orbits of the $0$-th jet groupoid of $\Gamma^M_{g}$ do not necessarily define a regular foliation of $\BB$. Nonetheless, $\Gamma^M_{g}$ is a Lie pseudogroup with $k_{\Gamma^M_g}=1$.
\end{example}

\begin{example}
The pseudogroup $\Gamma_{x^n}$ from Example~\ref{ex:gamma_coord} is Lie of order $0$, and $k_{\Gamma_{x^n}}=0$. 

The pseudogroup $\Gamma_\Omega$ from Example~\ref{ex:gamma_conf} is Lie of order 1, and $k_{\Gamma_{\Omega}}=1$. 

On the other hand, the pseudogroup $\Gamma_f$ from Example~\ref{ex:gamma_hyp} is still of order $1$ but it is not Lie. In fact, the system of differential equations defining $\Gamma_f$ is singular and does not define a submanifold of $J^1\Gamma_f$; moreover, the systems obtained by differentiating the defining one are also singular. In other words, no $k_\Gamma$ exists satisfying Definition~\ref{def:Lie_pseudo}.
\end{example}

Let $\Gamma$ be a pseudogroup. Recall that, under Haefliger's correspondence described in Theorem~\ref{thm:Haef_corresp}, $\Gamma$-structures are equivalently encoded as principal $\Germ(\Gamma)$-bundles whose moment map is étale (this last condition explains while we imposed Axiom \ref{axiom_moment_map} in the Appendix):
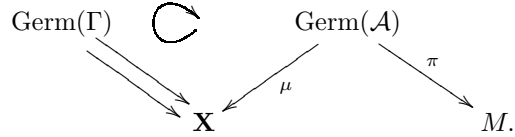
\[
\xymatrix{
\Germ(\Gamma) \ar@<0.25pc>[dr] \ar@<-0.25pc>[dr]  & \ar@(dl, ul) &  \Germ(\mathcal{A}) \ar[dl]^{\mu}\ar[dr]^{\pi} &    \\
&\BB  & &  M.}
\]
Here $\Germ(\mathcal{A})$ denotes the set of germs of charts in a $\Gamma$-atlas $\mathcal{A}$, cf.\  Definition~\ref{def:Gamma-atlas}. If $\Gamma$ is also Lie, for all $k\geq k_\Gamma$ we have the induced principal bundles (the notiation having the obvious meaning)
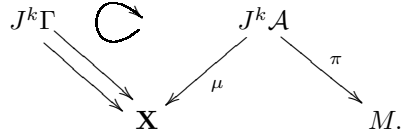
\[
\xymatrix{
J^k\Gamma \ar@<0.25pc>[dr] \ar@<-0.25pc>[dr]  & \ar@(dl, ul) &  J^k\mathcal{A} \ar[dl]^{\mu}\ar[dr]^{\pi} &    \\
&\BB  & &  M.}
\]
\begin{definition}\label{def:almost-gamma-structure}
Let $\Gamma$ be a Lie pseudogroup over $\BB$ and let $k\geq k_\Gamma$ be a positive integer. An {\bf almost $\Gamma$-structure of order $k$} is a principal $J^k\Gamma$-bundle
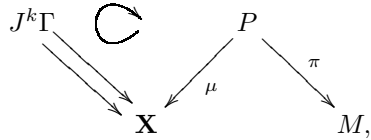
\[
\xymatrix{
J^k\Gamma \ar@<0.25pc>[dr] \ar@<-0.25pc>[dr]  & \ar@(dl, ul) &  P \ar[dl]^{\mu}\ar[dr]^{\pi} &    \\
&\BB  & &  M,}
\]
such that the inclusion of Lie groupoids $J^k\Gamma\hookrightarrow J^k({\rm Diff}_{\rm loc}(\BB))$ extends to an inclusion of principal bundles into 
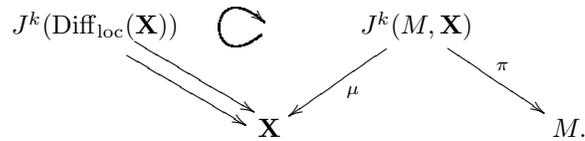
\[
\xymatrix{
J^k({\rm Diff}_{\rm loc}(\BB))\ar@<0.25pc>[dr] \ar@<-0.25pc>[dr]  & \ar@(dl, ul) &  J^k(M,\BB) \ar[dl]^{\mu}\ar[dr]^{\pi} &    \\
&\BB  & &  M.}
\]
\end{definition}

\begin{remark}
In the terminology from~\cite{FRANCESCO}, an almost $\Gamma$-structure of order $k$ is a $J^k\Gamma$-reduction of $J^k (M,\BB)$. The techniques developed there concerning integrability work actually for a more general class of principal groupoid bundles, namely principal Pfaffian bundles, which we present in Definition~\ref{def:princ-pfaff-bundles}. We stress already here that, anyways, any almost $\Gamma$-structure is Pfaffian.
\end{remark}

\begin{example}
A Riemannian manifold $(M,g)$ can be equivalently described as a $J^1\Gamma_{\rm eucl}$-structure, where $\Gamma_{\rm eucl}$ is the pseudogroup of local isometries of the Euclidean metric $g^{\rm eucl}$ on $\mathbb{R}^n$. Explicitly proving such claim is rather instructive; we sketch the argument below.

The key fact to use is that principal groupoid bundles are equivalently encoded into equivalence classes of groupoid cocycles, a fact already exploited in Theorem~\ref{thm:Haef_corresp} for étale groupoids.

On the one hand, let $(M,g)$ be a Riemannian manifold. Notice that for all $(x,y)\in \mathbb{R}^n\times \mathbb{R}^n$ there is a subbundle ${\rm Aut}_{g^{\rm eucl}}(T\mathbb{R}^n,T\mathbb{R}^n)$ of the bundle ${\rm Aut}(T\mathbb{R}^n,T\mathbb{R}^n)$ of linear maps between $T_x\mathbb{R}^n$ and $T_y\mathbb{R}^n$ consisting of linear isometries from $(T_x\mathbb{R}^n,g^{\rm eucl}_x)$ to $(T_y\mathbb{R}^n,g^{\rm eucl}_y)$. If $(\{U_i\}_{i\in I},\{f_i\}_{i\in I})$ is an atlas of $M$, then we can ``complete'' it by choosing, for each par $i$, $j$ of indices in $I$, a smooth section $L_{ij}$ of ${\rm Aut}_{g^{\rm eucl}} (T\mathbb{R}^n,T\mathbb{R}^n)$ over $U_i\cap U_j$ lifting the transition function $\varphi_{ij}$ from $U_i\cap U_j$. The pair $(\varphi_{ij},L_{ij})$ is also called a {\it formal isometry} over $f_i(U_i\cap U_j)$. Notice that, if $\varphi_{ij}$ was an isometry, then we could choose $L_{ij}=d\varphi_{ij}$. The pair $(\varphi_{ij},L_{ij})$ canonically defines a map 
\[
\sigma_{ij}: U_i \cap U_j \to J^1\Gamma_{\rm eucl};
\]
such map is tangent to the Cartan distribution on $J^1\Gamma_{\rm eucl}$ -- see Definition~\ref{def:cart-dist} and Proposition~\ref{prp:cart-dist} -- precisely when $\varphi_{ij}$ is an isometry and $L_{ij}=d\varphi_{ij}$. It follows immediately that collection of $\sigma_{ij}$'s defines a cocycle valued in $J^1\Gamma_{\rm eucl}$, hence a principal $J^1\Gamma_{\rm eucl}$-bundle. and that such cocycle can be ``promoted'' to a $\Gamma_{\rm eucl}$-cocycle if and only if the fixed atlas is a $\Gamma_{\rm eucl}$-atlas; that is, if and only if the Riemannian metric $g$ is flat.

On the other hand, given an almost $J^1\Gamma_{\rm eucl}$-structure of order 1 on $M$, we can choose a $J^1\Gamma_{\rm eucl}$-cocycle supported by an open cover $\{U_i\}_{i\in I}$ of $M$, and local sections over the opens $U_i$'s of the principal $J^1\Gamma_{\rm eucl}$-bundle compatible with the cocycle. The cocycle can be used to construct an atlas $(\{U_i\}_{i\in I},\{f_i\}_{i\in I})$ with transition functions $\varphi_{ij}$'s ``completed'' with linear isometries $L_{ij}$'s as above. The local sections define sections of 
\[
{\rm Iso}(TM|_{U_i},T\mathbb{R}^n|_{f_i(U_i)})\to U_i,
\] for each $i\in I$. A Riemannian metric $g_i$ over $U_i\subset M$ can be defined using such a section to pullback $g^{\rm eucl}$; the compatibility between the cocycle $J^1\Gamma_{\rm eucl}$ and the sections of the principal $J^1\Gamma_{\rm eucl}$-bundle imply that there exists a (unique) metric $g$ on $M$ such that, for all $i\in I$ 
\[
g|_{U_i} = g_i.
\]
Then on each $T_xM$ one defines a metric $g_x$ on $T_xM$ as the pullback via $d_xf$ of ${g^{\rm eucl}}$ at $f(x)\in \mathbb{R}^n$.
\end{example}

Let us sum up the main points discussed in this subsection. Given a Lie pseudogroup $\Gamma$, we have a commutative diagram of Lie groupoids
\[
\begin{tikzcd}
\dots\arrow[r] & J^{k_\Gamma +1} \Gamma \arrow[r]& J^{k_\Gamma} \Gamma  \arrow[shift right,r]\arrow[r,shift left] & \BB\\
\Germ(\Gamma)\arrow[u]\arrow[ur]\arrow[urr,] & & 
\end{tikzcd} 
\]
where each horizontal map is a surjective submersion.
This object should be interpreted as the abstract groupoid-theoretical concept corresponding to the {\it Lie} pseudogroup $\Gamma$ in the same sense in which the effective étale Lie groupoid $\Germ(\Gamma)$ corresponds to the pseudogroup $\Gamma$. 
In the next section, we will consider (see Definitions~\ref{def_Pfaffian_groupoid-form} and~\ref{def_Pfaffian_groupoid-distribution}) a class of groupoids whose properties are designed to capture the properties of the pair $(J^k\Gamma,\omega)$, where $\Gamma$ is a Lie pseudogroup and $\omega$ is (the restriction to $J^k\Gamma$ of) the Cartan form. Similarly, we will introduce an abstract version of almost $\Gamma$-structures (see Definitions \ref{def:princ-pfaff-bundles} and \ref{def:princ-pfaff-bundles-dist}).

\section{The Pfaffian language}\label{sec:Pfaffian-language}

As discussed in the previous section, geometries on manifolds can be described using principal groupoid bundles. Moreover, the groupoids involved in such a description (and, in fact, the bundles themselves) are jet prolongations, and as such they carry additional structures: the Cartan distribution (see Section \ref{app:jets}). The framework that we present below, introduced in~\cite{MARIA} and later developed in~\cites{ORI,FRANCESCO}, allows to study jet groupoids and their principal bundles from an abstract point of view. This approach will be quite advantegeous especially when dealing with Morita equivalences in Section~\ref{sec:PME}. 

There are two equivalent ways of describing all the concepts presented in this section: one makes use of distributions, and the other one of vector-valued 1-forms. We present them both, as it is often the context to dictate which approach provides the best insight. 

\subsection{Pfaffian groupoids}

\begin{definition}[cf.\ Definition \ref{def_Pfaffian_groupoid-distribution}]\label{def_Pfaffian_groupoid-form}
A {\bf Pfaffian groupoid} $(\G, \omega, E)$ over $\BB$ consists of a Lie groupoid $\G \rightrightarrows \BB$ together with a representation $E \to \BB$ of $\G$ and a differential form $\omega \in \Omega^1 (\G, t^* E)$ such that
\begin{enumerate}
\item $\omega$ is {\bf multiplicative} i.e.\
\[
(m^*\omega)_{(g, h)}=(\pr_1^*\omega)_{(g, h)}+g\cdot (\pr_2^*\omega)_{(g, h)},\quad (g, h)\in \G\times_\BB\G;
\]
\item $\omega$ has constant rank; 
\item $\omega$ is {\bf $s$-transversal}, i.e.\ $T \G = \ker (\omega) + \ker (ds) \subset T \G$;
\item the subbundle
\[
\g(\omega):=\ker(\omega)\cap \ker(ds) \subset T\G
\]
is involutive;
\item it holds
\begin{equation*}
\ker (\omega) \cap \ker(dt) = \ker (\omega) \cap \ker(ds). 
\end{equation*}
\end{enumerate}
We call $\g(\omega)$ the {\bf symbol bundle} of $(\G, \omega, E)$. A {\bf holonomic bisection} of $(\G, \omega, E)$ is a local bisection $\sigma:U\subset \BB \to \G$ such that $\sigma^*\omega = 0$.

A Pfaffian groupoid $(\G,\omega, E)$ is called {\bf full} if the form $\omega$ is pointwise surjective.
\end{definition}

The main examples to have in mind are the jet groupoids of a Lie pseudogroup together with their Cartan form -- see Example~\ref{exm:jet-groupoids} below.

\begin{remark}\label{symbol_space_is_ideal}
One can replace the condition 4.\ on $\g(\omega)$ with the involutivity of $\mathfrak{g}_\BB: = (\ker (\omega) \cap \ker (ds))|_{ \BB} \subset A$, i.e.\ asking that $\g_\BB$ is a Lie subalgebroid of $A:={\rm Lie}(\G)$. We will often use this point of view. We also stress here that $\g_\BB$ is in fact an {\it ideal} of $A$; see Proposition~\ref{prp:coefficent-space-is-alg}, point $1$.
\end{remark}


\begin{remark}
Actually, the $s$-transversality condition in Definition~\ref{def_Pfaffian_groupoid-form} follows from the multiplicativity condition and the fact that $\omega$ has constant rank. See~\cite[Proposition 3.4.12]{FRANCESCO}.
\end{remark}
We list below some properties of Pfaffian groupoids that are relevant for our discussion; see~\cite{MARIA} for details.

\begin{proposition}\label{prp:adjoint-rep}
Let $(\Sigma, \omega, E)$ be a Pfaffian groupoid: then the following facts hold true.
\begin{itemize}
\item The quotient $A/\mathfrak{g}_\BB\to \BB$ is a representation of $\Sigma\tto \BB$. The action is given by
\[
g\cdot [V_x] \mapsto [dR_{g^{-1}}\cdot dm({\rm hor}^\omega_g(dt(V_x)), V_x)],
\]
where $g\in \Sigma$, $x=s(g)$, $[V_x]\in \ker(ds)|_\BB/\g_\BB$, and ${\rm hor}^\omega_g(v_x)$ is any $s$-lift of $v_x$ tangent to $\ker(\omega)$.
\item The tangent space $T\BB\to \BB$ is a representation of $\Sigma\tto \BB$. The action is given by
\[
\Sigma\times_\BB T\BB \to T\BB,\quad (g, v_x) \mapsto dt({\rm hor}^\omega_g(v_x)),
\]
where ${\rm hor}^\omega_g(v_x)$ is any $s$-lift of $v_x$ tangent to $\ker(\omega)$.
\item The map $A/\mathfrak{g}_\BB\to T\BB$ induced by the anchor of $A$ is a morphism of representations.
\end{itemize}
\end{proposition}

\begin{remark}
If $(\Sigma, \omega, E)$ is a Pfaffian groupoid, the representation on $T\BB$ from the above proposition extends the normal representation of $\Sigma$ (see the discussion after Proposition \ref{prp:Transverse-geometry}). 

Moreover, there is an induced representation of the isotropy group $\Sigma_x$ on $A_x/\g_\BB$. If $\g_x$ denotes the Lie algebra of $\Sigma_x$ carrying the adjoint representation, we see that the map
\[
\g_x\to A_x/\g_\BB
\]
induced by the quotient map is a morphism of $\G_x$-representations.
\end{remark}

In what follows, we denote by $T_\BB\Sigma$ the pullback of $T\Sigma \to \Sigma$ via the unit map $u: \BB \to \Sigma$.

\begin{proposition}\label{prp:coefficent-space-is-alg}
Let $(\Sigma, \omega, E)$ be a Pfaffian groupoid. 
\begin{enumerate}
\item The image $\omega(T_\BB\Sigma)$ is a vector bundle isomorphic to  $A/\mathfrak{g}_\BB$ and possesses a canonical Lie algebroid structure making 
\[
\omega|_{\ker(ds)|_\BB}: \ker(ds)|_\BB\to \omega(\ker(ds)|_\BB)
\]
into a Lie algebroid map.
\item The inclusion $\omega(T_\BB\Sigma)\hookrightarrow E$ is a map of representation.
\item The following map is an isomorphism of representations:
\[
\omega(T_\BB \Sigma)\to A/\mathfrak{g}_\BB,\quad \omega(v_g)\mapsto [v^{\ker(ds)}_g],
\]
where $v_x\in T_x\Sigma$ and $v_x=v^{\ker(ds)}_x+v_x^{\ker(\omega)}$, $v^{\ker(ds)}_x\in \ker_x(ds)$, $v_x^{\ker(\omega)}\in \ker_x(\omega)$.
\end{enumerate}
\end{proposition}

As a consequence, {\it full} Pfaffian groupoids $(\Sigma, \omega, E)$ are completely encoded into the Lie groupoid $\Sigma$ and the form $\omega$ -- the representation $E\to \BB$ can be reconstructed out of these two pieces of data. 

\begin{remark}
A remark about terminology: in \cites{MARIA, ORI, FRANCESCO} a Pfaffian groupoid was not required to satisfy $\ker(\omega) \cap \ker(dt) = \ker(\omega) \cap \ker(ds)$ (a full Pfaffian groupoid satisfying such condition was called ``Lie-Pfaffian''). In such setting, Proposition \ref{prp:coefficent-space-is-alg} does not hold. For our purposes, it is more convenient to require the ``Lie-Pfaffian'' condition directly in the definition. A more detailed discussion on the Lie-Pfaffian condition, also in the non-full case, can be found in \cite{INPROGRESS}.
\end{remark}

As Proposition~\ref{prp:coefficent-space-is-alg} suggests, full Pfaffian groupoids can be equivalently described in terms of distributions.

\begin{proposition}\label{prp:dist-picture}
A full Pfaffian groupoid is equivalently described as a groupoid $\G\tto \BB$ equipped with a distribution $\CC$ that is
\begin{itemize}
\item {\bf multiplicative} in the sense that $\CC\subset T\Sigma$ is a wide subgroupoid (the right hand side denotes the tangent groupoid $T\Sigma\tto T\BB$, where the structure maps are obtained by differentiating those of $\Sigma\tto \BB$);
\item {\bf $s$-transversal} in the sense that $\CC+\ker(ds)=T\G$;
\item such that $\g(\CC):=\CC\cap \ker(ds)$ is involutive;
\item such that $\CC\cap \ker(ds)=\CC\cap \ker(dt)$.
\end{itemize}
\end{proposition}

More precisely, one has the following facts (see~\cite{CrainicSalazarStruchiner15} for the details).
\begin{itemize}
\item Given a distribution $\CC$ on $\G$ as Proposition~\ref{prp:dist-picture}, the projection onto the normal bundle $T\G\to \nu_{\CC}$ post-composed with the right multiplication defines a form $\omega\in \Omega^1(\G, t^*(A/\g(\CC)|_\BB) )$. Such form $\omega$ makes $\G$ into a full Pfaffian groupoid.
\item On the other hand, given a full Pfaffian groupoid $(\G, \omega, E)$, the kernel $\ker(\omega)$ has the properties listed in Proposition~\ref{prp:dist-picture}.
\end{itemize}

Motivated by the above proposition, we give the following alternative definition of (not necessarily full) Pfaffian groupoid.

\begin{definition}[cf.\ Definition \ref{def_Pfaffian_groupoid-form}]\label{def_Pfaffian_groupoid-distribution}
A {\bf Pfaffian groupoid} $(\G, \CC, E)$ over $\BB$ is the datum of
\begin{itemize}
\item a Lie groupoid $\G\tto \BB$;
\item a regular distribution $\CC$ on $\Sigma$ that is:
\begin{itemize}
\item multiplicative;
\item transversal to $s$;
\item such that $\g(\CC):=\CC\cap \ker(ds)$ is involutive;
\item such that $\g(\CC)=\CC\cap \ker(dt)$;
\end{itemize}
\item a representation $E\to \BB$ of $\Sigma\tto \BB$ admitting the subrepresentation on $\nu_\CC|_\BB \to \BB$
\[
g\cdot [V_x]:= [dR_{g^{-1}}\cdot dm({\rm hor}^s_g(dt(V_x)), V_x)],
\]
where $g\in \Sigma$, $x=s(g)$, and $[V_x]\in \nu_\CC|_\BB$ is an element of the normal bundle of $\CC$ at $x\in \BB$ -- see Remark~\ref{rmk:explain_rep} for the notation.
\end{itemize}
We call $\g(\CC)$ the {\bf symbol bundle} of $(\G, \CC, E)$. A {\bf holonomic bisection} of $(\G, \CC, E)$ is a local bisection $\sigma:U\subset \BB \to \G$ tangent to the distribution $\CC$.

A Pfaffian groupoid $(\Sigma, \CC, E)$ is called {\bf full} if $E$ is isomorphic to $\nu_\CC|_\BB$ as representations.
\end{definition}

\begin{remark}\label{rmk:explain_rep}
Let us explain the notation appearing in the last point of Definition~\ref{def_Pfaffian_groupoid-distribution}. First of all, ${\rm hor}^s_g$ denotes any $s$-projectable lift tangent to $\CC$.

We observe that, given a Lie groupoid $\Sigma\tto \BB$, the normal bundle $\nu_\BB$ can be identified with $\ker(ds)|_\BB$, the total space of $A={\rm Lie}(\Sigma)$. If $(\G, \CC, E)$ is a Pfaffian groupoid in the sense of the previous definition, $\CC$ is transversal to $s$, and we may identify $\nu_\CC$ with $\nu_\BB/\mathfrak{g}(\CC)$. Consequently, equivalence classes in $\nu_\CC$ can be represented by $s$-vertical vectors. 
\end{remark}

%



In what follows, we will alternatively use the form description of Pfaffian groupoids from Definition~\ref{def_Pfaffian_groupoid-form} and the distribution one from Definition~\ref{def_Pfaffian_groupoid-distribution}, according to convenience. We will often use the notations $(\Sigma, \omega)$ and $(\Sigma, \CC)$ when working with full Pfaffian groupoids. In such a situation, the reader should always keep in mind that $\omega\in \Omega^1(\Sigma, t^*E)$ is a vector-valued form and its coefficent space is the pullback of a representation $E\to \BB$ of $\Sigma$ possessing a Lie algebroid structure (Proposition~\ref{prp:coefficent-space-is-alg}); moreover, $E\cong A/\mathfrak{g}_\BB$ both as an algebroid and as a representation.
\begin{definition}\label{def:morphism_Pfaffian_groupoids}
Let  $(\G_1, \omega_1, E_1)$ and $(\G_2, \omega_2, E_2)$ be Pfaffian groupoids on $\BB_1$ and $\BB_2$, respectively.
 A {\bf morphism of Pfaffian groupoids} from $(\G_1, \omega_1, E_1)$ to $(\G_2, \omega_2, E_2)$ is a pair $(\Phi, \Psi)$, where $\Phi: \G_1 \to \G_2$ is a Lie groupoid morphism (Definition \ref{def_morphisms_groupoids}) and $\Psi: E_1 \to E_2$ is a morphism of representations (see below), such that $\Phi$ and $\Psi$ lift the same map $\phi:\BB_1\to \BB_2$ and 
 \[
 \Phi^*\omega_2 = \Psi \circ \omega_1.
 \] 
\end{definition}
By ``morphism of representations'' we mean a vector bundle morphism $\Psi$ from $\pi_1:E_1\to \BB_1$ to $\pi_2:E_2\to \BB_2$ covering $\phi:\BB_1\to \BB_2$, such that
\[
\Psi(g_1\cdot e_1)=\Phi(g_1)\cdot \Psi(e_1)
\]
for all $e_1\in E_1$, $g_1\in \G_1$ such that $\pi_1(g_1)=s(g_1)$.

Notice that, in the above definition, if both $(\Sigma_1, \omega_1)$ and $(\Sigma_2, \omega_2)$ are full Pfaffian groupoids then $\Psi$ is uniquely determined by $\Phi$. In fact, $E_i\cong A_i/(\g_i)_\BB$, therefore $\Psi$ is forced to be induced by the map ${\rm Lie}(\Phi):A_1\to A_2$ from Proposition~\ref{prp:morph-algebroid}.

\begin{example}[Main example]\label{exm:jet-groupoids}
As anticipated at the end of Section~\ref{sec:intro_section}, when $\Gamma$ is a Lie pseudogroup over $\BB$, the jet groupoid $J^k\Gamma\tto \BB$ is a Pfaffian groupoid for all $k\geq k_\Gamma$ (see Definition~\ref{def:Lie_pseudo}). The Pfaffian form is given by the restriction to $J^k\Gamma$ of the Cartan form $\omega^k$ on $J^k(\BB\times \BB)$, from Proposition \ref{prp:cart-dist}. It is an instructive exercise to check that such restriction has constant rank and has all the properties listed in Definition~\ref{def_Pfaffian_groupoid-form}. Moreover, by Lemma \ref{lemma_holonomic_bisection_jet_pseudogroup}, the holonomic bisections of $(J^k \Gamma,\omega^k)$ are in bijection with the elements of $\Gamma$, for any $k$ greated or equal than the order of $\Gamma$ (Definition \ref{def:order}).

Notice also that the projection $\pr: (J^{k+1}\Gamma, \omega^{k+1})\to (J^{k}\Gamma, \omega^{k})$ is a morphism of Pfaffian groupoids for all $k\in \mathbb{N}$, $k\geq k_\Gamma$. 
Actually, the symbol bundle $\g(\omega^{k+1})$ of $(J^{k+1}\Gamma,\omega^{k+1})$ is given precisely by the vertical bundle of the projection $\pr$ restricted to $\BB$ -- as one sees by Proposition \ref{prp:cart-dist}. It follows that the image of the restriction of $\omega^{k+1}$ on $J^{k+1}\Gamma$ is isomorphic to
\[
{\rm Lie}(J^{k+1}\Gamma)/ \g(\omega^{k+1})|_\BB \cong {\rm Lie}(J^{k}\Gamma).
\]
Consequently, the Pfaffian groupoid $(J^{k+1}\Gamma, \omega^{k+1}, {\rm Lie}(J^{k}\Gamma))$ is full. 

For more details see, for instance,~\cite{MARIA,ORI}.
\end{example}

A second class of examples is explored in the next subsection.

\subsection{Pfaffian groups}\label{subs:Pf-groups}

Recall that the {\bf Maurer-Cartan form} $\omega_{\rm MC}$ of a Lie groupoid $\Sigma\tto \BB$ is the $s$-foliated $t^*A$-valued form
\[
\omega_{\rm MC}: \ker(ds) \to t^*A\cong t^*\left(\ker(ds)|_\BB\right), \quad  V_g\mapsto (g, dR_{g^{-1}}(V_g)),
\] 
equivalently encoded into the bundle map
\[
\begin{tikzcd}
\ker(ds)\arrow[r, "\omega_{\rm MC}"]\arrow[d] & A\cong \ker(ds)|_\BB\arrow[d]\\
\Sigma\arrow[r, "t"] & \BB 
\end{tikzcd}
\]
sending $V_g$ to $dR_{g^{-1}}(V_g)\in A_{t(g)}$.

\begin{example}
When $\Sigma\tto \BB$ is a Lie group $G$, the Maurer-Cartan form becomes the usual Maurer-Cartan form of a Lie group (defined using right multiplication).
\end{example}

\begin{lemma}~\label{lemma:Pfaffian_and_MC}
Let $(\G, \omega, E)$ be a Pfaffian groupoid. The map $\omega|_{\ker(ds)}: \ker(ds)\to E$ factors as $l\circ \omega_{\rm MC}$, where $l:A\to E$ is defined as $\omega|_\BB\circ \iota_A$. 
\end{lemma}
In the statement above, $\iota_A$ denotes the inclusion of the Lie algebroid $A\cong \ker(ds)|_\BB$ into $T\Sigma_\BB$, and $\omega|_\BB$ is the restriction of the vector bundle map $\omega:T\Sigma\to t^*E$ to the image of $\BB$ in $\Sigma$ via the unit bisection.
\begin{proof}
Given
\[
\omega: T\G\to E, \quad v_g\mapsto \omega(v_g).
\]
we use the multiplicativity equation
\[
(m^*\omega)_{(g, h)}=(\pr_1^*\omega)_{(g, h)}+g\cdot (\pr_2^*\omega)_{(g, h)}
\]
on a pair of vectors of the form $(v_g, 0_{g^{-1}})$, $v_g\in \ker_g(ds)$.
We get
\[
\omega(dR_{g^{-1}}(v_g))=\omega(v_g),\quad v_g\in \ker_g(ds),
\]
which implies the claim.
\end{proof}

The following Proposition and its proof should be compared with Example $3.4.6$ in~\cite{FRANCESCO}.

\begin{proposition}[Characterisation of Pfaffian groups]\label{prp:char-pf-groups}
Let $G$ be a Lie group and $V$ be a representation of $G$. A form $\omega\in \Omega^1(G, V)$ makes the triple $(G, \omega, V)$ into a Pfaffian group over the point if and only if
\begin{itemize}
\item $\omega$ is right invariant;
\item $\omega_e:\g\to V$
is a map of representations (with $\g$ carrying the adjoint representation).
\end{itemize}
In fact, there is a one to one correspondence between Pfaffian group structures $(G, \omega, V)$ and representation maps
\[
l:\g\to V.
\]
\end{proposition}

According to the correspondence above, the symbol bundle $\g(\omega)$ of a Pfaffian group corresponds to $\ker(l)$.

\begin{proof}
If $(G, \omega, V)$ is a Pfaffian group, then $\omega$ is right invariant and $\omega_e$ is a map of representations thanks to multiplicativity. In fact
\[
(m^*\omega)_{(g, h)}=(\pr_1^*\omega)_{(g, h)}+g\cdot (\pr_2^*\omega)_{(g, h)}, \quad g,\ h\in G
\]
implies:
\begin{itemize}
\item left equivariance 
\[
\omega_{gh}\circ dL_g=g\cdot \omega_h,\quad g,\ h\in G;
\]
\item right invariance
\[
\omega_{gh}\circ dR_h=\omega_{g}\quad g,\ h\in G;
\]
\end{itemize}
from which it follows
\[
\omega_e\circ dL_g\circ dR_{g^{-1}}=g\cdot \omega_e,
\]
i.e.\ $\omega_e$ is a map of representations. Conversely, if $\omega_e$ is a map of representations and $\omega$ is right invariant, then $\omega$ is also left equivariant. Multiplicativity of $\omega$ follows since
\[
(m^*\omega)_{(g, h)}(v_g, v_h)=(m^*\omega)_{(g, h)}(v_g, 0_h)+(m^*\omega)_{(g, h)}(0_g, v_h).
\]
Notice that for this last claim we use the fact that $G$ is a Lie group, i.e.\ its unit manifold is the point.


If now
\[
l:\g\to V
\]
is a map of representations, then there is one and only one induced right invariant form $\omega_l\in \Omega^1(G, V)$ such that $(\omega_l)_e=l$. Consequently, the one to one correspondence from the claim holds true.
\end{proof}

\begin{corollary}\label{cor:Pfaffian-groups-MC}
Let $(G, \omega, V)$ be a Pfaffian group. Then
\[
\omega=l \circ \omega_{\rm MC}
\]
where $l:\g\to V$ is the map of representations from Proposition~\ref{prp:char-pf-groups} -- i.e.\ $\omega_e$.
\end{corollary}
\begin{proof}
By Lemma~\ref{lemma:Pfaffian_and_MC} together with Proposition~\ref{prp:char-pf-groups}.
\end{proof}

\begin{remark}\label{rmk:Pf-form-inv-equiv}
If $(\Sigma, \omega, E)$ is a Pfaffian groupoid, then multiplicativity implies that
\begin{itemize}
\item $\omega|_{\ker(ds)}$ is right invariant;
\item $\omega|_{\ker(dt)}$ is left equivariant;
 \item for all $x\in \BB$ the form $\omega_x|_{\Sigma_x}:T_x\Sigma_x\to E_x$ is a map of representations;
 \end{itemize}
 by the same computations as in the above proposition. However, these properties do not characterise Pfaffian forms on Lie groupoids. Notice that the fact that $\omega_x|_{\Sigma_x}$ is a map of representations was already a consequence of Proposition~\ref{prp:coefficent-space-is-alg}. See also~\cite[Lemma 3.4.11]{FRANCESCO}. 
\end{remark}
\begin{proposition}[Isotropies of Pfaffian groupoids]\label{prp:isotropies-pf-groups}
Let $(\Sigma, \omega, E)$ be a Pfaffian groupoid over $\BB$. Then the triple $(\Sigma_x, \omega|_{\Sigma_x}, E_x)$ is a Pfaffian group. Furthermore for all $x, y\in \BB$ and $g\in \Sigma$ such that $s(g)=x$, $t(g)=y$ the isomorphism of Lie groups
\[
\phi_g: \Sigma_x\to \Sigma_y, \quad h\mapsto ghg^{-1}
\]
and the isomorphism of representations
\[\psi_g: E_x\to E_y, \quad \alpha_x\mapsto g\cdot \alpha_y\]
make $(\psi_g, \phi_g)$ into an isomorphism of Pfaffian groups from $(\Sigma_x, \omega|_{\Sigma_x}, E_x)$ to $(\Sigma_y, \omega|_{\Sigma_y}, E_y)$.
\end{proposition}
\begin{proof}
The fact that the triple $(\Sigma_x, \omega|_{\Sigma_x}, E_x)$ is a Pfaffian group is immediate. The fact that the coniugation by $g$ is a Pfaffian isomorphism follows from the equivariance and invariance properties of $\omega$ coming from multiplicativity, see Remark~\ref{rmk:Pf-form-inv-equiv}.
\end{proof}
\begin{remark}\label{remark:isotropy-not-full-not-Lie}
Let $(\Sigma, \omega, E)$ be a {\it full} Pfaffian groupoid over $\BB$ -- i.e.\ $\omega$ is pointwise surjective. Its isotropy Pfaffian group $(\Sigma_x, \omega|_{\Sigma_x}, E_x)$ is almost never full; that happens if and only if the Lie groupoid $\Sigma$ is a bundle of Lie groups over $\BB$.
\end{remark}

\subsection{Principal Pfaffian bundles}

Let $(\G, \omega, E)$ be a Pfaffian groupoid over $\BB$ (Definition \ref{def_Pfaffian_groupoid-form}).

\begin{definition}[cf.\ Definition \ref{def:princ-pfaff-bundles-dist}]\label{def:princ-pfaff-bundles}
A {\bf principal $(\G, \omega, E)$-bundle} over $M$ is a principal $\G$-bundle $\pi:P\to M$ with moment map $\mu:P\to \BB$ together with a 
1-form $\theta\in \Omega^1(P, \mu^*E)$ such that 
\begin{itemize}
\item the form $\theta$ is {\bf multiplicative} with respect to the action $m_P:\G\times_\BB P\to P$ of $\G$ and the form $\omega$, i.e.\
\[
(m_P^*\theta)_{(g, p)}=(\pr_1^*\omega)_{(g,p)}+g\cdot (\pr_2^*\theta)_{(g, p)},\quad (g, p)\in \G\times_\BB P;
\]
\item $\ker(d\mu)\cap \ker(\theta)=\ker(d\pi)\cap \ker(\theta)$.
\end{itemize}
The principal bundle is called {\bf full} if $\theta\in \Omega^1(P, \mu^*E)$ is pointwise surjective. 

When the Pfaffian groupoid $(\G, \omega, E)$ is clear from the context, we will often use the simpler terminology {\bf principal Pfaffian bundles}.
\end{definition}

We will picture a principal $(\G, \omega, E)$-bundle as
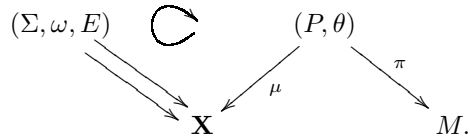
\[
\xymatrix{
(\Sigma, \omega, E) \ar@<0.25pc>[dr] \ar@<-0.25pc>[dr]  & \ar@(dl, ul) &  (P, \theta) \ar[dl]^{\mu}\ar[dr]^{\pi} &    \\
&\BB  & &  M.}
\]

Of course, an analogous definition can be given where the groupoid acts from the right -- see also Remark~\ref{remark:mullt-from-right} later on.

\begin{remark}\label{rmk:inf_action_Pfaffian}
The multiplicativity of $\theta$ implies $\theta(a(\alpha))=\omega(\alpha)$ for all $\alpha\in \ker(ds) \cong t^*A$. As a consequence, the isomorphism $\mu^*\ker(ds)\cong \ker(d\pi)$ induced by the infinitesimal $\Sigma$-action on $P$ (Remark \ref{infinitesimal_action}) restricts to an isomorphism
 \[
\mu^*(\ker(\omega)\cap\ker(ds))\cong \ker(\theta)\cap \ker(d\pi). \qedhere
\]
We conclude that $\ker(d\pi)\cap \ker(\theta)$ is involutive since $\g(\omega) = \ker(\omega) \cap \ker(ds)$ is so.
\end{remark}

\begin{remark}
We observe that fullness of $(P, \theta)$ follows from fullness of $(\Sigma, \omega, E)$. On the other hand, if $\ker(\theta)$ is required to be transverse to $\ker(d\pi)$ -- i.e.\ if $(P, \theta)$ is a {\bf Pfaffian fibration} according to~\cite[Definition 3.1]{CATTMARIUSMARIA} -- then fullness of $(\Sigma, \omega, E)$ follows from fullness of $(P, \theta)$. At a first sight, this seems to suggest that the definition of principal ``Pfaffian'' bundle could be modified requiring $\ker(\theta)$ to be transverse to $\ker(d\pi)$ (a property satisfied by $(\Sigma, \omega, E)$ acting on itself from the left). However, some examples that we think are of great importance are ruled out by such a definition. Most notably, the action of the isotropy group $\Sigma_x$ at $x\in \BB$ of a (full!) principal $(\Sigma, \omega, E)$-bundle on $s^{-1}(x)$ does not give a principal $(\Sigma_x, \omega|_{\Sigma_x}, E_x)$-bundle when equipped with the form $\omega|_{s^{-1}(x)}$. Compare also with Theorem~\ref{prop_transitive_Pfaffian_Morita_equivalence}.
\end{remark}

As for principal bundles, the definition of morphism between principal Pfaffian bundles is quite natural, provided that the structure groupoid remains the same.

\begin{definition}\label{def:princ-pfaff-bundles-morph}
Let $(P_1, \theta_1)$ and $(P_2, \theta_2)$ be principal $(\Sigma, \omega, E)$-bundles. A {\bf morphism of principal $(\Sigma, \omega, E)$-bundles}
\[
(P_1, \theta_1) \to (P_2, \theta_2)
\]
consists of a morphism of principal $\Sigma$-bundles
\[
\Psi:P_1\to P_2
\]
such that $\Psi^*(\theta_2)=\theta_1$.
\end{definition}
\begin{remark}
Just like for principal $\Sigma$-bundles, any morphism of principal $(\Sigma, \omega, E)$-bundles is an isomorphism.
\end{remark}

Considering a Pfaffian groupoid $(\G, \CC, E)$ in terms of distributions (Definition~\ref{def_Pfaffian_groupoid-distribution}) we have the following equivalent definition of principal Pfaffian bundle.
\begin{definition}[cf.\ Definition \ref{def:princ-pfaff-bundles}]\label{def:princ-pfaff-bundles-dist}
A {\bf principal $(\G, \CC, E)$-bundle} over $M$ is a principal $\G$-bundle $\pi:P\to M$ with moment map $\mu:P\to \BB$ together with a distribution $\CC_P\subset TP$ such that 
\begin{itemize}
\item the distribution $\CC_P$ is {\bf multiplicative} with respect to the action $m_P:\G\times_\BB P\to P$ of $\G$ and the distribution $\CC$, i.e.\ $\CC\cdot \CC_P\subset \CC_P$ (here $\cdot$ denotes the differential of the action map);
\item the injection $\mu^*(\CC\cap\ker(ds))\hookrightarrow \CC_P\cap \ker(d\pi)$ induced by the infinitesimal action is an isomorphism;
\item $\ker(d\mu)\cap \CC_P=\ker(d\pi)\cap \CC_P$.
\end{itemize}
It is called {\bf full} when the inclusion (see Remark~\ref{rmk:full_pp_dist_well_given} below) 
\[
\nu_{\CC_P\cap \ker(d\pi)}\hookrightarrow \mu^*E
\]
of the normal bundle of $\CC_P\cap \ker(d\pi)$ in $\ker(d\pi)$ induced by the infinitesimal action extends to an isomorphism
\[
\nu_{\CC_P}\cong \mu^*E.
\]
\end{definition}
We will use the notation
\[
\xymatrix{
(\Sigma, \CC, E) \ar@<0.25pc>[dr] \ar@<-0.25pc>[dr]  & \ar@(dl, ul) &  (P, \CC_P) \ar[dl]^{\mu}\ar[dr]^{\pi} &    \\
&\BB  & &  M,}
\]
\begin{remark}\label{rmk:full_pp_dist_well_given}
Recall that $(\nu_{\CC})|_\BB$ is a vector subbundle of $E$ (by Definition~\ref{def_Pfaffian_groupoid-distribution} of Pfaffian groupoid). If we require the infinitesimal action, cf.\ Remark~\ref{rmk:inf_action_Pfaffian}, to induce an isomorphism
\[
\mu^*(\CC\cap\ker(ds))\cong \CC_P\cap \ker(d\pi),
\]
then the normal bundle of $\CC_P\cap \ker(d\pi)$ in $\ker(d\pi)$ can be seen as a vector subbundle of $\mu^*E$ and the Definition~\ref{def:princ-pfaff-bundles-dist} above makes sense.
\end{remark}

%

The following statement, proved in~\cite[Proposition 5.4.3]{FRANCESCO}, adapts the gauge construction from Proposition \ref{prop_gauge_construction} to the Pfaffian setting.

\begin{proposition}\label{prop:Pfaffian_Gauge_construction}
Let $(\G, \omega, E)$ be a Pfaffian groupoid over $\BB$ and $(P, \theta)$ a (left) principal $(\G, \omega, E)$-bundle over $M$. The gauge groupoid $\Gaug(P) = (P\times_M P)/\G\tto M$ can be endowed with a 1-form $\omega_{\rm gauge}$ and a representation $E_{\rm gauge}$ in such a way that the right action on $P$ makes $(P, \theta)$ into a (right) principal $(\Gaug(P), \omega_{\rm gauge}, E_{\rm gauge})$-bundle.
\end{proposition}


We conclude this section by recalling some examples.

\begin{example}\label{ex:trivial_pfaffian_bundle}
The action of a Pfaffian groupoid $(\G, \omega, E)$ on itself by left multiplication yields by definition a principal $(\G, \omega, E)$-bundle.
\end{example}

\begin{example}[Main example]
If $\Gamma$ is a Lie pseudogroup over $\BB$, any almost $\Gamma$-structure of order $k$ (Definition~\ref{def:almost-gamma-structure})
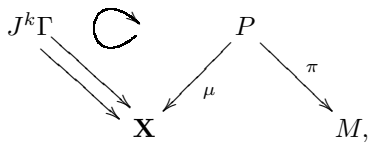
\[
\xymatrix{
J^k\Gamma \ar@<0.25pc>[dr] \ar@<-0.25pc>[dr]  & \ar@(dl, ul) &  P \ar[dl]^{\mu}\ar[dr]^{\pi} &    \\
&\BB  & &  M,}
\]
can be given the structure of full principal $J^k\Gamma$-bundle by taking the restriction to $P$ of the Cartan distribution on $J^k(M\times\BB)$ (see subsection~\ref{app:jets}).
\end{example}

In the next subsection, we explore one further class of examples: principal $(G,\omega,V)$-bundles, where $(G,\omega,V)$ is a Pfaffian group. This is the example most of the literature following Lie's work focused on.

\subsection{Principal Pfaffian bundles -- the group case}\label{subsection_Cartan_bundles}

In this subsection, we will make use of the notation $(G, \h, V)$ for Pfaffian groups (see Proposition~\ref{subs:Pf-groups} for Pfaffian groups). In other words, we will make use of the distribution description of Pfaffian groupoids from Proposition~\ref{prp:dist-picture} and Definition~\ref{def_Pfaffian_groupoid-distribution}, and we will restrict the Pfaffian distribution to the identity, getting a subspace $\h\subset \g$. We recall here that, if $(\G,\omega,E)$ is a Pfaffian groupoid with Pfaffian form $\omega\in \Omega^1(\G,E)$, one can pass to the distribution description by setting $\CC:= \ker(\omega)$. If the Pfaffian groupoid is in fact a Pfaffian group $(G,\omega,V)$, the multiplicativity of $\CC$ implies that $\CC$ can be reconstructed from the subspace $\h\subset \g$. 

Notice that, on Pfaffian groups, the Pfaffian distribution coincide with the symbol bundle (cf.\ Definition~\ref{def_Pfaffian_groupoid-distribution}). Consequently, the subspace $\h\subset \g$ is the restriction of the symbol bundle of the Pfaffian group $(G, \h, V)$ to the identity. From the properties of the symbol bundle it follows that $\h$ is a Lie subalgebra of $\g$ and, by Proposition~\ref{prp:coefficent-space-is-alg}, an ideal of $\g$.

Throughout this subsection, we allow the Pfaffian group $(G, \h, V)$ to be non-full (see Remark~\ref{remark:isotropy-not-full-not-Lie}), but we restrict to full principal $(G, \h, V)$-bundles $(P,\theta)$ -- i.e.\ $\theta\in \Omega^1(P,V)$ is pointwise surjective.

\begin{axiom}{F}\label{axiom_full}
Let $(G, \h, V)$ be a Pfaffian group. From now on, unless otherwise specified, $(G, \h, V)$-bundles $(P,\theta)$ are assumed to be full -- i.e.\ $\theta$ is assumed to be pointwise surjective, see Definition~\ref{def:princ-pfaff-bundles}. 
\end{axiom}

Below, when $p:Y\to M$ is a surjective submersion and $\mathcal{F}$ is a foliation on $Y$, we say that $\mathcal{F}$ is {\bf vertical} if the leaves of $\mathcal{F}$ are tangent to the vertical bundle of $p$.

\begin{proposition}\label{equivalence_Cartan_bundles_foliations}
A principal $(G,\h, V)$-bundle over $M$ is equivalent to the datum of a principal $G$-bundle $P$ over $M$ (action from the left) together with an $G$-equivariant vertical foliation $\F$ on $P$ and an isomorphism $\Phi: \nu_{\F} \overset{\cong}{\to} P\times V$ such that 
\[
\Phi([\alpha^\dagger_p])=(p, l(\alpha)), \quad \alpha\in \g.
\]

Alternatively, a principal $(G,\h, V)$-bundle over $M$ is equivalent to the datum of a principal $G$-bundle $P$ over $M$ together with an equivariant 1-form $\theta\in \Omega^1(P, V)$ such that $\ker(\theta)$ is involutive, is vertical, and 
\[
\theta(\alpha^\dagger_p)=l(\alpha).
\]
\end{proposition}
\begin{proof}
A principal Pfaffian $(G, \h, V)$-bundle is given by a principal $G$-bundle $P$ with a vertical distribution $\CC_P$ such that the action is multiplicative.
Multiplicativity implies the equivariance of $\CC_P$. Involutivity follows combining multiplicativity and verticality, since $\h$ is a subalgebra of $\g$. In other words, the foliation $\F$ from the claim is tangent to $\CC_P$ and coincides with the foliation $\F_\h$ given by the image of the ideal $\h\subset \g$ under the infinitesimal action. The isomorphism $\Phi: \nu_{\F}\overset{\cong}{\to} P\times V$ comes from the isomorphism between the pullback of $V$ to $P$ and $\nu_{\mathcal{F}}$ from Remark \ref{rmk:full_pp_dist_well_given}, which is compatible with the infinitesimal action.

The form description is dual to the foliation one. Notice that $\theta(\alpha^\dagger_p)=l(\alpha)$ is an immediate consequence of multiplicativity.
\end{proof}

\begin{definition}\label{def:Cartan-bundle}
Principal $(G,\h, V)$-bundles are called {\bf Cartan bundles}. 
\end{definition}
When $(P, \theta)$ is a principal $(G,\h, V)$-bundle, we will also write {\bf Cartan $(G,\h, V)$-bundle} to emphasise the Pfaffian group $(G,\h, V)$.

Cartan bundles appeared first in~\cite{FRANCESCOPAPER}. The reason for the name comes from the fact that, after Lie's seminal work~\cite{LieEngel88TrGroupAll}, it was Cartan~\cite{CARTANINFINITEGROUPS} who first studied (transitive!) Lie pseudogroups and their ``actions'' in great generality (i.e.\ allowing for pseudogroups that do not arise from Lie group actions). Moreover, the so called ``Cartan geometries''~ \cite{SHARPE,CAPSLOVAK,ALEKSEEVSKYMICHOR} -- that is, principal bundles possessing a {\it Cartan connection} -- are a prominent example of Cartan bundle (see Example~ \ref{exm:Cart-geo} below).

\begin{example}[G-structures]\label{exm:G-struct}
The case when $\h=0$ (i.e., $\h$ is as small as possible) is already rather interesting. Let $\pi:(P,\theta)\to M$ be a full principal $(G,0,V)$-bundle over $M$. Since $\theta\in \Omega^1(P,V)$ is pointwise surjective, we get $\dim(V) = \dim(M)$. Due to multiplicativity of the $G$-action, $\theta$ is automatically $G$-equivariant. Moreover, since $\h=0$, $\ker(\theta) = \ker(d\pi)$. This is the setting of what are sometimes called {\bf abstract $G$-structures}. 

Indeed, recall that, when $G\subset \GL(n,\mathbb{R})$, a {\bf classical $G$-structure} is a $G$-reduction of the frame bundle $\Fr(M)\to M$ (which is canonically a principal $\GL(n,\mathbb{R})$ bundle over $M$). Recall also that any classical $G$-structure is equipped with the {\bf tautological (or soldering) form} $\theta_{\rm taut}\in \Omega^1(P,\mathbb{R}^n)$ defined by:
\[
(\theta_{\rm taut})_p: v_p \mapsto p\cdot d\pi(v),\quad v_p\in T_pP,
\]
where $p\in P\subset \Fr(M)$ is a frame at $\pi(p)\in M$, i.e.\ a linear isomorphism $p:T_{\pi(p)}M\to \mathbb{R}^n$. See e.g.~\cite{STERNBERG} for more details. One readily checks that $\theta_{\rm taut}$ is pointwise surjective, equivariant, and $\ker(\theta_{\rm taut}) = \ker(d\pi)$. The tautological form is the fundamental object underlying the geometry of $G$-structures. An abstract $G$-structure $(P,\theta)$ consists precisely in a principal bundle $P$ equipped with an $\mathbb{R}^n$-valued differential 1-form $\theta$ having the same key properties of the tautological form.
\end{example}

\begin{example}[Cartan geometries]\label{exm:Cart-geo}
Consider now the other extreme case, $\h=\g$ (i.e., $\h$ is as large as possible). Let $\pi:(P,\theta)\to M$ be a full principal $(G,\g,V)$-bundle over $M$. Due to $\theta\in \Omega^1(P,V)$ being pointwise surjective, we get $\dim(V) = \dim(P)$. Due to multiplicativity of the $G$-action, $\theta$ is automatically $G$-equivariant. Moreover, since $\h=\g$, $\ker(\theta) = 0$. In other words, $\theta$ defines a $V$-valued equivariant absolute parallelism 
\[
\theta: TP \to P\times V.
\]
In the literature, such parallelisms are called {\bf Cartan connections} and have been the subject of much attention, e.g.~\cite{SHARPE,CAPSLOVAK,ALEKSEEVSKYMICHOR}. The pair $(P,\theta)$ is also called {\bf Cartan geometry}\footnote{We warn the reader that, in the literature on Cartan geometries, $V$ is also assumed to be a Lie algebra $\z$ extending $\g$. However, the Lie bracket on $V$ plays a role only when introducing the {\it curvature} of $\theta$; a treatment of this issues goes beyond the scope of this paper, is dealt in full generality of Cartan bundle in \cite{AccorneroCattafi2}.}.
\end{example}

As the examples above show, already in these apparently ``trivial'' cases, Cartan bundles correspond to objects that have been extensively studied in the literature (and the same holds when considering ``intermediate'' $\h$). This is not a coincidence. Most of the literature on pseudogroups that stemmed from Lie's seminal work~\cite{LieEngel88TrGroupAll} and Cartan's fundamental subsequent developments~\cite{CARTANINFINITEGROUPS} takes transitivity (Definition~\ref{def:transitive-orbits}) as part of the definition of Lie pseudogroup. However, when $\Gamma$ is a transitive Lie pseudogroup, almost $\Gamma$-structures can be equivalently presented as Cartan bundles. Intuitively, such equivalence is due to the fact that, due to transitivity, no information is lost when ``passing to the isotropy''. In Section~\ref{sec:PME}, we rigorously prove a categorical equivalence between almost $\Gamma$-structures of order $k$ and Cartan bundles, as an application of the notion of Pfaffian Morita equivalence. A similar conclusion was reached in~\cite{FRANCESCOPAPER} using a different approach.
\begin{remark}
Cartan bundles exhibit a very rich geometry and provide a nice and powerful framework to encode and study transitive geometries, to be compared e.g.\ with~\cite{SingerSternberg65,GUILLEMINSTERNBERG}. The approach from Section~\ref{sec:PME} down below -- where one sees Cartan bundles as objects arising from principal Pfaffian bundles of transitive Pfaffian groupoids -- providea several insights that are not available otherwise. We undertook an extensive treatment of these matters in~\cite{AccorneroCattafi2}.
\end{remark}
\section{Pfaffian Morita equivalences}\label{sec:PME}

\subsection{Principal Pfaffian bibundles and relation with the gauge construction}

Let $\Sigma\tto \BB$ be a Lie groupoid. Recall that, when $\mu:P\to \BB$ is a left $\Sigma$-space (Definition \ref{def_groupoid_action}), we can form the {\bf action groupoid} $\Sigma\ltimes P\tto P$. The arrow space is given by $\Sigma\ltimes P=\Sigma\tensor[_s]{\times}{_\mu}P$, the source map is the second projection and the target map sends $(g, p)$ with $\mu(p)=s(g)$ to $g\cdot p$. The multiplicaton is defined by 
\[
(h,  g\cdot p)\cdot (g, p)=(hg, p)
\]
for all $h, g\in \Sigma$, $p\in P$ with $\mu(p)=s(g)$, $t(g)=s(h)$. Similarly, when $\mu:P\to \BB$ is a right $\Sigma$-space, one can define the action groupoid $P\rtimes \Sigma \tto P$.

\begin{definition}\label{def:princ-Pfaff-bibundle}
Let $(\Sigma_1, \omega_1, E_1)$ and $(\Sigma_2, \omega_2, E_2)$ be Pfaffian groupoids over $\BB_1$ and $\BB_2$ respectively. A {\bf principal Pfaffian bibundle} between them is a triple $(P, \theta, \Phi)$ where
\begin{itemize}
\item $(P, \theta)$ is a left principal $(\Sigma_1, \omega_1, E_1)$-bundle;
\item $\Phi:\mu_1^*E_1\to \mu_2^*E_2$ is an isomorphism of vector bundles;
\end{itemize} 
such that 
\begin{itemize}
\item $(P, \Phi\circ \theta)$ is a right principal $(\Sigma_2, \omega_2, E_2)$-bundle;
\item the two actions on $P$ commute;
\item $\Phi:\mu_1^*E_1\to \mu_2^*E_2$ is a morphism of
\begin{itemize} 
\item left $\Sigma_1\ltimes P$-representations, where
\[
\Sigma_1\ltimes P \times_P \mu_1^*E_1 \to \mu_1^*E_1, \quad \left ((g_1, p), (p, \alpha_{\mu_1(p)})\right)\mapsto \left(g_1\cdot p, g_1\cdot \alpha_{\mu_1(p)}\right)
\]
is the representation naturally induced on $\mu_1^*E_1$, while $\mu_2^*E_2$ carries the trivial representation
\[
\Sigma_1\ltimes P \times_P \mu_2^*E_2 \to \mu_2^*E_2, \quad \left( (g_1, p), (p, \alpha_{\mu_2(p)})\right)\mapsto \left(g_1\cdot p, \alpha_{\mu_2(p)}\right);
\]
\item right $P\rtimes \Sigma_2$-representations, where
\[
P\rtimes \Sigma_2 \times_P \mu_2^*E_2 \to \mu_2^*E_2, \quad \left ((p, g_2), (p, \alpha_{\mu_2(p)})\right)\mapsto \left(p\cdot g_2, g^{-1}_2\cdot \alpha_{\mu_2(p)}\right)
\]
is the representation naturally induced on $\mu_2^*E_2$, while $\mu_1^*E_1$ carries the trivial representation
\[
 P\rtimes \Sigma_2 \times_P \mu_1^*E_1 \to \mu_1^*E_1, \quad \left( (p, g_2), (p, \alpha_{\mu_2(p)})\right)\mapsto \left( p\cdot g_2, \alpha_{\mu_1(p)}\right).
\]
\end{itemize}
\end{itemize}
The bibundle is called {\bf full} if $\theta$ is pointwise surjective. 
\end{definition}
We denote principal Pfaffian bibundles using the ``decorated'' butterfly
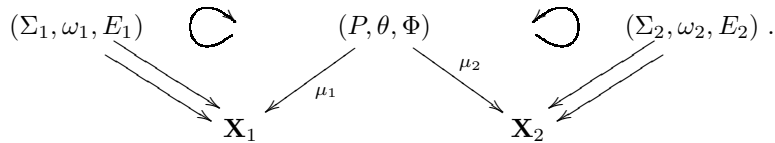
\[
\xymatrix{
(\G_1, \omega_1, E_1) \ar@<0.25pc>[dr] \ar@<-0.25pc>[dr]  & \ar@(dl, ul) &  (P, \theta, \Phi) \ar[dl]^{\mu_1}\ar[dr]^{\mu_2} & \ar@(dr, ur) & (\G_2, \omega_2, E_2) \ar@<0.25pc>[dl] \ar@<-0.25pc>[dl] \\
&\BB_1  & &  \BB_2 &}.
\]
\begin{remark}\label{rmk:Phi-explicit-prop}
The fact that the isomorphism $\Phi:\mu_1^*E_1\to \mu_2^*E_2$ is also a morphism of representations is equivalently encoded into the formulae
\[
\Phi_{p\cdot g_2}(p\cdot g_2, \alpha_{\mu_1(p\cdot g_2)})=g_2^{-1}\cdot \Phi_{p}(p, \alpha_{\mu_1(p)}), \quad (p, g_2)\in  P\rtimes \G_2,\ (p, \alpha_{\mu_1(p)})\in \mu_1^*E_1
\]
and
\[
\Phi_{p}(p, \alpha_{\mu_1(p)})=\Phi_{g_1\cdot p}(g_1\cdot p, g_1\cdot \alpha_{\mu_1(p)}), \quad (g_1, p)\in \G_1\ltimes P,\ (p, \alpha_{\mu_1(p)})\in \mu_1^*E_1. \qedhere
\]
\end{remark}
\begin{remark}
Notice that, in a principal Pfaffian bibundle, $\ker(\theta)\cap \ker(d\mu_1) = \ker(\theta)\cap \ker(d\mu_2)$.
\end{remark}
\begin{remark}\label{remark:mullt-from-right}
Notice that in a principal Pfaffian bibundle both actions are multiplicative.
Recall that the multiplicativity equation for the left action 
\[
(m_{\G_1}^*\theta)_{(g_1, p)}=(\pr_1^*\omega_1)_{(g_1, p)}+g_1\cdot (\pr_2^*\theta)_{(g_1, p)}
\]
is an equality of forms on $\G_1\ltimes P$.


As for the right action, the multiplicativity equation becomes the following equality of forms on $P\rtimes \G_2$:
\[
m^*_{\G_2}(\Phi\circ \theta)_{(p, g_2)}=g_2^{-1}\cdot (\pr_1^*(\Phi\circ \theta))_{(p, g_2)}+g_2^{-1}(\pr_2^* \omega_2 )_{(p, g_2)}
\]

This could look a bit puzzling. The reader is encouraged to look at the case of $(\G, \omega)$ acting on itself from the left and from the right and, in particular, at the multiplicativity equation satisfied by $(i^*\omega)_g=-g^{-1}\cdot \omega$. In such example, $\Phi$ is defined as $\Phi: v_{t(g)}\mapsto g^{-1}\cdot v_{t(g)}$.  
\end{remark}
\begin{remark}\label{remark:distribution_POV}
The definition of principal Pfaffian bibundle can of course be reformulated using distributions instead of vector-valued 1-forms. There, the multiplicativity of the two actions is encoded by $\CC_1\cdot \CC_P\subset \CC_P$ and $\CC_P\cdot \CC_{2}\subset \CC_P$ (compare with Remark~\ref{remark:mullt-from-right} above). 
\end{remark}
\begin{remark}
It can be illuminating to recast the definition of principal Pfaffian bibundle as follows. It is:
\begin{itemize}
\item a principal $(\Sigma_1, \omega_1, E_1)$-bundle $(P,\theta)$;
\item decorated by an isomorphism $\Phi:\mu_1^*E_1\to \mu^*_2E_2$ that makes $(P, \theta)$ into a principal $(\Sigma_2, \omega_2, E_2)$-bundle.
\end{itemize}
The resulting two principal Pfaffian bundles satisfies two additional compatibility properties:
\begin{itemize}
\item the actions on $P$ commute;
\item $\Phi$ is a representation map.
\end{itemize}  
The first property has nothing to do with Pfaffian groupoids, but it ensures that a principal Pfaffian bibundle is a bibundle in the usual sense, see Definition~\ref{def:classical-ME}. The second property should be thought of as a Pfaffian analogue of the first one (and it indeed plays a completely analogous role in later statements, see especially Proposition~\ref{prop:Pfaffian_principal_category}).
\end{remark}
If there is a principal Pfaffian bibundle between the Pfaffian groupoids $(\Sigma_1, \omega_1, E_1)$ and $(\Sigma_2, \omega_2, E_2)$, then $\Sigma_1$ and $\Sigma_2$ are Morita equivalent -- because principal Pfaffian bibundles are in particular bibundles.
\begin{definition}\label{def:Pfaffain-Morita-equivalence}
Two Pfaffian groupoids $(\G_1, \omega_1, E_1)$ and $(\G_2, \omega_2, E_2)$, over $\BB_1$ and $\BB_2$ respectively, are called {\bf Pfaffian Morita equivalent} if there exists a principal Pfaffian bibundle between them. The Morita equivalence is called {\bf full} when the bibundle is full.
\end{definition}
\begin{example}\label{exm:gauge_is_PME}
The gauge groupoid of a Pfaffian groupoid $(\G_1, \omega_1, E_1)$ with respect to a principal Pfaffian bundle $(P,\theta)$ (cf.\ Proposition~\ref{prop:Pfaffian_Gauge_construction}) is Pfaffian Morita equivalent to $(\G_1, \omega_1, E_1)$ by construction. As we will see below (cf.\ Theorem~\ref{prp:PME_is_gauge_construction}), all the examples of Pfaffian Morita equivalences take, up to isomorphism, this form.
\end{example}
The definition of Pfaffian Morita equivalence appeared first in~\cite[Section 5.4]{FRANCESCO}, where the following result is discussed. 
\begin{proposition}\label{prp:PME_is_eq_rel}
Pfaffian Morita equivalence is an equivalence relation.
\end{proposition}
\begin{proof}
The only non-immediate step is to prove transitivity. 
For the reader's convenience, below we sketch the main steps of the proof. Let us first fix notation: we will work with the principal Pfaffian bibundle 
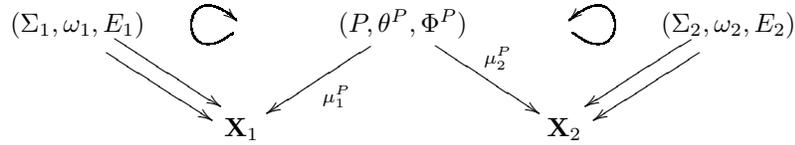
\[
\xymatrix{
(\G_1, \omega_1, E_1) \ar@<0.25pc>[dr] \ar@<-0.25pc>[dr]  & \ar@(dl, ul) &  (P, \theta^P, \Phi^P) \ar[dl]^{\mu^P_1}\ar[dr]^{\mu^P_2} & \ar@(dr, ur) & (\G_2, \omega_2, E_2) \ar@<0.25pc>[dl] \ar@<-0.25pc>[dl] \\
&\BB_1  & &  \BB_2 &}
\]
together with the principal Pfaffian bibundle
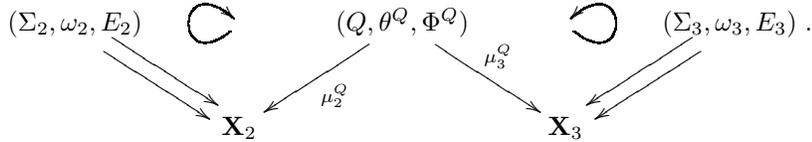
\[
\xymatrix{
(\G_2, \omega_2, E_2) \ar@<0.25pc>[dr] \ar@<-0.25pc>[dr]  & \ar@(dl, ul) &  (Q, \theta^Q, \Phi^Q) \ar[dl]^{\mu^Q_2}\ar[dr]^{\mu^Q_3} & \ar@(dr, ur) & (\G_3, \omega_3, E_3) \ar@<0.25pc>[dl] \ar@<-0.25pc>[dl] \\
&\BB_2  & &  \BB_3 &}.
\]
\begin{itemize}
\item One starts by observing that ordinary Morita equivalence is transitive. In fact, $(P\times_{\BB_2} Q)/\Sigma_2$ is a principal bibundle between $\Sigma_1$ and $\Sigma_3$.
\item The fibred product $P\times_{\BB_2} Q$ carries the form $\hat{\theta} = \theta_P+\theta_Q$. Such form is basic for the (left) action of $\Sigma_2$, and descends to a form $\bar{\theta}$ on the quotient $(P\times_{\BB_2} Q)/\Sigma_2$.
\item There exists an induced isomorphism of vector bundles $\bar{\Phi}$ between the pullbacks of $E_1$ and $E_3$ to $(P\times_{\BB_2} Q)/\Sigma_2$. By direct computation, one checks that it is an isomorphism of representations as in Definition~\ref{def:princ-Pfaff-bibundle}.
\item One shows -- again by direct computation -- that the action of $(\G_1, \omega_1, E_1)$ on $((P\times_{\BB_2} Q)/\Sigma_2,\bar{\theta})$ is a principal Pfaffian bundle; the same holds for the action of $(\G_3, \omega_3, E_3)$ on $((P\times_{\BB_2} Q)/\Sigma_2,\bar{\Phi}\circ\bar{\theta})$.
\end{itemize}
\end{proof}

Just as isomorphic Lie groupoids are Morita equivalent, Pfaffian isomorphic Pfaffian groupoids are Pfaffian Morita equivalent:
\begin{lemma}\label{lemma:morphism_which_are_PME}
Let $(\G_1, \omega_1, E_1)$, $(\G_2, \omega_2, E_2)$ be Pfaffian groupoids over $\BB_1$, $\BB_2$. An isomorphism of Pfaffian groupoids $(\Psi, \Phi)$ between them (Definition~\ref{def:morphism_Pfaffian_groupoids}) 
induces naturally a principal Pfaffian bibundle. If $(\G_1, \omega_1, E_1)$ is full, the Pfaffian bibundle is full as well.
\end{lemma}
\begin{proof}
The bibundle from the claim is $(\G_1, \omega_1, \hat{\Psi})$ with right action of $\G_2$ induced by the isomorphism of Lie groupoids $\Phi$, while the isomorphism $\hat{\Psi}$ is induced by $\Psi$:
\[
\hat{\Psi}: t^*E_1\to s^*\phi^*E_2, \quad (g, e_1) \mapsto (g, (\Phi(g))^{-1}\cdot (\Psi(e_1))).
\]
Here, $\phi:\BB_1\to \BB_2$ is the diffeomorphism of the base manifolds induced by $\Phi$. In other words, one looks at the butterfly diagram
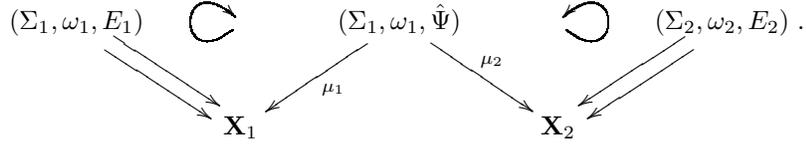
\[
\xymatrix{
(\G_1, \omega_1, E_1) \ar@<0.25pc>[dr] \ar@<-0.25pc>[dr]  & \ar@(dl, ul) &  (\G_1, \omega_1, \hat{\Psi}) \ar[dl]^{\mu_1}\ar[dr]^{\mu_2} & \ar@(dr, ur) & (\G_2, \omega_2, E_2) \ar@<0.25pc>[dl] \ar@<-0.25pc>[dl] \\
&\BB_1  & &  \BB_2 &}.
\]
Careful (but straightforward -- it is enlightening to start with $\Phi=id$ and $(\G_1, \omega_1, E_1)=(\G_2, \omega_2, E_2)$, as in Example \ref{ex:trivial_pfaffian_bundle}) computations show that all the properties of a Pfaffian Morita equivalence are satisfied.
\end{proof}

As we already hinted at in Example~\ref{exm:gauge_is_PME}, and analogously to what happens for ordinary Morita equivalence (Theorem \ref{thm:gauge_eq_morita}), Pfaffian Morita equivalence can be completely encoded in terms of the gauge construction (cf.\ Proposition~\ref{prop:Pfaffian_Gauge_construction}). The following lemma is the first step in that direction.

\begin{lemma}\label{prp:Phi_descends}
Let $(\Sigma_1, \omega_1, E_1)$ and $(\Sigma_2, \omega_2, E_2)$ be Pfaffian groupoids over $\BB_1$ and $\BB_2$, respectively.
\begin{itemize}
 \item Let
\[
\BB_1\overset{\mu_1}{\lto} (P, \theta, \Phi)\overset{\mu_2}{\to} \BB_2
\] 
be the triple realising a Pfaffian Morita equivalence. Then $\Phi$ induces
an isomorphism of $\Sigma_2$-representations 
\[
\bar{\Phi}:\mu_1^*E_1/(\G_1\ltimes P)\overset{\cong}{\to} E_2.
\] 
\item Let $(P, \theta, \Phi)$ be a triple where
\begin{itemize}
\item $(P, \theta)$ is a (left) principal $(\Sigma_1, \omega_1, E_1)$-bundle;
\item $(P, \Phi\circ \theta)$ is a (right) principal $(\Sigma_2, \omega_2, E_2)$-bundle;
\item the two actions on $P$ commute;
\item $\Phi:\mu_1^*E_1\to \mu_2^*E_2$ is the map induced on the pullback bundles by an isomorphism of $\Sigma_2$-representations
\[
\bar{\Phi}:\mu_1^*E_1/(\G_1\ltimes P)\overset{\cong}{\to} E_2.
\]
\end{itemize}
Then $(P, \theta, \Phi)$ is a principal Pfaffian bibundle.
\end{itemize}
\end{lemma}
\begin{proof}
First of all notice that, given a bibundle $P$ between two Lie groupoids $\Sigma_1$ and $\Sigma_2$, the quotient $\mu_1^*E_1/(\G_1\ltimes P)$ is indeed a vector bundle over $\BB_2$, with projection
\[
[(p, \alpha_{\mu_1(p)})]\in \mu_1^*E_1/(\G_1\ltimes P)\mapsto \mu_2(p)
\]
which is well defined because 
\[\mu_2\circ \pr_1: \mu_1^*E_1\to \BB_2
\]
 is invariant under the action of $\Sigma_1\ltimes P$. The {\it right} representation of $P\rtimes \Sigma_2$ on $\mu_1^*E_1$ given by
 \[
P\rtimes \Sigma_2 \times_P \mu_1^*E_1 \to \mu_1^*E_1, \quad \left ((p, g_2), (p, \alpha_{\mu_1(p)})\right)\mapsto \left(p\cdot g_2, \alpha_{\mu_1(p)}\right)
\]
is equivariant with respect to the action of $\Sigma_1\ltimes P$; the same holds for the {\it left} representation obtained acting with inverses, that is
\[
P\rtimes \Sigma_2 \times_P \mu_1^*E_1 \to \mu_1^*E_1, \quad \left ((p, g_2), (p\cdot g_2, \alpha_{\mu_1(p\cdot g_2)})\right)\mapsto \left(p, \alpha_{\mu_1(p\cdot g_2)}\right).
\]
It follows that $\mu_1^*E_1/(\G_1\ltimes P)$ carries a natural left $\Sigma_2$-representation (of course, it also carries a natural right $\Sigma_2$-representation and the left one is obtained from the right one acting with inverses). This uses the fact that $P$ is a bibundle: the compatibility of the $\Sigma_1$-action with the $\Sigma_2$-action is needed.
 
That said, the first part of the claim follows from the properties of $\Phi$ in Definition~\ref{def:Pfaffain-Morita-equivalence} -- the explicit formulae from Remark~\ref{rmk:Phi-explicit-prop} are useful.

Indeed, $\Phi:\mu_1^*E_1\to \mu_2^*E_2$ is a morphism of $\Sigma_1\ltimes P$ representations. This, in particular, implies that $\Phi$ descends to a map 
\[
\bar{\Phi}: \mu_1^*E_1/(\G_1\ltimes P) \to E_2,
\]
because the orbit of the action of $\Sigma_1$ on $P$ are the $\mu$-fibres. Moreover, $\bar{\Phi}$ is an isomorphism of vector bundles: by construction, it is an injective morphism between vector bundles of the same rank. Last, since $\Phi:\mu_1^*E_1\to \mu_2^*E_2$ is a morphism of $ P\rtimes \Sigma_2$ representations, the induced map $\bar{\Phi}$ is a morphism of representations. 

As for the second part of the claim, if 
\[
\bar{\Phi}: \mu_1^*E_1/(\G_1\ltimes P) \to E_2
\]
is an isomorphism of $\Sigma_2$-representations, then there is an induced map
\[
\Phi:\mu_1^*E_1\to \mu_2^*E_2, \quad (p, \alpha_1)\mapsto (p, \bar{\Phi}( [p, \alpha_1])).
\]
This map is a morphism of $\Sigma_1\ltimes P$-representations by definition, and a morphism of (right) $P\rtimes \Sigma_2$-representations because $\bar{\Phi}$ is a morphism of (left) $\Sigma_2$-representations. In fact, if $(p, \alpha_{\mu_1(p)})\in \mu_1^*E_1$, $g_2\in \Sigma_2$, $t(g_2)=\mu_2(p)$ then 
\[
(p, g_2)\cdot (p, \alpha_{\mu_1(p)})=(p\cdot g_2, \alpha_{\mu_1(p)})
\]
\[
\Phi(p\cdot g_2, \alpha_{\mu_1(p)})=(p\cdot g_2, \bar{\Phi}( [p\cdot g_2, \alpha_{\mu_1(p)}]));
\]
we have 
\[
[p\cdot g_2, \alpha_{\mu_1(p)}]=g_2^{-1}\cdot [p, \alpha_{\mu_1(p)}]
\]
hence 
\[
\bar{\Phi}( [p\cdot g_2, \alpha_{\mu_1(p)}])=g_2^{-1}\cdot \bar{\Phi}([p, \alpha_{\mu_1(p)}]).
\]
In the end
\[
\Phi\left((p, g_2)\cdot (p, \alpha_{\mu_1(p)})\right)=\left(p\cdot g_2, g_2^{-1}\cdot \bar{\Phi}([p, \alpha_{\mu_1(p)}])\right)
\]
and the left hand side is exactly the formula for the right representation of $P\rtimes \Sigma_2$ on $\mu_2^*E_2$.
\end{proof}

We arrive finally to the Pfaffian version of Theorem \ref{thm:gauge_eq_morita}.

\begin{theorem}\label{prp:PME_is_gauge_construction}
Let $(\Sigma_1, \omega_1, E_1)$ be a Pfaffian groupoid over $\BB_1$ and $(\Sigma_2, \omega_2, E_2)$ be a Pfaffian groupoid over $\BB_2$. They are Pfaffian Morita equivalent if and only if there exists a principal $(\Sigma_1, \omega_1, E_1)$-bundle $(P,\theta)$ such that 
\[
(\Gaug(P),(\omega_1)_{\rm gauge},(E_1)_{\rm gauge}) \cong (\Sigma_2,\omega_2,E_2)
\]
where the left hand side is the gauge Pfaffian groupoid from Proposition \ref{prop:Pfaffian_Gauge_construction}.
\end{theorem}
\begin{proof}
On the one hand, if the isomorphism in the statement holds, then $(\Sigma_2,\omega_2,E_2)$ is isomorphic to a Pfaffian groupoid which is Pfaffian Morita equivalent to $(\Sigma_1,\omega_1,E_1)$ by construction (see~\cite[Proposition 5.4.3]{FRANCESCO}). Lemma~\ref{lemma:morphism_which_are_PME} and Proposition~\ref{prp:PME_is_eq_rel} imply that $(\Sigma_1,\omega_1,E_1)$ and $(\Sigma_2,\omega_2,E_2)$ are Pfaffian Morita equivalent.

Let us assume that $(\Sigma_1,\omega_1,E_1)$ and $(\Sigma_2,\omega_2,E_2)$ are Pfaffian Morita equivalent, and let us denote $(P,\theta,\Phi)$ a principal Pfaffian bibundle between them. In particular, $P$ is a principal Pfaffian bundle between $\Sigma_1$ and $\Sigma_2$. It follows that $\Gaug(P)\cong \Sigma_2$ (which is the non-Pfaffian incarnation of the present statement). The isomorphism is constructed using the isomorphism of Lie groupoids
\[
\Psi: P\rtimes \Sigma_2 \to P\times_{\mu_1} P,\quad (p,g) \to (p,q):\ q=p\cdot g.
\]
The map $\Psi$ is invariant under the left action of $\Sigma_1\ltimes P$ on $P\rtimes \Sigma_2$ induced by the action of $\Sigma_1$ on $P$ and the left diagonal action of $\Sigma\ltimes P$ on $P\times_{\mu_1}P$. Consequently, it descends to an isomorphism $\Psi_P:\Sigma_2\to \Gaug(P)$. Furthermore, the vector bundle $(E_1)_{\rm gauge}$ of $\Gaug(P)$ (cf.\ Proposition~\ref{prop:Pfaffian_Gauge_construction}) is the bundle $\mu_1^*E_1/(\Sigma_1\ltimes P)$. The representation on it descends to the quotient from the equivariant trivial representation
\[
(P\times_{\mu_1} P) \times_P \mu_1^*E_1 \to \mu_1^*E_1, \quad \left ((p, q), (q, \alpha_{\mu_1(q)})\right)\mapsto \left(p, \alpha_{\mu_1(q)}\right).
\]
As Proposition~\ref{prp:Phi_descends} and its proof clarify, there is an isomorphism of representations $\Psi_{\Phi}: E_2 \to {(E_1)}_{\rm gauge}$. What remains to show is that 
\[
\Psi_\Phi\circ \omega_2 = \Phi_P^*((\omega_1)_{\rm gauge}).
\]
The form $\omega_2$ can be thought of as obtained from the form $\pr_1^*\omega_2$ on $P\rtimes \Sigma_2$, which is basic with respect to the action of $\Sigma_1\ltimes P$ fixing the second factor. On the other hand $(\omega_1)_{\rm gauge}$ is obtained (see~\cite[Proposition 5.4.3]{FRANCESCO}) from the form $\pr_2^*\theta-\pr_1^*\theta$ on $P\times_{\mu_1}P$, which is basic with respect to the diagonal action of $\Sigma_1\ltimes P$. If we take into account that $\Psi_P$ lifts to $\Psi:P\rtimes \Sigma_2 \to P\times_{\mu_1}P$ and $\Psi_\Phi$ lifts to $\Phi^{-1}:\mu_2^*E_2\to \mu_1^*E_1$, we see that it is enough to show that
\[
(\Phi\circ\Psi)^*(\pr_2^*\theta-\pr_1^*\theta) = \pr_2^*\omega_2.
\]
The statement follows recalling the multiplicativity equation for the form $\Phi\circ\theta$, see Remark~\ref{remark:mullt-from-right}, and the fact that the representation of $P\times_{\mu_1}P$ on $\mu_1^*E_1$ is trivial.
\end{proof}

\begin{remark}[A more categorical approach]
Using the alternative characterisation of Morita equivalence via Morita maps (Remark \ref{Morita_maps}), it is possible to prove that two Pfaffian groupoids $(\Sigma_1, \omega_1, E_1)$ and $(\Sigma_2, \omega_2, E_2)$ are Pfaffian Morita equivalent if and only if there exists a third Pfaffian groupoid $(\Sigma, \omega_, E)$ together with two Pfaffian Morita maps $(\Sigma, \omega, E) \to (\Sigma_1, \omega_1, E_1)$ and $(\Sigma, \omega, E) \to (\Sigma_1, \omega_1, E_1)$. Here by {\bf Pfaffian Morita map} we simply mean a Morita map which is also a Pfaffian groupoid morphism in the sense of Definition \ref{def:morphism_Pfaffian_groupoids}.
 
 The equivalence of this notion with Definition \ref{def:Pfaffain-Morita-equivalence} builds on the standard \cite[Theorem 4.3.6]{MATIAS} and is just a matter of unraveling the definitions. It could be also interesting to exploit the alternative characterisation of Morita maps provided in \cite[Theorem 4.3.1]{MATIAS} to obtain a more geometric approach of Pfaffian Morita equivalence.
\end{remark}

\subsection{Some properties of Pfaffian Morita equivalences}
It is well known that various geometric objects associated to a Lie groupoid are preserved under classical Morita equivalence; see Proposition~\ref{prp:Transverse-geometry}. A {\it Pfaffian} Morita equivalence of Pfaffian groupoids between $(\Sigma_1, \omega_1, E_1)$ and $(\Sigma_2, \omega_2, E_2)$ is in particular a Morita equivalence between $\Sigma_1$ and $\Sigma_2$ -- hence, following Proposition~\ref{prp:Transverse-geometry} orbit spaces, isotropy groups and normal representations are preserved. However, there is additional {\it Pfaffian} structure preserved under Pfaffian Morita equivalence.
\begin{proposition}\label{prp:Pfaffian-transverse-geometry}
Let $(\Sigma_1, \omega_1, E_1)$ be a Pfaffian groupoid over $\BB_1$ and $(\Sigma_2, \omega_2, E_2)$ be a Pfaffian groupoid over $\BB_2$. Assume that they are Pfaffian Morita equivalent and let
\[
\BB_1\overset{\mu_1}{\lto} (P, \theta, \Phi)\overset{\mu_2}{\to} \BB_2
\] 
be the triple realising the equivalence. 
Then, if $x\in \BB_1$, $y\in \BB_2$ are points whose orbits are related by $P$ (in the sense of Proposition~\ref{prp:Transverse-geometry}), the following facts hold true.
\begin{enumerate}
\item The isomorphism from Proposition~\ref{prp:Transverse-geometry}
\[
\varphi_p: (\Sigma_1)_x\to (\Sigma_2)_y
\]
induced by the point $p \in P$ sends 
\[
\g_x(\omega_1)=\ker_x(\omega_1)\cap \ker_x(ds)\subset \ker_x(dt)\cap\ker_x(ds)=T_x(\Sigma_1)_x
\]
to 
\[
\g_y(\omega_2)=\ker_y(\omega_2)\cap \ker_y(ds)\subset \ker_y(dt)\cap\ker_y(ds)=T_y(\Sigma_2)_y.
\]
\item The representations of $(\Sigma_1)_x$ on $(E_1)_x$ and of $(\Sigma_2)_y$ on $(E_2)_y$ are isomorphic.
\item The isomorphism $\varphi_p$ between $\Sigma_x$ and $\Sigma_y$ is an isomorphism of Pfaffian groups.
\end{enumerate}
\end{proposition}
\begin{proof}
Let us proceed point by point.
\begin{enumerate}
\item Recall that 
\[
\varphi_p: (\Sigma_1)_x\to (\Sigma_2)_y, \quad g_1\mapsto g_2:=\varphi_p(g_1),
\]
where $g_2\in (\Sigma_2)_x$ is the only element such that $p\cdot g_2=g_1\cdot p$. 
Recall also that the symbol bundle $\ker(ds)\cap \ker(\omega_1)$ is involutive and is contained in $\ker(ds)\cap \ker(dt)$ (Definition~\ref{def_Pfaffian_groupoid-form}). Therefore, it defines a foliation on $(\Sigma_1)_x$. This foliation is both left and right invariant by Remark~\ref{rmk:Pf-form-inv-equiv}/Propositions~\ref{prp:char-pf-groups} and~\ref{prp:isotropies-pf-groups} (or simply by Definition~\ref{def_Pfaffian_groupoid-distribution}). The same considerations hold for $(\Sigma_2)_y$.

Consider the tangent map at $x$ of $\varphi_p$
\[
d_x\varphi_p:T_x (\Sigma_1)_x\to T_y(\Sigma_2)_y;
\] to prove that $d_x\varphi_p$ sends $\g_x(\omega_1)$ to $\g_x(\omega_2)$, it is enough to show that $\varphi_p$ sends leaves to leaves. But this is clear, because $\varphi_p$ is an isomorphism of Lie groups -- hence it sends invariant foliations to invariant foliations.


\item The isomorphism $\Phi: \mu_1^*E_1\to \mu^*_2E_2$ induces an isomorphism of vector spaces $\Phi_p:(E_1)_x\to (E_2)_y$. The properties of $\Phi$ -- see Definition~\ref{def:Pfaffain-Morita-equivalence} and Remark~\ref{rmk:Phi-explicit-prop} -- imply immediately that the pair $(\Phi_p, \varphi_p)$ is an isomorphism of representations.

\item It follows by definition from the previous two points. \qedhere
\end{enumerate}
\end{proof}




At this point we have
\begin{corollary}\label{lemma:Pfaffian_ME_of_groups}
Two Pfaffian groups $(G_1, \omega_1, V_1)$, $(G_2, \omega_2, V_2)$ are Pfaffian Morita equivalent if and only if they are Pfaffian isomorphic.
\end{corollary}
\begin{proof} On the one hand we have simply Lemma~\ref{lemma:morphism_which_are_PME}: a Pfaffian isomorphism naturally induce a Pfaffian Morita equivalence. On the other hand, the third point in Proposition~\ref{prp:Pfaffian-transverse-geometry} implies, in particular, that Pfaffian Morita equivalent Pfaffian groups are Pfaffian isomorphic.
%
\end{proof}

\subsection{Interaction with principal bundles}

A classical fact in Lie groupoid theory is that Morita equivalent Lie groupoids have equivalent categories of principal bundles, see Proposition~\ref{prop:princ_ME}; 
we aim here at a version of that result in the Pfaffian setting. First of all let us notice that the construction of the equivalence in Proposition~\ref{prop:princ_ME} is not so mysterious. Let $P$ be a ($\G_1, \G_2$)-bibundle, with moment maps $\mu_1:\G_1\to \BB_1$ and $\mu_2:\G_2\tto \BB_2$. In order to construct a principal $\G_2$-bundle $P_2$ out of a principal $\G_1$-bundle $P_1$, one considers the pullback $\mu_1^*P_1\cong P_1\times_{\BB_1} P$ carrying the principal diagonal action of $\G_1\ltimes P$ and the action 
\[
((p, g_2), (p_1, p))\in (P\rtimes \Sigma_2)\times (P_1\times_{\BB_1} P)\mapsto (p_1, p\cdot g_2)
\]
of $P\rtimes \G_2$ on the second factor. The quotient by $\G_1\ltimes P$ is then a $\G_2$-space and, in fact, a principal $\G_2$-bundle $P_2$; the action is from the right, but can of course be turned into a left action by precomposing with the inversion. 

\begin{theorem}\label{prop:Pfaffian_principal_category}
Let $(\G_1, \omega_1, E_1)$ and $(\G_2, \omega_2, E_2)$ be Pfaffian groupoids. If they are Pfaffian Morita equivalent, there is a
one to one correspondence
\[ 
 \left\{   \begin{array}{c}
            \text{Principal } (\G_1, \omega_1, E_1) \text{-bundles}\\
             \text{up to isomorphisms}
            \end{array} 
\right\} 
\stackrel{1-1}{\longleftrightarrow}
\left\{   \begin{array}{c}
           \text{Principal } (\G_2, \omega_2, E_2) \text{-bundles}\\
           \text{up to isomorphism}
           \end{array} 
\right\}.
\]\
\end{theorem}

This theorem could be of course restated as an equivalence of appropriate categories; with the same ideas of Remark \ref{rm:stacks}, this would lead to a suitable notion of {\it Pfaffian stack}.

While the stacky point of view to Lie groupoids has many advantages, and therefore could be further explored when bringing Pfaffian structures into the picture, throughout this paper we deliberately focused on the more concrete approach involving bibundles, which is more suitable for the geometric problems motivating us.

\begin{proof}
Let us consider a principal bibundle $\BB_1\lto P\to \BB_2$ between $\G_1$ and $\G_2$ supporting a Pfaffian Morita equivalence:
\[
\xymatrix{
(\G_1, \omega_1, E_1) \ar@<0.25pc>[dr] \ar@<-0.25pc>[dr]  & \ar@(dl, ul) &  (P, \theta, \Phi) \ar[dl]^{\mu_1}\ar[dr]^{\mu_2} & \ar@(dr, ur) & (\G_2, \omega_2, E_2) \ar@<0.25pc>[dl] \ar@<-0.25pc>[dl] \\
&\BB_1  & &  \BB_2 &}.
\]
By Definition~\ref{def:Pfaffain-Morita-equivalence} we have therefore a form $\theta\in \Omega^1(P, \mu_1^*E)$ and an isomorphism $\Phi:\mu_1^*E_1\to \mu_2^*E_2$.

Let now $P_1$ be a principal $\G_1$-bundle, with moment map $\tau_1:P_1\to \BB_1$.
Recall (see the discussion after Proposition~\ref{prop:princ_ME}) that we can construct a principal $\G_2$-bundle $P_2$ by taking the pullback $\mu_1^*P_1\cong P_1\times_{\BB_1} P$ and then the quotient by the principal action of $\G_1\ltimes P$ induced by the diagonal action of $\Sigma_1$. The action of $P\rtimes \G_2$ induced by $\G_2$ acting on the second factor makes the quotient into a principal $\G_2$-bundle $P_2$. 
We use $\tau_2:P_2\to \BB_2$ to denote the moment map of $P_2$.

What one has to do is constructing a form $\theta_2\in \Omega^1(P_2, \mu_2^*E_2)$ which makes the action of $\G_2$ a principal Pfaffian bundle. We can consider the form $\widehat{\theta}=\pr_2^*\theta-\pr_1^*\theta_1$ on $\mu_1^*P$. Notice that, post-composing with the isomorphism of coefficents induced by $\Phi$, $\pr_2^*\theta$ can be thought of as $\pr_2^*\mu_2^*E_2$-valued; the same is true for $\pr_1^*\theta_1$ since $\Phi$ induces the isomorphism (still denoted by $\Phi$)
\[
\Phi: \pr_1^*\tau_1^*E_1 \to \pr_1^*\mu_2^*E_2,\quad (p, a_1, p_1)\mapsto (p, \Phi(p, a_1), p_1).
\]

In what follows, $m_{P_1}$ denotes the action of $\Sigma_1$ on $P_1$, $m_P^1$ denotes the action of $\Sigma_1$ on $P$ and $m_P^2$ denotes the action of $\Sigma_2$ on $P$.
\begin{enumerate}

\item First, we claim that {\it $\widehat{\theta}$ is $\G_1\ltimes P$-basic}.

\begin{itemize}
\item {\it $\widehat{\theta}$ is horizontal}. To see horizontality, first observe that the space tangent to the orbits is given by vectors of the form
\[
\hat{v}=(dm_{P_1}(v,0), dm_P^1(v,0))
\]
where $v\in \ker(ds)\subset T\G_1$; then $\widehat{\theta}(\hat{v})=0$ follows from multiplicativity of $\theta_1$ and $\theta$. 

\item {\it $\widehat{\theta}$ is equivariant}. Here we need to observe that $\G_1\ltimes P$-equivariance is in principle well defined only in the direction normal to $\G_1\ltimes P$-orbits and with respect to the normal representation of $(\G_1\ltimes P)\ltimes \mu_1^*P$. However, by using the Pfaffian forms, the normal representation can be extended to a representation on the whole $\mu_1^*P$; equivariance holds in this last (stronger!) sense. The action of $(g_1, p)$ on $(v_1, v)\in \mu_1^*P_1$ is given by 
\[
\left( dm_{P_1}(d\tau_1(v_1)^{\omega_1}_g, v_1), dm_P^1(d\mu_1(v)^{\omega^1}_g, v) \right)
\]
where $d\tau_1(v_1)^{\omega_1}_g$ and $d\mu_1(v)^{\omega_1}_g$ denote $s$-lifts at $g$ tangent to $\ker(\omega_1)$. One then sees, again using multiplicativity of $\theta_1$ and $\theta$ with respect to $\omega_1$, that $\widehat{\theta}$ is equivariant. 
\end{itemize}
Since $\widehat{\theta}$ is basic, there is a form $\theta_2$ induced on the quotient $P_2$.

\item {\it $\theta_2$ is multiplicative}. Recall that we turned $P_2$ into a left $\G_2$-space by acting with inverses. The left action of $\Sigma_2$ on $P_2$ is denoted by $m_{P_2}$. If $v_2\in TP_2$, we know that $\theta_2(v_2)=\Phi\circ \widehat{\theta}(\hat{v}_2)$ for any lift $\hat{v}_2$ with respect to the quotient projection. Because of how the action is constructed, the lift of $dm_{P_2}(w_2, v_2)$ at $(p_1, p)$, where $w_2\in T_{g_2}\G_2$ is a compatible vector, is computed as 
\[
d\widetilde{m}(\hat{v_2}, \widehat{di(w_2)})
\]
where $\widehat{di(w_2)}$ is the lift of $di(w_2)$ at $(p\cdot g_2, g^{-1}_2)$ compatible with $\hat{v}_2$ and $\widetilde{m}$ is the right action of $P\rtimes \G_2$ on $\mu_1^*P_1$. With this remark at hand, the multiplicativity equation for $\theta_2$ is obtained from the multiplicativity of $\theta$ with respect to $\omega_2$ and the properties of $\Phi$. In fact
\begin{align*}
(\theta_2)_{\mu_2(p)}(dm_{P_2}(v_2, w_2))=\Phi\circ (\pr_2^*\theta-\pr_1^*\theta_1)_{(p_1, p)}(d\widetilde{m}(\hat{v_2}, \widehat{di(w_2)}))\\
=g_2\cdot (\theta)_{p}(v_P)+(\omega_2)_{g_2}(w_2)-g_2\cdot (\theta_1)_{p_1}(v_1)
\end{align*}
where we write $\hat{v_2}=(v_1, v_p)\in T(P_1\times_{\BB_1} P)$. On the other hand 
\begin{align*}
(\omega_2)_{g_2}(w_2)+g_2\cdot (\theta_2)_{\mu_2(p)}(v_2)=(\omega_2)_{g_2}(w_2)+g_2\cdot \left( \theta_{p}(v_P)-(\theta_1)_{p_1}(v_1) \right).
\end{align*}

\item {\it $\theta_2$ is Pfaffian}. Now one has to show that $\ker(d\tau_2)\cap \ker(\theta_2)=\ker(d\pi_2)\cap \ker(\theta_2)$, where we use $\pi_i:P_i\to M_i$ for the quotient projections. The intersection $\ker(d\tau_2)\cap \ker(\theta_2)$ is computed as the quotient of the intersection of $\ker(\hat{\theta})$ with the vertical bundle of the projection $\mu_1^*P_1\to P$, which is readily seen to be $ \ker(d\tau_1)\cap\ker(\theta_1)$. As for the intersection $\ker(d\mu_2)\cap \ker(\theta_2)$, it is isomorphic to $ \ker(d\tau_1)\cap \ker(\theta_1)$, since the quotient projection $\pi_2$ is induced by the projection $\mu_1^*P_1\to M_1$ and we have
\[
\ker(d\mu_1)\cap\ker(\theta) = \ker(d\mu_2)\cap\ker(\theta),
\]
\[
\ker(d\mu_2)\cap\ker(\theta)\cong  \ker (d\pi_1)\cap\ker(\theta_1),
\]
and
\[
\ker (d\pi_1)\cap\ker(\theta_1)= \ker(d\tau_1)\cap \ker(\theta_1).
\]
\end{enumerate}

Performing the construction backwards, one gets a principal Pfaffian bundle $(P_1,\theta_1')$ which is isomorphic to $P_1$ via a map preserving the form (it is worth noticing here that the base manifolds of $P_1$ and of $P_2$ are diffeomorphic).
\end{proof}

\begin{remark}\label{rmk:pb-correspondence-when-iso}
Let us assume that $(\Sigma_1, \omega_1, E_1)$ and $(\Sigma_2, \omega_2, E_2)$ are Pfaffian isomorphic. The Pfaffian isomorphism 
\[
(\Psi, \Phi): (\Sigma_1, \omega_1, E_1)\to (\Sigma_2, \omega_2, E_2)
\]
induces a one to one correspondence between $(\Sigma_1, \omega_1, E_1)$-bundles and principal $(\Sigma_2, \omega_2, E_2)$-bundles. If $(P_1, \theta_1)$ is a principal $(\Sigma_1, \omega_1, E_1)$-bundle, then $(P_1, \Psi\circ \theta_1)$ is a principal $(\Sigma_2, \omega_2, E_2)$-bundle -- where $\Sigma_2$ acts on $P_1$  through $\Phi$ in the obvious sense -- and this clearly defines a one to one correspondence. 
Using the Morita equivalence from Lemma~\ref{lemma:morphism_which_are_PME}, we see that the principal $(\Sigma_2, \omega_2, E_2)$-bundle constructed out of a principal $(\Sigma_1, \omega_1, E_1)$-bundle by following the lines of the proof of Proposition~\ref{prop:Pfaffian_principal_category} is actually isomorphic to $(P_1, \Psi\circ \theta_1)$.
\end{remark}

\subsection{Morita equivalences and transitive groupoids}

In this subsection, we focus on the case when the Pfaffian groupoid $(\G, \omega, E)$ is in particular a transitive Lie groupoid, i.e.\
\[
(s, t): \Sigma\to \BB\times \BB\quad g\mapsto (s(g), t(g))
\]
is a surjective submersion (Definition \ref{def:trans_gpds}). When the arrow space $\G$ is second countable, if $(s, t)$ is surjective then it is also submersive (Remark \ref{rmk:transitivity_and_second_countability}). For the rest of this chapter, we do not need to work with groupoids of germs anymore. Consequently, we make the following assumption.

\begin{axiom}{H-SC}\label{ax:Haus_and_2nd_count}
From now on, unless otherwise specified, the arrow space $\G$ of any Lie groupoid $\G\tto \BB$ is assumed to be Hausdorff and second countable. 
\end{axiom}

In this setting, $\G$ is transitive if and only if for all $x\in \BB$:
\[
\text{for all }y\in \BB \text{ there is } g \in \G \text{ such that } s(g)=x,\ t(g)=y .
\]

Recall from Example \ref{ex:ME_transitive} that a transitive Lie groupoid $\G\tto \BB$ is Morita equivalent to its isotropy Lie group $\G_x$, for any $x\in \BB$. The bibundle realising the equivalence between $\G$ and the isotropy $\G_x$ at (say) $x\in \BB$ is the $s$-fibre $s^{-1}(x)$ over $x$.

It makes sense to investigate how to make the above statement Pfaffian. The outcome is

\begin{theorem}\label{prop_transitive_Pfaffian_Morita_equivalence}
Any full transitive Pfaffian groupoid $(\G, \omega, E)$ is Pfaffian Morita equivalent to the (not neceesarily full) Pfaffian group $(G, \bar{\omega}_{\rm MC}, V)$, where 
\begin{itemize}
\item $G=\Sigma_x$ is the isotropy group of $\G$ at $x\in \BB$;
\item $V=E_x$;
\item $\bar{\omega}_{MC} = l \circ \omega_{MC}$ is the composition of the Maurer-Cartan form $\omega_{MC}:G\to \mathfrak{g}$ and the linear map $l: \mathfrak{g}\to V$ induced by $\omega_x$ (see Proposition~\ref{prp:char-pf-groups}).
\end{itemize}
\end{theorem}

Notice that the Pfaffian Morita equivalence class of $(G, \omega|_{G}, E_x )$ is independent from the choice of $x\in \BB$, see Proposition~\ref{prp:isotropies-pf-groups} and Corollary~\ref{lemma:Pfaffian_ME_of_groups}.

\begin{proof}
The bibundle needed 
to prove the claim is not hard to guess: one considers the triple $(s^{-1}(x), \omega|_{s^{-1}(x)}, \Phi)$. 
Here, 
\[
\Phi: t^*E\to s^{-1}(x)\times E_x
\]
is given by 
\[
(g, \alpha_{t(g)})\mapsto (g, g^{-1}\cdot \alpha_{t(g)}).
\]
Careful but straightforward computations show that $(s^{-1}(x), \omega|_{s^{-1}(x)})$ is a (right) principal $(G, \omega|_{G}, E_x )$-bundle and that $(s^{-1}(x), \omega|_{s^{-1}(x)}, \Phi)$ satisfies all the defining properties of a principal Pfaffian bibundle. See also~\cite{FRANCESCOPAPER} for the full case. By Lemma~\ref{lemma:Pfaffian_and_MC}, $\omega|_{G}$ coincide with $\bar{\omega}_{\rm MC}$, where the map $l: \g \to E_x$ is a map of representations.
\end{proof}
Recall here Remark~\ref{remark:isotropy-not-full-not-Lie}: $(G,\bar{\omega}_{\rm MC}, V)$ is, in general, not full, 
even when $(\Sigma, \omega, E)$ is so. 
\begin{remark}\label{remark_combination_props}
Let us combine Proposition~\ref{prop:Pfaffian_principal_category} with Theorem~\ref{prop_transitive_Pfaffian_Morita_equivalence}. Given a full transitive Pfaffian groupoid $(\G, \omega, E)$ over $\BB$, a principal $(\G, \omega, E)$-bundle $(P, \theta)$ over $M$ induces a principal $(\G_x, \omega|_{x}\circ \omega_{\rm MC}, E_x)$-bundle $(P_x, \theta_{P_x})$ over $M$ following the construction in the proof of Proposition~\ref{prop:Pfaffian_principal_category}. However, if $\mu$ denotes the moment map of the action of $\Sigma$ on $P$, the pair $(\mu^{-1}(x), \theta|_{\mu^{-1}(x)})$ is also a principal $(\G_x, \omega|_{x}\circ \omega_{\rm MC}, E_x)$-bundle over $M$. 

The bundle $(P_x, \theta_{P_x})$ can also be obtained from $(\mu^{-1}(x), \theta|_{\mu^{-1}(x)})$ by restricting the bibundle $(s^{-1}(x), \omega|_{s^{-1}(x)}, \Phi)$ from above to a bibundle between $(\G_x, \omega|_{x}\circ \omega_{\rm MC}, E_x)$ and $(\G_x, \omega|_{x}\circ \omega_{\rm MC}, E_x)$ itself and following the construction in the proof of Proposition~\ref{prop:Pfaffian_principal_category}.

In conclusion, there is an isomorphism of principal $(\G_x, \omega|_{x}\circ \omega_{\rm MC}, E_x)$-bundles
\[
(P_x, \theta_{P_x})\to (\mu^{-1}(x), \theta|_{\mu^{-1}(x)}).
\]
over the identity $M\to M$.
\end{remark}

\subsection{Morita equivalence and almost $\Gamma$-structures}

Let $\Gamma$ be Lie pseudogroup over $\BB$; then the pair $(J^k\Gamma,\omega^k)$, where $\omega^k$ is the restriction of the Cartan form on $J^k(\BB\times \BB)$, is a full Pfaffian groupoid for all $k\geq k_\Gamma$ (Example \ref{exm:jet-groupoids}). If we assume $\Gamma$ to be transitive (Definition~\ref{def:transitive-orbits}), the jet groupoids $J^k\Gamma\tto \BB$ are transitive as well (Remark~\ref{rmk:transitivity_of_pseudo}). We assume for simplicity $k_\Gamma = 0$ (Definition~\ref{def:Lie_pseudo}).

\begin{proposition}\label{prop:group_tower}
The following facts hold true. 

\begin{enumerate}
\item The isotropy Pfaffian group $(G^k,l^k,V^{k-1})$ at $x\in \BB$ of $(J^k\Gamma,\omega^k)$ from Theorem \ref{prop_transitive_Pfaffian_Morita_equivalence} is given by
\begin{itemize}
\item $V^{k-1}\cong A^{k-1}_x$, where we use the notation $A^k:={\rm Lie}(J^k\Gamma)$;
\item $l^k: G^k \to V^{k-1}$ is the only map making the diagram
\[
\begin{tikzcd}
G^k\arrow[r,"\pr|_{G^k}"]\arrow[dr,"l^k"]\arrow[d,"\omega_{\rm MC}^k"] & G^{k-1}\arrow[d,"\omega_{\rm MC}^{k-1}"]\\
\g^k\arrow[r,"d\pr|_{\g^{k}}"] & \g^{k-1} \subset V^{k-1}
\end{tikzcd}
\]
commutative (here, $\pr:J^{k}\Gamma\to J^{k-1}\Gamma$ is the canonical projection, and we are using the fact that $V^{k-1}\cong A^{k-1}_x$).
\end{itemize}
\item The canonical projection $\pr:J^{k+1}\Gamma\to J^{k}\Gamma$ induces a morphism of Pfaffian groups
\[
(\Phi,\Psi):(G^{k+1},l^{k+1},V^{k})\to (G^{k},l^{k},V^{k-1}) 
\]
such that:
\begin{itemize}
\item $\Phi:G^{k+1}\to G^{k}$ is a surjective submersion;
\item $\ker(d\Phi) = \ker(l^{k+1})$;
\item $\Psi|_{\g^k}:\g^k\to V^{k-1}$ coincides with the differential of $\Phi:G^k\to G^{k-1}$ at the identity (in particular, it takes values in $\g^{k-1}$ and is a Lie algebra map).
\end{itemize}
\end{enumerate}
\end{proposition}
\begin{proof}
Point $(i)$ follows by taking into account the properties of the jet tower 
\[
\dots \to (J^k\Gamma,\omega^k) \to (J^{k-1}\Gamma, \omega^{k-1})\to \dots \to \BB\times \BB\tto \BB
\]
discussed in Example~\ref{exm:jet-groupoids}. 

For point $(ii)$, observe that the Morita equivalences between jet groupoids and their isotropies are related and fit into a diagram of the form
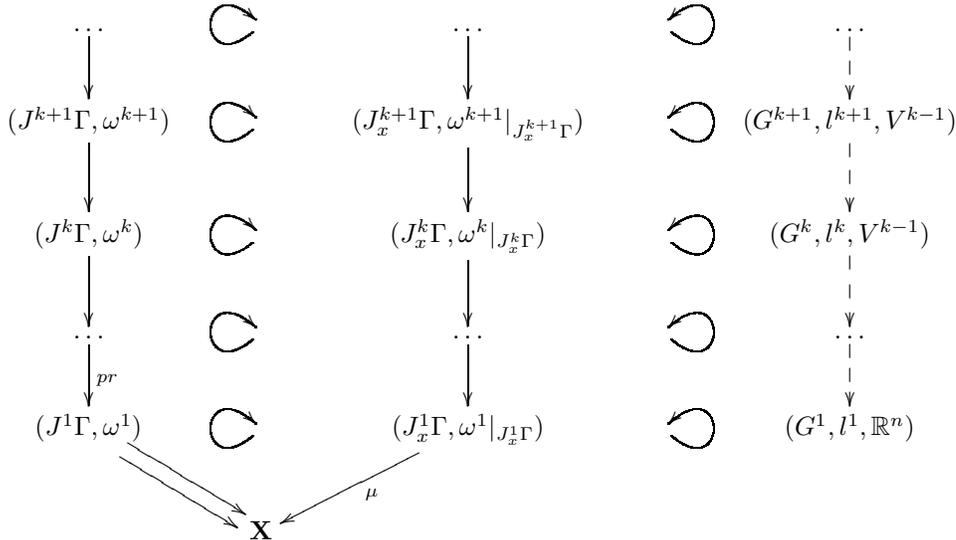
\[
\xymatrix{
\dots\ar[d] & \ar@(dl, ul) &  \dots\ar[d] & \ar@(dr, ur) & \dots \ar@{-->}[d]  \\
(J^{k+1}\Gamma, \omega^{k+1})\ar[d]  & \ar@(dl, ul) &  (J^{k+1}_x\Gamma,\omega^{k+1}|_{J^{k+1}_x\Gamma})\ar[d] & \ar@(dr, ur) & (G^{k+1}, l^{k+1}, V^{k-1})\ar@{-->}[d]
\\
(J^k\Gamma, \omega^k)\ar[d] & \ar@(dl, ul) &  (J^k_x\Gamma,\omega^k|_{J^k_x\Gamma})\ar[d] & \ar@(dr, ur) & (G^k, l^k, V^{k-1})\ar@{-->}[d]  \\
\dots\ar[d]^{pr} & \ar@(dl, ul) &  \dots\ar[d] & \ar@(dr, ur) & \dots\ar@{-->}[d]  \\
(J^1\Gamma, \omega^1) \ar@<0.25pc>[dr] \ar@<-0.25pc>[dr]  & \ar@(dl, ul) &  (J^1_x\Gamma,\omega^1|_{J^1_x\Gamma}) \ar[dl]^{\mu} & \ar@(dr, ur) & (G^1, l^1, \mathbb{R}^n)  \\
&\BB  & &  &}
\]
where $n = \dim(X)$, the vertical arrows are the canonical projections, and we have to show that the dashed arrows exist.

Notice first that
\[
\pr^*(\omega^{k}|_{J^k_x\Gamma}) = d\pr \circ (\omega^{k+1}|_{J^{k+1}_x\Gamma}).
\]
In simpler terms, the jet projections induce morphisms of principal Pfaffian bundles in a generalised sense: the Pfaffian groupoid itself is allowed to vary, cf.\ Definition~\ref{def:princ-pfaff-bundles-morph}.
The dashed arrows are then constructed by making use of Theorem~\ref{prp:PME_is_gauge_construction}, and the properties listed in the claim follow from the analogue properties of the corresponding maps of jet groupoids, cf.\ Example~\ref{exm:jet-groupoids}.
\end{proof}
\begin{remark}
The tower of Pfaffian groups
\[
\dots \to (G^{k+1},l^{k+1},V^{k})\to (G^k,l^k,V^{k-1})\to \dots \to (G^1,l^1,\mathbb{R}^n)
\]
coming from Proposition \ref{prop:group_tower} can be equivalently encoded infinitesimally.
A complete discussion of the infinitesimal objects associated to Pfaffian group(oid)s, namely Spencer operators and Pfaffian algebroids, is beyond the scope of our work; see~\cite{MARIA}. However, by making use of the properties of the jet projections, see Example~\ref{exm:jet-groupoids}, and the properties listed in Proposition~\ref{prop:group_tower} above, we can observe what follows. Let us set
\[
\m^{k} := V^{k}/V^{k-1}\cong \ker (d_e l^{k+1}),\quad k\in \mathbb{N}.
\]
The direct sum
\[
\m_* := \bigoplus\limits_{k\geq 1} \m^k
\]
is graded Lie algebra. The Lie bracket on $\m_*$ extends to a graded bracket on 
\[
\m = \bigoplus\limits_{k\in \mathbb{N}} \m^k,
\]
due to the fact that each Lie algebra $\g^{k+1} = {\rm Lie}(G^{k+1})$ comes with a representation on $V^{k}$, and that such representation is compatible with the adjoint representation, in the sense that the following diagram
\[
\begin{tikzcd}
\g^{k+1}\arrow[r]\arrow[d,"d_e l^k"]
	& \GL(V^{k})\supset \GL(\g^k) \\
\g^k \arrow[ur,"{\rm ad}",swap]
	&
\end{tikzcd}
\]
commutes\footnote{This is a consequence of Corollary~\ref{cor:Pfaffian-groups-MC}, i.e.\ that Pfaffian forms on Lie groups are necessarily given by the Maurer-Cartan form post-composed with a projection}. It is possible to show that the bracket on $\m$ satisfies the Jacobi identity as well. The proof requires to look at the projective limit of the tower of jet groupoids; see for instance~\cite{AccCr2020}.

The algebra $\m$ was studied by Guillemin and Sternberg~\cite{GUILLEMINSTERNBERG} and Singer and Sternberg~\cite{SingerSternberg65}\footnote{In~\cite{SingerSternberg65} the authors adopt the framework of Lie algebra sheaves, the infinitesimal incarnation of pseudogroups.} as a fundamental abstract tool to discuss Cartan's equivalence problem for geometries arising from pseudogroups~\cite{CARTANINFINITEGROUPS}.

The conceptual reason behind the importance of $\m$ is clarified by Proposition~\ref{prop:group_tower}: the tower of isotropy Pfaffian groups is Morita equivalent to the tower of jet groupoids of $\Gamma$.
\end{remark}

The following corollary -- an immediate application of Proposition~\ref{prop:Pfaffian_principal_category} -- leads us back to Cartan bundles~\ref{subsection_Cartan_bundles} (recall that we imposed Axiom \ref{axiom_full}).

\begin{corollary}\label{cor:corresp_cartan_bundles}
There is a one to one correspondence between almost $\Gamma$-structures of order $k$ (Definition~\ref{def:almost-gamma-structure}) and principal Cartan $(G^k,l^k,V^{k-1})$-bundle, where $(G^k,l^k,V^{k-1})$ is the Pfaffian isotropy group of $(J^k\Gamma, \omega^k)$.
\end{corollary}

\begin{example}
Let $G\subset \GL(n,\mathbb{R})$ be a group of matrices. The set
\[
\Gamma_G := \{\varphi\in {\rm Diff}_{\rm loc}(\mathbb{R}^n):\ d_x\varphi \in G,\ x\in \mathbb{R}^n\}
\]
is a transitive Lie pseudogroup of order 1. The isotropy group of its first jet groupoid $J^1\Gamma_G\tto \mathbb{R}^n$ is precisely $G$. 
The symbol bundle $\g(\omega^1)$ can be seen to coincide with the bundle $\mathbb{R}^n\times \g$ of isotropy Lie algebras.
Theorem~\ref{prop_transitive_Pfaffian_Morita_equivalence} yields a Pfaffian Morita equivalence between the full Pfaffian groupoid $(J^1 \Gamma_G, \omega^1)$ and the Pfaffian group $(G,0, \RR^n)$. By using Theorem~\ref{prop:Pfaffian_principal_category}, we recover the classical one to one correspondence between almost $\Gamma_G$-structures of order 1 and $G$-structures, a particular instance of the correspondence with Cartan bundles from Corollary~\ref{cor:corresp_cartan_bundles}.
\end{example}

\begin{example}
If we consider the Pfaffian group $(G,\omega_{\rm MC}, \g \times \RR^n)$ instead, we see that principal $(G,\omega_{\rm MC}, \g \times \RR^n)$-bundles correspond to $G$-structures with a compatible principal connection, a particular case of Cartan geometries, cf.\ Example~\ref{exm:Cart-geo}. 
\end{example}

As already anticipated, in \cite{AccorneroCattafi2} we treat these (and more) examples in the full generalities of Cartan bundles.

\appendix

\section{Background material}


\subsection{Lie groupoids and Morita equivalences}\label{app:Lie_gpds}

\subsubsection*{Lie groupoids}

A {\bf groupoid} is a small category where all morphisms are invertible. More explicitly, a {\bf groupoid} $\G$ over $\BB$, denoted by $\G\tto \BB$, consists of a set $\Sigma$ called {\bf arrow space}, together with a set $\BB$ called {\bf unit space}, and equipped with the following {\bf structure maps}:
\begin{itemize}
\item two surjective maps $s:\G\to \BB$ and $t:\G\to \BB$, called respectively the {\bf source} and the {\bf target};
\item an associative ``partial'' {\bf multiplication}
\[
\G \tensor[_s]{\times}{_t} \G \to \G
\]
such that
\begin{itemize}
\item there exists a canonical global {\bf unit} section $u: \BB\to \G$ of $s$ which is a bijection between elements of $\BB$ and units for the multiplication;
\item each element of $\G$ has an inverse.
\end{itemize}
\end{itemize}

A {\bf subgroupoid} of $\G \tto \BB$ is a subcategory $\mathcal{H} \tto \mathbf{Y}$ which is closed under composition and inversion, and therefore is itself a groupoid; it is called {\bf wide} if $\mathbf{Y} = \BB$. A {\bf morphism} of groupoids, simply denoted by $\Phi:\Sigma_1\to \Sigma_2$, is a functor, or, more explicitly, a commutative diagram 
\[
\begin{tikzcd}
\Sigma_1\arrow[d, shift left=-.5ex]\arrow[d, shift right=-.5ex]\arrow[r, "\Phi"]& \Sigma_2\arrow[d, shift left=-.5ex]\arrow[d, shift right=-.5ex]\\
\BB_1\arrow[r, "\phi"]& \BB_2
\end{tikzcd}
\]
respecting the multiplication of arrows; that is, $\Phi(g\cdot h)=\Phi(g)\cdot \Phi(h)$ for all $h, g\in \Sigma_1$ such that $t(h)=s(g)$. 
 
Let $\G$ be a groupoid over $\BB$. The following collection of facts and terminology is used throughout the paper.
\begin{itemize}
\item The {\bf orbit} $\mathcal{O}_x$ of $x\in \BB$ is the set of $y\in \BB$ such that there is some $g\in \G$ with $s(g)=x$, $t(g)=y$. 
\item The intersection $\G_x:=s^{-1}(x)\cap t^{-1}(x)$ possesses a group structure and is called the {\bf isotropy group at $x$}. Isotropy groups at points in the same orbits are isomorphic. 
\item The isotropy group $\G_x$ acts on $s^{-1}(x)$ by right multiplication, the quotient $s^{-1}(x)/\G_x$ is in bijection with the orbits of $x$, and the quotient map can be identified with $t:s^{-1}(x)\to \mathcal{O}_x \subset \BB$. A similar fact hold analogously for the action of $\G_x$ on $t^{-1}(x)$ by left multiplication.
\end{itemize}

A {\bf topological groupoid} is a groupoid $\Sigma\tto \BB$ such that both the arrow space $\Sigma$ and the unit space $\BB$ are topological spaces and all the structure maps (source and target, multiplication, unit section and inversion map) are continuous.

\begin{definition}\label{def:Lie_gpd}
A {\bf Lie groupoid} $\Sigma\tto \BB$ is a groupoid such that 
\begin{itemize}
\item the arrow space is a smooth manifold (possibly non-Hausdorff and non-second countable);
\item $\BB$ is a smooth manifold (Hausdorff, second countable);
\item source and target, multiplication, unit section, and inversion map are smooth;
\item $s$, hence $t$, is a surjective submersion (with Hausdorff fibres).
\end{itemize}
\end{definition}
For details on (Lie) groupoids see e.g.\ \cite[Chapter 1]{MACKENZIE} or \cite[Section 5.1]{MOERDIJK}.
\begin{example}
A Lie group $G$ is a Lie groupoid over the point: $G\tto \{\bullet\}$.
\end{example}
\begin{example}
The main examples that we have in mind are, of course, those associated to pseudogroups $\Gamma$. For any pseudogroup $\Gamma$ over $\BB$, its germ groupoid (cf.\ Theorem~\ref{thm:Haef_corresp}) is a Lie groupoid when equipped with the étale topology (see also Definition~\ref{def:etale_gpd}). Moreover, if $\Gamma$ is a Lie pseudogroup, then $J^k\Gamma\tto \BB$ is a Lie groupoid for all $k\geq k_\Gamma$ (cf.\ Definition~\ref{def:Lie_pseudo}).

Notice that, for any pseudogroup $\Gamma$ over $\BB$, the orbits of the groupoid to $J^k\Gamma$ coincides with the orbits of $\Gamma$ as a pseudogroup (Definition \ref{def:transitive-orbits}). 
\end{example}
\begin{remark}
The fact that in Definition~\ref{def:Lie_gpd} the arrow space is allowed to be non-Hausdorff and non-second countable might look strange at a first glance. We allow this to happen because, if $\Gamma$ is a pseudogroup over $\BB$, then the germ groupoid ${\rm Germ}(\Gamma)\tto \BB$ has arrow space which is rarely Hausdorff and/or second countable (cf.\ Theorem~\ref{thm:Haef_corresp}). However, most of the novel content in this paper is presented under the assumption of Hausdorff and second countable arrow spaces, see Axiom~\ref{ax:Haus_and_2nd_count}. In particular, we stress that in Definition~\ref{def:Lie_pseudo} each $J^k\Gamma$ is understood to be a manifold in the usual sense, i.e.\ Hausdorff and second-countable.
\end{remark}
\begin{definition}\label{def_morphisms_groupoids}
A {\bf morphism of Lie groupoids} $\Sigma_1\tto \BB_1$ and $\Sigma_2\tto \BB_2$ is a morphism of groupoids 
\[
\begin{tikzcd}
\Sigma_1\arrow[d, shift left=-.5ex]\arrow[d, shift right=-.5ex]\arrow[r, "\Phi"]& \Sigma_2\arrow[d, shift left=-.5ex]\arrow[d, shift right=-.5ex]\\
\BB_1\arrow[r, "\phi"]& \BB_2
\end{tikzcd}
\]
where both $\Phi$ and $\phi$ are smooth maps.
\end{definition} 
When $\Sigma\tto \BB$ is a Lie groupoid (see~\cite[Theorem 5.4]{MOERDIJK} for details):
\begin{itemize}
\item $s^{-1}(x)$ is an embedded submanifold of $\Sigma$, for all $x\in \BB$;
\item the isotropy group $\Sigma_x:=s^{-1}(x)\cap t^{-1}(x)$ is an embedded submanifold and a Lie group; 
\item the right action of $\Sigma_x$ on $s^{-1}(x)$ is smooth;
\item the orbit $\mathcal{O}_x$ of $x\in \BB$ is an immersed submanifold of $\BB$;
\item $t: s^{-1}(x)\to \BB$ is a right principal $\Sigma_x$-bundle over $\mathcal{O}_x$ (similarly, replacing $s$ with $t$ -- in that case one gets a left principal bundle).
 \end{itemize}
The following class of groupoids is fundamental when dealing with pseudogroups.
\begin{definition}\label{def:etale_gpd}
A topological groupoid is called an {\bf étale groupoid} when the source map (and consequently the target map) are local homeomorphisms.
\end{definition}
Let $\BB$ be a manifold and let $\Sigma\tto \BB$ be an étale topological groupoid. The arrow space $\Sigma$ is automatically a (possibly non-Hausdorff, non second countable) manifold: an atlas on $\Sigma$ is obtained by pulling back an atlas on the base $\BB$ via the source map. With such smooth structure, $\Sigma\tto \BB$ is a Lie groupoid; in these cases, we will simply use the terminology ``étale groupoids''. Throughout the paper, étale groupoids are denoted by $\mathcal{G}\tto \BB$, while the notation $\G \tto \BB$ is used for arbitrary Lie groupoids.

\begin{definition}\label{def:trans_gpds}
A Lie groupoid $\Sigma\tto \BB$ is called {\bf transitive}\footnote{in \cite[Definition 1.3.2]{MACKENZIE} this condition is also called locally trivial, the name transitive being reserved for what we call weakly transitive.} when the {\bf anchor map}
\[
(s, t): \Sigma\to \BB\times \BB\quad g\mapsto (s(g), t(g))
\]
is a surjective submersion. We call it {\bf weakly transitive} when the anchor is only surjective.
\end{definition} 
Observe that weak transitivity makes sense for any, not necessarily Lie, groupoid $\G\tto \BB$.
Observe also that weak transitivity is equivalent to asking that for each $x\in \BB$, $y\in \BB$ there is some $g\in \Sigma$ such that $s(g)=x$ and $t(g)=y$ -- i.e., $\Sigma$ possesses a single orbit: for each $x\in \BB$, $\mathcal{O}_x=\BB$. If $\Sigma$ is also transitive, $t: s^{-1}(x)\to \BB$ is a principal $\Sigma_x$-bundle. Observe that the smooth structure of the quotient $s^{-1}(x)/\Sigma_x$ -- i.e.\ the smooth structure on $\BB$ {\it as an orbit of $\Sigma$} -- is diffeomorphic to the given smooth structure on $\BB$ thanks to the fact that the anchor map is submersive. In fact, this implies that the restricted target map $t:s^{-1}(x)\to \BB$ is submersive as well.
\begin{remark}[Transitivity and second countability]\label{rmk:transitivity_and_second_countability}
Let $\Sigma\tto \BB$ be a Lie groupoid and let the anchor
\[
(s,t): \Sigma\tto \BB\times \BB
\]
be surjective. It follows that $t|_{s^{-1}(x)}: s^{-1}(x)\to \BB$ is a principal $\Sigma_x$-bundle projection, where the quotient smooth structure on $\BB$ makes it into an {\it immersed submanifold} of $\BB$ with the smooth structure as the base of $\Sigma\tto \BB$ (this can be seen by means of a rather explicit argument, see e.g.\ Theorem $5.4$ in~\cite{MOERDIJK}). 

Assume now that the quotient topology on $\BB$ is strictly finer than the topology as the base of $\Sigma\tto \BB$. Then $\BB$ is locally Euclidean with both topologies; it follows that $\BB$ with the quotient smooth structure is a manifold with strictly lower dimension than $\BB$ with the smooth structure as base of $\Sigma\tto \BB$. This implies that $\BB$ with the quotient topology is not second countable. Consequently, also $s^{-1}(x)$ is not second countable. To sum up we see that, if a Lie groupoid $\Sigma\tto \BB$ has surjective but {\it not submersive} anchor, then its arrow space has to be non-second countable. Hence, if the arrow space $\Sigma$ of a Lie groupoid is second countable, then the anchor map is surjective if and only if it is surjective {\it and submersive}, i.e.\ the groupoid is weakly transitive if and only if it is transitive.
\end{remark}

\begin{remark}[Transitivity of (Lie) pseudogroups]\label{rmk:transitivity_of_pseudo}
Observe that transitivity of a (not necessarily Lie) pseudogroup $\Gamma$ over $\BB$ (see Definition \ref{def:transitive-orbits}) is equivalent to the map
\[
(s, t): J^k\Gamma\to \BB\times \BB\quad j^k_x\varphi \mapsto (x, \varphi(x))
\]
being surjective. 

Consequently, if the (not necessarily Lie) jet groupoid $J^k\Gamma\tto \BB$ of a  pseudogroup $\Gamma$ is weakly transitive, then $\Gamma$ itself is transitive. In fact, for $\Gamma$ to be transitive it is enough for the anchor to be surjective.

Viceversa, if the pseudogroup $\Gamma$ is transitive, then the anchor of $J^k\Gamma\tto \BB$, for all $k\in \mathbb{N}$, is surjective. 

Let now $\Gamma$ be a {\it Lie} pseudogroup, and recall that we always assume the jet groupoids $J^k\Gamma$, $k\in \mathbb{N}$, of Lie pseudogroups to be manifolds in the usual sense, i.e.\ Hausdorff and second-countable. Combining this with the previous remark, we see that for all $k\geq k_\Gamma$ (cf.\ Definition~\ref{def:Lie_pseudo}), $J^k\Gamma\tto \BB$ is transitive as well.

Compare with Remark~\ref{rmk:transitivity_of_germ_groupoid}.
\end{remark} 

\begin{remark}[Transitivity of germ groupoids]\label{rmk:transitivity_of_germ_groupoid}
Let $\Gamma$ be a pseudogroup over $\BB$. Its associated germ groupoid $\Germ(\Gamma)\tto \BB$ (see Theorem~\ref{thm:Haef_corresp}) is weakly transitive if and only if $\Gamma$ is transitive. On the other hand, if $\Gamma$ is transitive then $\Germ(\Gamma)\tto \BB$ is not necessarily transitive (even if $\Gamma$ is assumed to be a Lie pseudogroup!). In fact, if $\Gamma$ is transitive, the anchor map 
\[
\Germ(\Gamma)\to \BB\times \BB, \quad \germ_x(\varphi) \mapsto (x, \varphi(x))
\]
is surely surjective, but it cannot be submersive unless $\dim(\BB)< 1$, because the $s$-fibres of étale groupoids are necessarily discrete.

Compare with Remarks~\ref{rmk:transitivity_and_second_countability} and~\ref{rmk:transitivity_of_pseudo}\footnote{When comparing with Remark~\ref{rmk:transitivity_and_second_countability} notice that when $ \dim(\BB)\geq 1$ then $\BB$ is non-countable. Consequently, if the anchor of $\Sigma\tto \BB$ is surjective, the $s$-fibres of $\Sigma$ (hence $\Sigma$ itself) are not second countable.}
\end{remark}

\begin{definition}
Let $\G\tto \BB$ be a Lie groupoid. A local section $\sigma:U\subset \BB\to \G$ of the source map is called {\bf bisection} if the composition $t\circ \sigma:U\to \BB$ is a diffeomorphism onto its image.
\end{definition}

\begin{definition}\label{def:effective_gpd}
An étale groupoid $\mathcal{G}\tto \BB$ is called {\bf effective} when, for any two local bisections $\sigma$, $\sigma'$, if $t\circ \sigma=t\circ \sigma'$ then $\sigma=\sigma'$.
\end{definition}

\begin{example}\label{effective_etale_pseudogroups}
Effective étale Lie groupoids are precisely germ groupoids of pseudogroups, see Theorem~\ref{thm:Haef_corresp}. Notice that, given any étale groupoid $\mathcal{G}$, the set
\[
\Gamma_{\mathcal{G}}:=\{t\circ \sigma:\ \sigma\ \text{is a local bisection of }\mathcal{G}\}
\]
defines a pseudogroup. The morphism of groupoids
\[
\mathcal{G}\to \Germ(\Gamma_\mathcal{G}), \quad g \mapsto \germ_x(t\circ \sigma_g),
\] 
where $s(g)=x$ and $\sigma_g$ is a bisection defined around $x$ and such that $\sigma_g(x)=g$, is an isomorphism precisely when $\mathcal{G}$ is effective.
\end{example}

\subsubsection*{Lie algebroids}\label{subsec:Lie-alg}

\begin{definition}\label{def:Lie_alg}
A {\bf Lie algebroid} $(A,\rho,[\cdot ,\cdot ])$ over $\BB$ is a vector bundle $A\to \BB$ equipped with
\begin{itemize}
\item a vector bundle map $\rho:A\to TM$, called {\bf anchor};
\item a Lie bracket $[\cdot ,\cdot ]$ on the space of sections $\Gamma(A)$;
\end{itemize}  
such that the Leibniz identity
\[
[\alpha, f\beta]=f[\alpha, \beta]+L_{\rho(\alpha)}(f)\beta
\]
holds for all $\alpha, \beta\in \Gamma(A)$, $f\in C^\infty(\BB)$. 
\end{definition}
When dealing with a Lie algebroid $(A,\rho,[\cdot ,\cdot ])$ over $\BB$, whenever there is no need to specify anchor and bracket, we use the notation $A\to \BB$.


\begin{definition} 
A {\bf Lie subalgebroid} of a Lie algebroid $(A,\rho,[\cdot ,\cdot ])$ over $\BB$ is a Lie algebroid $(A',\rho',[\cdot ,\cdot ]')$ over $\BB$ such that 
\begin{itemize}
\item $A'\to \BB$ is a vector subbundle of $A\to \BB$;
\item $\rho'$ is the restriction of $\rho$ to $A'$;
\item the inclusion of $A'$ into $A$ makes $(\Gamma(A'),[\cdot ,\cdot ]')$ into a Lie subalgebra of $(\Gamma(A),[\cdot ,\cdot ])$.
\end{itemize}
A subalgebroid $A'\to \BB$ of $A\to \BB$ is called {\bf ideal} if $\Gamma(A')$ is a Lie ideal in $\Gamma(A)$. 
\end{definition}

We conclude by recalling that a {\bf morphism of Lie algebroids} over the same base is simply a vector bundle morphism preserving the anchor and the Lie bracket. For the general notion of morphism (between Lie algebroids over different bases), see e.g.\ in \cite[Definition 4.3.1]{MACKENZIE} and \cite[Section 6.2]{MOERDIJK}.

\

The Lie algebroid $A:={\rm Lie}(\Sigma)\to \BB$ of a Lie groupoid $\G\tto \BB$ is constructed as follows:
\begin{itemize}
\item the total space $A$ is given by $\ker(ds)|_\BB$ (recall that we can see $\BB$ as an embedded submanifold of $\Sigma$ using the unit (bi)section of $\Sigma\tto \BB$);
\item the anchor map is the restriction $dt|_A: A \to T\BB$;
\item the bracket on sections of $A$ is induced by the bracket of {\it right invariant vector fields} on $\Sigma\tto \BB$.
\end{itemize}

The construction of the Lie algebroid of a Lie groupoid is reminescent to the construction of the Lie algebra of a Lie group, which is indeed a particular case -- any Lie group $G$ is a Lie groupoid over the point and its Lie algebra $\g$ is a Lie algebroid over the point. Part of the classical Lie theory can be extended to Lie algebroids and Lie groupoids -- e.g.~\cite{MARIUSRUI} and references therein. In particular, we make use of the following: 
\begin{proposition}\label{prp:morph-algebroid}
Let $\Sigma_1\tto \BB_1$ and $\Sigma_2 \tto \BB_2$ be Lie groupoids and $A_1\to \BB_1$, $A_2\to \BB_2$ be their Lie algebroids. Let
\[
\begin{tikzcd}
\Sigma_1\arrow[r, "\Phi"]\arrow[d, shift left = 0.75]\arrow[d, shift right = 0.75] &\Sigma_2 \arrow[d, shift left = 0.75]\arrow[d, shift right = 0.75]\\
\BB_1\arrow[r, "\phi"] & \BB_2
\end{tikzcd}
\]
be a Lie groupoid morphism. The tangent map of $\Phi$ induces a vector bundle map
\[
\begin{tikzcd}
A_1\arrow[r, "{\rm Lie}(\Phi)"]\arrow[d] & A_2 \arrow[d]\\
\BB_1\arrow[r, "\phi"] & \BB_2
\end{tikzcd}
\]
which is a morphism of Lie algebroids.
Furthermore, if $\Sigma_3\tto \BB_3$ is a Lie groupoid and there is a morphism $\Psi$ from $\Sigma_2\tto \BB_2$ to $\Sigma_3\tto \BB_3$ lifting $\psi:\BB_2\to \BB_3$, then 
\[
{\rm Lie}(\Psi\circ \Phi)={\rm Lie}(\Psi)\circ {\rm Lie}(\Phi).
\]
\end{proposition}

Throughout the paper, whenever the context is clear, the notation $A\to \BB$ is used in place of ${\rm Lie}(\Sigma)\to \BB$ to denote the Lie algebroid of a groupoid $\Sigma\tto \BB$.

\subsubsection*{Principal bundles}

\begin{definition}\label{def_groupoid_action}
Let $\Sigma\tto \BB$ be a Lie groupoid. A smooth {\bf (left) action} of $\Sigma$ on a manifold $P$ along a map $\mu:P\to \BB$, or, more briefly, a {\bf (left) $\Sigma$-space} $\mu:P\to \BB$, is a smooth map 
\[
m_P: \G\tensor[_s]{\times}{_\mu} P\to P
\]
such that 
\begin{itemize}
\item $m_P(g_1, m_P(g_2, p))=m_P(g_1g_2, p)$, for all $g_1, g_2\in \Sigma$, $s(g_1)=t(g_2)$, $p\in P$, $\mu(p)=s(g_2)$;
\item  $m_P(1_x, p)=p$ for all $p\in P$ such that $\mu(p)=x\in \BB$, where $1_x$ denotes the identity in $\Sigma$ over $x$.
\end{itemize}
We will denote $m_P(g, p)$ by $g\cdot p$. The map $\mu$ is also called the {\bf moment map} of the action.
\end{definition}
The notion of right action/right $\Sigma$-space is defined analogously.

When the moment map $\mu:E\to \BB$ is a vector bundle projection and each arrow $g\in \Sigma$ acts as a linear map (hence an isomorphism) between fibers, then E is called a (left) {\bf representation} of $\Sigma\tto \BB$. 

Given a $\Sigma$-space $\mu:P\to \BB$, the {\bf orbit space} of the action of $\Sigma$ is the quotient $P/\Sigma$ with respect to the equivalence relation
\[
p_1\sim p_2 \quad \text{if and only if}\quad p_2=g\cdot p_1 \text{ for some } g\in \Sigma, s(g)=\mu(p_1).
\]

\begin{definition}\label{def:princ_gpd_bundle}
Let $\Sigma\tto \BB$ be a Lie groupoid and $M$ be a manifold. A {\bf principal $\G$-bundle over $M$} is a $\Sigma$-space $\mu:P\to \BB$ together with a surjective submersion $\pi:P\to M$ such that 
\begin{itemize}
\item $\pi$ is $\G$-invariant: $\pi(g\cdot p)=\pi(p)$ for all $g\in \G$, $p\in P$ such that $s(g)=\mu(p)$;
\item the action of $\G$ is transitive on $\pi$-fibres: if $p_1$, $p_2$ are points in $P$ and $\pi(p_1) = \pi(p_2)$, then $p_2 =g\cdot p_1$ for some $g\in \G$ such that $s(g) = \mu(p_1)$;
\item the action is free: for all $p\in P$, $g\in \Sigma$ such that $s(g)=\mu(p)$, if $g\cdot p=p$ then $g=1_{\mu(p)}$;
\item the action is proper: the map
\[
\Sigma\tensor[_s]{\times}{_\mu}P \to P\times P, \quad (g, p) \mapsto (g\cdot p, p)
\]
is a proper map.
\end{itemize}
\end{definition}
The orbit space of a principal $\G$-bundle over $M$ possesses a canonical manifold structure making the quotient projection into a surjective submersion; the map $P/\G\to M$ induced by $\pi$ is a diffeomorphism when $P/\G$ is given such a smooth structure.

\begin{remark}\label{infinitesimal_action}
Given an action $m_P$ of a Lie groupoid $\Sigma \tto \BB$ on $\mu:P\to \BB$, one has an induced {\bf infinitesimal action}, given by the map
\[
a: t^*A\to \ker(d\pi)\subset TP
\]
sending $\alpha_x\in A_{t(g)}=\ker_{t(g)}(ds)$ to $dm_P(\alpha_x, 0_p)\in \ker_p(d\pi)\subset T_pP$. When $P$ is a principal $\Sigma$-bundle over $M$, the infinitesimal action induces an isomorphism 
\[
\mu^*A \to \ker (d\pi), \quad \alpha_{\mu(p)}\mapsto dm_P(\alpha_{\mu(p)}, 0_p). \qedhere
\]
\end{remark}

\begin{definition}
A {\bf morphism} of principal $\Sigma$-bundles $\mu_1:P_1\to \BB$, $\mu_2: P_2\to \BB$ over the same base $M$ is a smooth map $F:P_1\to P_2$ commuting with the action of $\Sigma$, i.e.\ 
\begin{itemize}
\item $\mu_2\circ F=\mu_1$;
\item $F(g\cdot p)=g\cdot F(p)$ for all $g\in \Sigma$, $p\in P$, $\mu(p)=s(g)$.
\end{itemize}
\end{definition}
Throughout the paper, we work under the following axiom.
\begin{axiom}{MM}\label{axiom_moment_map}
Given a Lie groupoid $\Sigma\tto \BB$ and a $\Sigma$-principal bundle $\mu:P\to \BB$ over $M$, the moment map $\mu$ is always assumed to be a surjective submersion.
\end{axiom}
In particular, if one is dealing with an étale groupoid $\mathcal{G}\tto \BB$ and one has $\dim(\BB) = \dim(M)$, the above axiom is equivalent to the request that the moment map is étale.
\subsubsection*{Morita equivalence and transverse geometry}
The notion of {\it Morita equivalence} between Lie groupoids, which appeared first in \cite[Definition 2.1]{XUMORITAEQ} in the context of Poisson geometry, captures the idea of ``transverse geometry'' of a Lie groupoid $\Sigma\tto \BB$ (see Proposition \ref{prp:Transverse-geometry} below and \cites{JOAO, MATIAS, MOERDIJK} for more details on this point of view).
\begin{definition}\label{def:classical-ME}
A {\bf principal bibundle} between $\G_1\tto \BB_1$ and $\G_2\tto \BB_2$ is a manifold $P$ together with two maps $\mu_1:P\to \BB_1$ and $\mu_2:P\to \BB_2$ such that
\begin{itemize}
\item $\mu_1:P\to \BB_1$ is a left $\G_1$-space and a principal $\Sigma_1$-bundle over $\BB_2$, with projection $\mu_2:P\to \BB_2$;
\item $\mu_2:P\to \BB_2$ is a right $\G_2$-space and a principal $\Sigma_2$-bundle over $\BB_1$, with projection $\mu_1:P\to \BB_1$;
\item the two actions on $P$ commute.
\end{itemize}
Two Lie groupoids $\Sigma_1\tto \BB_1$ and $\Sigma_2\tto \BB_2$ are called {\bf Morita equivalent} if there exists a principal bibundle between them.
\end{definition}
We will denote Morita equivalences/principal bibundles by $\BB_1\overset{\mu_1}{\leftarrow} P \overset{\mu_2}{\rightarrow} \BB_2$ or, more extendedly, via {\bf butterfly diagrams}:
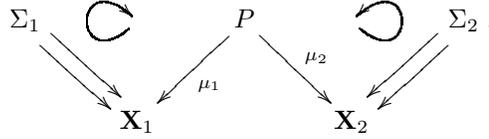
\[
\xymatrix{
\G_1 \ar@<0.25pc>[dr] \ar@<-0.25pc>[dr]  & \ar@(dl, ul) &  P \ar[dl]^{\mu_1}\ar[dr]^{\mu_2} & \ar@(dr, ur) & \G_2 \ar@<0.25pc>[dl] \ar@<-0.25pc>[dl] \\
&\BB_1  & &  \BB_2 &}.
\]
\begin{example}\label{ex:ME_transitive}
A transitive Lie groupoid is Morita equivalent to its isotropy group at any point. In fact, if $\G\tto \BB$ is transitive, then the $s$-fibre $s^{-1}(x)$ at any point is a principal $\G$-bundle over $\{x\}$ (along the restriction of the target map), but also a principal $\G_x$-bundle over $M$ along the (constant) restriction of the source map. See e.g.\ \cite[Proposition 5.14]{MOERDIJK}.
\end{example}

\begin{remark}\label{Morita_maps}
Morita equivalence between Lie groupoids can also be alternatively defined by means of Morita maps \cite{MATIAS}. 

A {\bf Morita map}\footnote{this term is not universal in the literature: it is known e.g.\ as {\it essential equivalence} in \cite{JANEZ} and {\it weak equivalence} in \cite{MOERDIJK}} $\Phi: \G_1 \to \G_2$ covering $\phi: \BB_1 \to \BB_2$ is a Lie groupoid morphism (Definition \ref{def_morphisms_groupoids}) which is also
 \begin{itemize}
  \item {\it fully faithful}: $\G_1$ is a pullback of $(\phi,\phi): \BB_1 \to \BB_2 \times \BB_2$ and $(s,t): \G_2 \to \BB_2 \times \BB_2$ in the smooth category, i.e.\ the following map is a diffeomorphism
  \[
   \G_1 \to (\BB_1 \times \BB_1) \times_{(\phi,\phi), (s,t)} \G_2, \quad g \mapsto (s(g), t(g), \Phi(g));
  \]
  \item {\it essentially surjective}: the smooth map $t \circ \pr_2: \G_1 \times_{\BB_2} \BB_1 \to \BB_2, (g,x) \mapsto t(g)$ is a surjective submersion.
 \end{itemize}
 In particular, a Morita map is fully faithful and essentially surjective as a functor between categories, i.e.\ it induces an equivalence of categories between $\Sigma_1$ and $\Sigma_2$. 

For the equivalence with the bibundle definition see e.g.\ \cite[Theorem 4.3.6]{MATIAS}: in particular, two Lie groupoids $\Sigma_1$ and $\Sigma_2$ are Morita equivalent if and only if there exist a third groupoid $\Sigma$ and two Morita maps $\Sigma \to \Sigma_1$ and $\Sigma \to \Sigma_2$.
\end{remark}

The following well known proposition (e.g.\ \cite[Lemma 1.28]{JOAO}) makes precise the claim that Morita equivalence captures the ``transverse geometry'' of a Lie groupoid. Recall that the {\bf orbit space} of a Lie groupoid $\Sigma\tto \BB$ is the leaf space of the singular foliation on $\BB$ by orbits of $\Sigma$ -- in other words, it is the quotient $\BB/\G$ of $\BB$ with respect to the equivalence relation identifying two points $x, y\in \BB$ when there is an arrow $g\in \Sigma$ such that $s(g)=x$, $t(g)=y$.
\begin{proposition}\label{prp:Transverse-geometry}
Any Morita equivalence $\BB_1\lto P\to \BB_2$ between two Lie groupoids $\G_1\tto \BB_1$ and $\G_2\tto \BB_2$ induces:
\begin{itemize}
\item a homeomorphism between the orbit spaces
\[
\Phi_P: \BB_1/\G_1\overset{\cong}{\to} \BB_2/\G_2; 
\]
\item for each $p\in P$, an isomorphism between the isotropy groups at points $x_1\in \BB_1$ and $x_2\in \BB_2$ such that $\mu_1(p)=x_1$ and $\mu_2(p)=x_2$,
\[
\varphi_p: (\G_1)_{x_1} \overset{\cong}{\to} (\G_2)_{x_2},\quad g_1\mapsto g_2:=\varphi_p(g_1),
\]
where $g_2$ is the (unique) element of $(\Sigma)_{x_2}$ such that $g_1\cdot p=p\cdot g_2$;
\item an isomorphism of the normal representations at points whose orbits are related by $\Phi_P$ (see below).
\end{itemize}
\end{proposition}
Let us comment on the last point. A representation of a groupoid $\G\tto \BB$ is a vector bundle $E\to \BB$ together with an action of $\G$ along $E\to \BB$ itself, such that each arrow acts by a linear isomorphism. If $\mathcal{O}_x$ is the orbit of $x \in \BB$, the restriction $\G|_{\mathcal{O}_x} \tto \mathcal{O}_x$ is a groupoid admitting a representation on the normal bundle $\nu_{\mathcal{O}_x}$ to the orbit. An arrow $g\in \G|_{\mathcal{O}_x}$, with $s(g)=y\in \mathcal{O}_x$, acts by
\[
g\cdot [v_y]:=[dt (v^s_g)],
\]
where $v^s_g \in T_g\G$ is any $s$-lift of $v_y \in T_y\BB$. This representation of $\Sigma|_{\mathcal{O}_x}$ on $\nu_{\mathcal{O}_x}$ is called {\bf normal representation} of $\Sigma$ at $x$.
By restricting it, one gets a representation of the isotropy group $\G_x$ on the normal space $\nu_x := (\nu_{\mathcal{O}_x})_x$; this latter representation is called {\bf normal representation} as well.

\begin{remark}[Normal representation of \`etale groupoids]
Observe that, when $\mathcal{G}\tto \BB$ is an étale Lie groupoid, 
there is representation of $\mathcal{G}$ on $T\BB$, because $s$ and $t$ are local diffeomorphisms. In fact, one defines
\[
g\cdot v_y:= dt (v^s_g), \quad g\in \mathcal{G},\ s(g)=y,\ v_y\in T_y\BB
\]
where now $v^s_g$ is {\it the} $s$-lift of $v_y$ at $g$ (which exists and is unique because $s$ is a local diffeomorphism).

Such a representation of $\mathcal{G}$ on $T\BB$ can be used to reconstruct the normal representation of $\mathcal{G}|_{\mathcal{O}_x}$ on $\nu_{\mathcal{O}_x}$ for all $x\in \BB$. Given $x\in \BB$, the ``étale'' representation above induces a representation of $\mathcal{G}|_{\mathcal{O}_x}$ on $i^*T\BB$, where $i:\mathcal{O}_x\hookrightarrow \BB$ is the inclusion. Let now $[v_y]\in \nu_{\mathcal{O}_x}$, $y\in \mathcal{O}_x$, $v_y\in T_y\BB$, and $v_y+w_y\in T_y\BB$, $w_y\in T_y\mathcal{O}_x\hookrightarrow T_y\BB$ be any representative of $[v_y]\in \nu_{\mathcal{O}_x}$. It is straightforward to check that, for all $g\in \mathcal{G}$ with $s(g)=y$, 
\[
g\cdot w_y \in T_{t(g)}\mathcal{O}_x.
\]
Consequently the quotient projection $i^*T\BB\to \nu_{\mathcal{O}_x}$ is a map of $\mathcal{G}|_{\mathcal{O}_x}$-representations.
\end{remark}

It turns out that any principal $\G$-bundle $P\to M$ can be ``completed'' and made into a bibundle, in the sense explained below. Let $\G\tto \BB$ be a Lie groupoid and $\pi: P\to M$ be a principal $\G$-bundle along $\mu:P\to \BB$, recalling that we always assume its moment map to be surjective and submersive (Axiom \ref{axiom_moment_map}). The {\bf gauge groupoid} $\Gaug(P)$ of $P$ is the Lie groupoid such that
\begin{itemize}
\item the space of arrows is the quotient of $P\times P$ by the diagonal action of $\G$;
\item  the source and target projection are induced by the first and second projections to $P$;
\item the multiplication is induced by the multiplication
\[
(P\times P)\tensor[_{\pr_1}]{\times}{_{\pr_2}}(P\times P)\to P\times P,\quad ((p_1, p_2), (p_2, p_3)) \mapsto (p_1, p_3)
\]
on the pair groupoid $P\times P\tto P$.
 \end{itemize}
 \begin{example}\label{ex:Gauge_groupoid}
In the particular case when $\G=G$ is a Lie group acting principally on $P\to M$, one discovers the classical {\bf gauge groupoid} $(P\times P)/G\tto M$. It is always transitive over $M$, its isotropy group at any point is isomorphic to $G$ and its $s$-fibre is isomorphic to $P$. Then Example~\ref{ex:ME_transitive} can be reformulated as follows: any transitive Lie groupoid is isomorphic to the gauge groupoid of its $s$-fibre acted from the right by the isotropy group (Proposition 5.14 in~\cite{MOERDIJK}). By making use of this fact, one can often rephrase constructions in terms of transitive Lie groupoids into more familiar constructions in terms of ordinary principal (group) bundles.
\end{example}

\begin{proposition}\label{prop_gauge_construction}
Let $\pi: P\to M$ be a principal $\G$-bundle along $\mu:P\to \BB$. Then $\mu:P \to \BB$ is a principal $\Gaug(P)$-bundle along $\pi:P\to M$ and the actions of $\Sigma$ and $\Gaug(P)$ commute. In other words, $P$ is a principal bibundle between $\Sigma\tto \BB$ and $\Gaug(P)\tto M$.
\end{proposition}
In our setting, the gauge construction is in fact equivalent to the notion of Morita equivalence:
\begin{theorem}\label{thm:gauge_eq_morita}
Two Lie groupoids $\Sigma_1\tto \BB_1$ and $\Sigma_2\tto \BB_2$ are Morita equivalent if and only if there exists a principal $\Sigma_1$-bundle $P$ such that $\Sigma_2\cong \Gaug(P)$.
\end{theorem}

The following proposition, which is well known (e.g.\ the proof of Theorem 4.2 in~\cite{XUMORITAEQ}), was one of our starting points for this paper.
\begin{proposition}\label{prop:princ_ME}
The categories of principal bundles of Morita equivalent groupoids are equivalent.
\end{proposition}

\begin{remark}[Differentiable stacks]\label{rm:stacks}
 Proposition \ref{prop:princ_ME} is classically used to show that two Lie groupoids $\G_1$ and $\G_2$ are Morita equivalent if and only if their associated {\it differentiable stacks} $B\G_1$ and $B\G_2$ are isomorphic (see e.g.\ \cite{STACKS}). This leads naturally to the rephrasement of the definition of a differentiable stack as a Lie groupoid up to Morita equivalence. 
\end{remark}

\subsection{Jet prolongations and the Cartan distribution}\label{app:jets}
Let $\pi:Y\to X$ be a surjective submersion between manifolds. Let $\sigma_1:U\subset X \to Y$ and $\sigma_2:V\subset X\to Y$ be smooth sections of $\pi$ defined on open subsets $U$, $V$ of $X$, and let $x\in U\cap V$. One says that $\sigma_1$ and $\sigma_2$ have {\bf contact of order $k$} at $x$ if their order $k$ Taylor polynomial at $x$ coincide. This defines an equivalence relation on the set of smooth sections defined locally around $x$, whose equivalence classes are denoted by
\[
j^k_x\sigma := \{\sigma':\ \sigma'\ \text{is a local section of } \pi,\ x\in {\rm dom}(\sigma),\ \sigma'\ \text{has contact of order } k\ \text{with } \sigma\ \text{at } x\}
\]
\begin{definition}\label{def:jet_prol}
Let $\pi:Y\to X$ be a surjective submersion between manifolds, and let $k\in \mathbb{N}$. The {\bf $k$-th jet prolongation (or jet bundle) of $\pi:Y\to X$}, denoted by $J^kY$, is the set
\[
J^kY := \{j^k_x\sigma:\ x\in X,\ \sigma\ \text{is a local section of } \pi\}
\]
\end{definition}
Observe that $J^0Y\cong Y$ canonically.
\begin{lemma}
For all $k\in \mathbb{N}$, there is a canonical manifold structure on $J^kY$ such that:
\begin{itemize}
\item the natural projections $J^kY\to J^hY$, $h\geq k$, are surjective submersions;
\item the natural projections $J^kY\to X$ are surjective submersions.
\end{itemize}
\end{lemma}
Intuitively, one can put coordinates on $J^kY$ by assigning to each contact class $j^k_x\sigma\in J^kY$ the coordinates of $x\in X$ and of $\sigma(x)\in Y$ together with the coefficients of the order $k$ Taylor polynomial at $x$ of $\sigma$. The chain rule for differentiation ensures that such a prescription defines an atlas.
\begin{definition}\label{def:hol_sec}
Let $\pi:Y\to X$ be a surjective submersion. A local section $\psi:U\subset X\to J^kY$ of the jet prolongation $J^kY\to X$ is called {\bf holonomic} if there exists a local section $\sigma:U\to Y$ such that $\psi = j^k\sigma$. Here $j^k\sigma:U\to Y$ denotes the {\bf order $k$ prolongation of $\sigma$}:
\[
j^k\sigma:U \to J^kY,\quad x\mapsto j^k_x\sigma
\]
\end{definition} 
\begin{remark} The reader unfamiliar with jet spaces might want to notice that not all sections are holonomic. In general, when $\sigma:U\to J^k(M,N)$ is a local section, then for all $x\in U$ one has $\sigma(x) = j^k_x f_x$, where the maps $f_x:U_x\to N$ are some smooth functions defined on neighbourhoods $U_x$ of $x$. The section $\sigma$ is holonomic if and only if the family of representatives $\{f_x\}_{x\in U}$ can be replaced by a single function $f:U\to N$.

Holonomic sections of $J^kY\to X$ are therefore in bijection with local sections of $Y\to X$.
\end{remark}

\begin{definition}\label{def:diff_eq}
A geometric {\bf order $k$ differential equation} is an embedded submanifold $R \subset J^kY$. A {\bf local solution of $R$} is a local section $\sigma:U\to Y$ of $Y\to X$ such that the prolongation $j^k\sigma:U\to J^kY$ takes values in $R$.

Alternatively, in view of Definition \ref{def:hol_sec}, a solution of $R \subset J^kY$ is a holonomic section of $J^kY$ taking values in $R$.
\end{definition}

Intuitively, the jet formalism turns (regular) systems of PDEs into submanifolds. Solutions can be identified with a distinguished class of maps into such submanifolds: it turns out that these maps can be detected by means of a canonical distribution on $J^kY$.
\begin{definition}\label{def:cart-dist}
The {\bf Cartan distribution} on the jet prolongation $J^kY$ is the regular distribution $\CC^k$ whose integral sections are precisely the holonomic sections.
\end{definition}
Such a distribution indeed exists and is unique; it can be characterised as follows. 

\begin{proposition}\label{prp:cart-dist}
Let $\pi:Y\to X$ be a surjective submersion and $\pr:J^kY\to J^{k-1}Y$, $\pi^k:J^kY\to X$ denote the canonical projections. The Cartan distribution $\CC^k$ is given by the kernel of the vector-valued 1-form 
\[
\omega^k\in \Omega^1(J^kY, VJ^{k-1}Y),\quad \omega^k_{j^k_x \sigma}:=d \pr - d(j^{k-1} \sigma) \circ ds
\]
where $VJ^{k-1}Y$ is the vertical bundle of $s: J^{k-1}Y\to X$.

Moreover, the kernel of $\pr$ equals the intersection of $\ker(\omega^k)$ with $VJ^{k-1}Y$.
\end{proposition}

See, for example,~\cites{SAUNDERS, Vinogradovetal99} for more details on jets and the geometric studies of PDEs.

\addcontentsline{toc}{section}{References}

\printbibliography

\end{document}